\documentclass[dsc,wide,cover,mac]{puc-rio_thesis}
\usepackage{float}
\usepackage[brazil]{babel}   

\author{Emília Carolina Santana Teixeira Alves}

\advisor{Nicolau Cor\c c\~ao Saldanha}

\advisorInst{Pontif\'{\i}cia Universidade Cat\'{o}lica do Rio de Janeiro}

\title{Topology of the space of locally convex curves on the $3$-sphere}
\titleBR{Topologia do espa\c co das curvas localmente convexas na $3$-esfera}
\CDD{00xx}

\department{Matem\'{a}tica}
\program{Matem\'{a}tica}
\programBR{Matem\'{a}tica}
\school{Centro T\'{e}cnico Cient\'{\i}fico}
\university{Pontif\'{\i}cia Universidade Cat\'{o}lica do Rio de Janeiro}
\uni{PUC--Rio}

\city{Rio de Janeiro}
\day{5} \month{April} \year{2016}

\jury
{
   \jurymember{Umberto Leone Hryniewicz}{Universidade Federal do Rio de Janeiro}
   
   \jurymember{Leonardo Navarro}{Universidade Federal Fluminense}
   
   \jurymember{Ricardo S\'{a} Earp}{Pontif\'{\i}cia Universidade Cat\'{o}lica do Rio de Janeiro}

   \jurymember{Paul Schweitzer}{Pontif\'{\i}cia Universidade Cat\'{o}lica do Rio de Janeiro}
   
   \jurymember{Carlos Tomei}{Pontif\'{\i}cia Universidade Cat\'{o}lica do Rio de Janeiro}
   
   \jurymemberDr{Pedro  Z\"{u}hlke}{Universidade de S\~ao Paulo}
   
}

\resume{

Gradute at Mathematics from Universidade Federal Fluminense (2004 - 2009). Master from Pontif\'{\i}cia Universidade Cat\'{o}lica do Rio de Janeiro (2010 - 2012) under the supervision of Nicolau Cor\c c\~ao Saldanha, at the time we worked on Theory Groups. Phd (2012 - 2016) at the same institution and with the same advisor, studying the homotopy type of locally convex curves on the $3$-sphere.

}

\acknowledgment
\dedication
{
	
}


\abstract
{

Given an integer $n \geq 2$ and $z \in \mathrm{Spin}_{n+1}=\widetilde{SO}_{n+1}$, let $\mathcal{L}\mathbb{S}^n(z)$ be the set of all locally convex curves $\gamma: [0,1] \rightarrow \mathbb{S}^n$ with fixed initial and final lifted Frenet frame $\tilde{\mathcal{F}}_\gamma(0)=\mathbf{1}$ and $\tilde{\mathcal{F}}_\gamma(1)=z$. Saldanha and Shapiro proved that there are just finitely many non-homeomorphic spaces among $\mathcal{L}\mathbb{S}^n(z)$ when $z$ varies in $\mathrm{Spin}_{n+1}$ (in particular, at most $3$ for $n=2$ and at most $5$ for $n=3$). For any $n \geq 2$, the topology of one of these spaces is well-known and coincides with the topology of a space of generic curves. But the topology of the other spaces is unknown in general. For $n=2$, Saldanha completely determined the topology of the other $2$ spaces, proving in particular that they are different from the space of generic curves.

The purpose of this thesis is to study the case $n=3$. We will obtain information on the topology of $2$ of these $4$ other spaces, allowing us to conclude that they are different from the space of generic curves. To do this, we will prove that any generic curve in $\mathbb{S}^3$ can be decomposed as a pair of generic curves in $\mathbb{S}^2$ (a generic curve in $\mathbb{S}^2$ is just an immersion); moreover, if the curve in $\mathbb{S}^3$ is locally convex, then one of the associated curve in $\mathbb{S}^2$ is also locally convex. Using this latter result, which is of independent interest, we will also be able to determine, in some cases, when a locally convex curve in $\mathbb{S}^3$ is globally convex by just looking at the corresponding locally convex curve in $\mathbb{S}^2$.

}

\keywordsBR
{
   \key{Geometria}
   \key{Topologia}
   \key{Curvas localmente convexas}
   \key{C\'{e}lulas de Bruhat}
}

\abstractBR
{

Dado um inteiro $n \geq 2$ e $z \in \mathrm{Spin}_{n+1}=\widetilde{SO}_{n+1}$, seja $\mathcal{L}\mathbb{S}^n(z)$ o conjunto de todas as curvas localmente convexas $\gamma: [0,1] \rightarrow \mathbb{S}^n$ com levantamento do Frenet frame inicial e final $\tilde{\mathcal{F}}_\gamma(0)=\mathbf{1}$ and $\tilde{\mathcal{F}}_\gamma(1)=z$. Saldanha and Shapiro provaram que s\'{o} tem um n\'{u}mero finito de espa\c cos $\mathcal{L}\mathbb{S}^n(z)$ quando $z$ varia em $\mathrm{Spin}_{n+1}$ (em particular, s\~{a}o no m\'{a}ximo $3$ quando $n=2$ e no m\'{a}ximo $5$ quando $n=3$). Para qualquer $n \geq 2$, a topologia de um desses espa\c cos \'{e} bem conhecida e coincide com a topologia do espa\c co das curvas gen\'{e}ricas. Mas a topologia dos outros espa\c cos \'{e} em geral desconhecida. Para $n=2$ , Saldanha determinou completamente a topologia dos outros 2 espa\c cos, provando em particular que estes s\~{a}o diferentes do espa\c co das curvas gen\'{e}ricas.

O objetivo desta tese \'{e} estudar o caso $n=3$. N\'{o}s iremos obter informa\c c\~{o}es sobre a topologia de 2 destes 4 espa\c cos, permitindo-nos concluir que estes s\~{a}o diferentes do espa\c co das curvas gen\'{e}ricas. Para isso, n\'{o}s iremos provar  que qualquer curva gen\'{e}rica em $\mathbb{S}^3$ pode ser decomposta como um par de curvas gen\'{e}ricas em $\mathbb{S}^2$ (uma curva gen\'{e}rica em $\mathbb{S}^2$ \'{e} apenas uma imers\~{a}o); mais ainda, se a curva em $\mathbb{S}^3$ \'{e} localmente convexa, ent\~{a}o uma das curvas associadas em $\mathbb{S}^2$ \'{e} tamb\'{e}m localmente convexa. Usando este \'{u}ltimo resultado, que \'{e} de interesse independente, n\'{o}s iremos tamb\'{e}m ser capazes de determinar, em alguns casos, quando uma curva localmente convexa em $\mathbb{S}^3$ \'{e} globalmente convexa apenas olhando a curva localmente convexa correspondente em $\mathbb{S}^2$.       

}

\tablesmode{fig}
\begin{document}
\def\N{\mathbb N}
\def\C{\mathbb C}
\def\Q{\mathbb Q}
\def\R{\mathbb R}
\def\T{\mathbb T}
\def\SS{\mathbb S}
\def\H{\mathbb H}
\def\A{\mathbb A}
\def\Z{\mathbb Z}
\def\1{\mathbf 1}
\def\k{\mathbf k}
\def\j{\mathbf j}
\def\i{\mathbf i}
\def\a{\mathbf a}

\newtheorem{Main}{Theorem}
\newtheorem{Coro}{Corollary}
\newtheorem{Conj}{Conjecture}
\newtheorem{Prop}{Proposition}

\renewcommand{\theMain}{\Alph{Main}}
\renewcommand{\theCoro}{\Alph{Coro}}
\renewcommand{\theProp}{\Alph{Prop}}


\chapter{Introduction}

\label{chapter1}

The purpose of this thesis is to study the topology of the space of locally convex (or non-degenerate) curves on the $n$-dimensional sphere, particularly for $n=3$. These are the curves $\gamma$ for which at all time, the derivatives $\gamma(t), \gamma'(t), \dots, \gamma^{(n)}(t)$ are linearly independent. To motivate the problem, let us start by recalling previous works on this question and related problems.

\section{History of the problem}

Locally convex curves are a particular class of immersed (or regular) curves, which are curves for which $\gamma'(t)$ is never zero. In \cite{Whi37}, Whitney completely classified immersion of closed curves in the plane $\R^2$, with the $C^r$-topology (for $r \geq 1$). There is an obvious invariant which is the rotation (or winding) number, and he proved that any two such immersions are homotopic (by a homotopy of immersion) if and only if they have the same rotation number. Later, in~\cite{Sma58}, Smale generalized this result of Whitney by studying immersions (up to homotopy of immersions) of $\SS^1$ into an arbitrary Riemannian manifold $M$: he proved that such curves are classified by the fundamental group of the unit tangent bundle of $M$. Smale went on to further classify immersions of the $2$-sphere $\SS^2$ (\cite{Sma59}) and then the $n$-sphere $\SS^n$ (\cite{Sma59b}), and with further work of Hirsch (\cite{Hir59}), one has a complete classification of immersions of a manifold $N$ into another manifold $M$ (where the dimension of $M$ is strictly bigger than the dimension of $N$).

In a different direction, in \cite{Lit70}, Little studied the space of closed curves $\gamma$ on $\SS^2$ with non-zero geodesic curvature, and with a fixed initial and final frame equal to the identity (in this context, this means that $\gamma(0)=\gamma(1)$ is the first vector of the canonical basis of $\R^3$ and $\gamma'(0)=\gamma'(1)$ is the second vector of the canonical basis of $\R^3$). These are a particular class of immersions, and we will see later (in Chapter~\ref{chapter2}) that locally convex curves in $\SS^2$ are the same as immersions with non-zero geodesic curvature. Little proved that this space has $6$ connected components (and if we fix an orientation, there are only $3$ connected components). In~\cite{Lit71}, he also investigated the corresponding problem in $\R^3$ and showed that there are $4$ connected components (and $2$ connected components if we fix an orientation). Let us also mention that Feldman also studied the topology of the space of immersion of $\SS^1$ on a Riemannian manifold of dimension at least $3$ (first for $\R^3$ in~\cite{Fel68} and then in~\cite{Fel71} in the general case) with non-zero geodesic curvature, but in higher dimension these curves are not necessarily locally convex.

The study of locally convex curves in higher dimensional spaces ($n$-sphere $\SS^n$ or $\R^{n+1}$, for $n \geq 2$) regain interest in the nineties as they are deeply related to the study of linear ordinary differential equations of order $n+1$ (see for instance Chapter~\ref{chapter3}, \S\ref{s34}).

First, in \cite{SS91} and \cite{Sha93}, M. Z. Shapiro and B. Z. Shapiro counted the number of connected components of the space of closed curves which are locally convex (in the Euclidean space, in the sphere or in the projective space). Fixing an orientation in the case of the Euclidean space or the sphere, this number turns out to be $2$ or $3$ depending on the parity of the dimension (in the projective case, this number is either $3$ or $10$, depending on the parity of the dimension). The case of not necessarily closed locally convex curves in the $2$-sphere $\SS^2$ was then studied by B. Z. Shapiro and Khesin in \cite{KS92} and \cite{KS99} (\cite{KS92} is only a partial announcement, the proof of which are contained in \cite{KS99}), allowing the initial and final frame to be distinct and related by a matrix $Q \in SO_3$ (called the monodromy); the case where $Q$ is the identity was the case studied by Little. Depending on the Jordan normal form of $Q$, they prove that the corresponding space of curves has $2$ or $3$ connected components. 

Even though the number of connected components of those spaces has been intensively studied, little information on the cohomology or higher homotopy groups was available, even on the $2$-sphere. It follows from~\cite{Sha93} that among the $3$ connected components of the space of locally convex curves in $\SS^2$ (in the orientation preserving case, with the initial and final frame being the identity), one connected component consists of simple curves and this component is contractible. The topology of these last two components remained mysterious until the works of Saldanha in \cite{Sal09I}, \cite{Sal09II} and \cite{Sal13}. After preliminary results in \cite{Sal09I} and \cite{Sal09II} where some information on higher homotopy and cohomology groups were obtained, a complete answer in the case $\SS^2$ was given in \cite{Sal13}; in particular the homotopy types of the remaining two components studied by Little are known. These results are in fact deduced from more general results of~\cite{Sal13}, in which the study the space of locally convex curves (not necessarily closed) whose initial frame is the identity and final frame an arbitrary matrix $Q \in SO_3$ (even more generally, one can consider an arbitrary final ``lifted" frame $z \in \mathrm{Spin}_3$, where $\mathrm{Spin}_3$ is the universal double cover of $SO_3$).

For an arbitrary $n \geq 2$, there are few results. In~\cite{SS12}, Saldanha and B. Z. Shapiro studied the following problem. Consider the spaces of locally convex curves in $\SS^n$ with an initial frame equal to the identity and a final frame equal to an arbitrary matrix $Q \in SO_{n+1}$ (or more generally, an arbitrary final "lifted" frame $z \in \mathrm{Spin}_{n+1}$). Theses spaces of locally convex curves are contained in a space of "generic" curves (a generalization, for any $n \geq 2$, of the space of immersions on $\SS^2$), and the topology of this space of generic curves is well-known. When $Q$ (or $z$) varies, Saldanha and B. Z. Shapiro proved that there are only a finite number of non-homeomorphic spaces of locally convex curves. They also proved that each of these spaces contains at least the topology of the space of generic curves (in the sense that the inclusion map induces a surjective homomorphism between homotopy and homology groups), and moreover that at least one them is always homotopy equivalent to the space of generic curves (see Chapter~\ref{chapter4}, \S\ref{s41} for a detailed exposition of their result). A first problem left open is to decide whether the other spaces of locally convex curves are homotopically equivalent to the space of generic curves or not. A second, more precise problem, is to understand if these finitely many spaces of locally convex curves are pairwise non-homeomorphic.

For the case $n=2$, there are only two such spaces of locally convex curves (when $Q$ varies in $SO_3$) and three such spaces of locally convex curves (when $z \in \mathrm{Spin}_3$). Since the topology of these spaces are known according to~\cite{Sal13}, the other spaces of locally convex curves are not homotopically equivalent to the space of generic curves, and moreover, all these spaces are in fact pairwise non-homeomorphic. 

However, for $n \geq 3$, nothing is known. It is the purpose of this thesis to study this problem for $n=3$, by studying in more detail the topology of the space of locally convex curves in $\SS^3$. For $n=3$, there are only $4$ spaces of locally convex curves for which one does not even know if they are homotopically equivalent to the space of generic curves. We will obtain information on the topology of $2$ of these $4$ spaces, proving in particular that they are not homotopically equivalent to the space of generic curves. Moreover, that they are pairwise non-homeomorphic.

\section{Plan of the thesis}

In Chapter~\ref{chapter2} we start with some algebraic preliminaries. In \S\ref{s21}, we recall some basic notions on the special orthogonal group and its universal cover the spin group. We also explain how $\mathrm{Spin}_3$ is isomorphic to the $3$-sphere $\SS^3$ and how $\mathrm{Spin}_4$ is isomorphic to the product $\SS^3 \times \SS^3$. In \S\ref{s22}, we recall some basic notions on signed permutation matrices which will be necessary in \S\ref{s23} to explain the Bruhat decomposition of the special orthogonal group and the lifted decomposition to the spin group. This decomposition was already an important tool in~\cite{SS12}, and it will be also very important for us.

In Chapter~\ref{chapter3} we present some basic notions on locally convex curves and generic curves. The section \S\ref{s31} is devoted to their definition and basic properties, while in \S\ref{s32}, we define and study the corresponding properties of globally convex curves, which are of fundamental importance in the study of locally convex curves. In \S\ref{s33} we properly define the Frenet frame curve associated to a locally convex curve, and more generally to a generic curve, so that we can define in \S\ref{s34} the spaces of curves we will be interested in. Still in \S\ref{s34} we will explain how the Bruhat decomposition already simplifies the study of our spaces of curves. In \S\ref{s35} we introduce another class of curves, the Jacobian curves and quasi-Jacobian, which are nothing but a different point of view on Frenet frame curves associated to locally convex curves and generic curves. This is part of the motivation to study the topology of the space of locally convex curves, and we will actually use this point of view in \S\ref{s36} to introduce a convenient topology for our spaces of curves.

In Chapter~\ref{chapter4}, we state the main results of this thesis. But prior to this, in \S\ref{s41}, \S\ref{s42} and \S\ref{s43} we state precisely some previous results from~\cite{Sal13} and~\cite{SS12}, since our results are a continuation of theirs (\S\ref{s41} deals with the case of an arbitrary $n \geq 2$, while \S\ref{s42} and \S\ref{s43} deal respectively with the special cases $n=2$ and $n=3$). Our main results are stated in \S\ref{s44}.

In Chapter~\ref{chapter5}, we recall some basic operations one can perform on the space of locally convex curves. In \S\ref{s51} we explain how to reverse time properly, in \S\ref{s52} we describe a certain duality in the space of locally convex curves and in \S\ref{s53}, we study the important operation of producing a new locally convex curve by cutting a little piece (at the beginning or at the end) of an old locally convex curve. 

In Chapter~\ref{chapter6}, we prove a result which will be crucial in the sequel, namely any generic curve in $\SS^3$ can be decomposed as a pair of curves in $\SS^2$. Moreover, if the curve in $\SS^3$ is locally convex, then one of the associated curve in $\SS^2$ is also locally convex. This will be the content of \S\ref{s61}. In \S\ref{s62}, we will give various examples, that will be used many times throughout this thesis, illustrating this decomposition result. This result, together with the examples we will give, will be used in \S\ref{s63} and \S\ref{s64} to characterize, in two different situation, when a curve in $\SS^3$ is globally convex by simply looking at its corresponding locally convex curve in $\SS^2$.   

In Chapter~\ref{chapter7} we will explain the link between convex arcs and the Bruhat decomposition in the special orthogonal group and the spin group. In \S\ref{s71} we give a characterization of convex arcs while in \S\ref{s72} we give the complete list of convex matrices and spins in the case $n=3$. In \S\ref{s73}, we will give a more convenient expression of the $4$ spaces of locally convex curves we are mainly in interested.   

Finally, Chapter~\ref{chapter8} is devoted to the proof of our main result, namely that $2$ of these $4 $ spaces are not homotopically equivalent to the space of generic curves. To do this, in \S\ref{s81} we will introduce a notion of ``adding a pair of loops" to a given curve. This notion will be detailed in the case $n=2$ in \S\ref{s82}, where we will recall important results from~\cite{Sal13}. In   \S\ref{s83} we will slightly modify this operation of adding a pair of loops for $n=3$, in order to transfer more easily the results from~\cite{Sal13} in the case $n=2$ in our case $n=3$. This transfer will be done in \S\ref{s83} and \S\ref{s84}, while  the main result will be proved in \S\ref{s85}.  



%
%

\chapter{Algebraic preliminaries}

\label{chapter2}

In this chapter we start with some algebraic preliminaries: first we recall some definitions and basic properties of the special orthogonal groups and the spin groups, and then we explain a decomposition of these groups (the Bruhat decomposition) into finitely many subsets which will play an important role in this thesis.

\section{Orthogonal groups and spin groups}\label{s21}

We denote by $M_{n+1}$ the space of square matrices of size $n+1$ with real coefficients, then we define the general linear group
\[  GL_{n+1}:=\{ A \in M_{n+1} \; | \; \mathrm{det}\;A \neq 0\}, \]
its connected component of the identity
\[  GL_{n+1}^+:=\{ A \in M_{n+1} \; | \; \mathrm{det}\;A > 0\} \]
and the special linear group 
\[  SL_{n+1}:=\{ A \in M_{n+1} \; | \; \mathrm{det}\;A=1\}. \] 
It is well-known that $GL_{n+1}$ is a Lie group of dimension $(n+1)^2$, and $GL_{n+1}^+$ and $SL_{n+1}$ are Lie subgroups (the first because it it is the identity component, the second because it is a closed subgroup), of dimension respectively $(n+1)^2$ and $(n+1)^2-1$. 

The Lie algebra of $GL_{n+1}$, and of $GL_{n+1}^+$ is nothing but the space of matrices $M_{n+1}$, while the Lie algebra of $SL_{n+1}$ is the subspace of matrices with zero trace. 

We define the orthogonal and special orthogonal group 
\[ O_{n+1}:=\{ A \in M_{n+1} \; | \; A^\top A=\mathrm{Id}\}, \quad SO_{n+1}:=\{ A \in O_{n+1} \; | \; \mathrm{det}\;A=1\} \]
where $A ^\top$ denotes the transpose of a matrix $A$. These are Lie subgroups of $GL_{n+1}$, and $SO_{n+1}$ is path-connected as it is the identity component of $O_{n+1}$. Their dimension is equal to $n(n+1)/2$. Moreover, the groups $O_{n+1}$ and $SO_{n+1}$ are compact, and their Lie algebra is the subspace of skew-symmetric matrices.

The group $SO_2$ is obviously homeomorphic to $\SS^1$, and it is not simply connected. For $n \geq 2$, the groups $SO_{n+1}$ are no longer simply connected; their fundamental groups $\pi_1(SO_{n+1})$ are isomorphic to $\Z_2$, the group with two elements. By definition, $n \geq 2$, the spin group $\mathrm{Spin}_{n+1}$ is the universal cover of $SO_{n+1}$, and it comes with a natural projection
\[ \Pi_{n+1} : \mathrm{Spin}_{n+1} \rightarrow SO_{n+1} \]
which is a double covering map. The group $\mathrm{Spin}_{n+1}$ is therefore a simply connected Lie group, which is also compact, and it has the same Lie algebra (and hence the same dimension) as $SO_{n+1}$. Throughout this thesis, the unit element in the group $\mathrm{Spin}_{n+1}$ will be denoted by $\1 \in \mathrm{Spin}_{n+1}$.  

For our purposes it will be sufficient to give description of $\mathrm{Spin}_{n+1}$ in the cases $n=2$ and $n=3$ where we will have $\mathrm{Spin}_3 \simeq \SS^3$ and $\mathrm{Spin}_4 \simeq \SS^3 \times \SS^3$.  

But first we need to recall the definition of the algebra of quaternions:
\[ \mathbb{H}:=\{a+b\i+c\j+d\k =a\1+b\i+c\j+d\k \; | \; (a,b,c,d) \in \R^4\} \]
where $\1=1$, and $\i,\j,\k$ satisfies the product rules
\[ \i^2=\j^2=\k^2=\i\j\k=-1. \]
Like complex numbers, a quaternion $q=a+b\i+c\j+d\k$ has a conjugate $\bar{q}=a-b\i-c\j-d\j$. As a real vector space, $ \mathbb{H}$ is isomorphic to $\R^4$, hence one can define a Euclidean norm on $\mathbb{H}$ and the set of quaternions with unit norm, $U(\mathbb{H})$, can be naturally identified with $\SS^3$. The space of imaginary quaternions
\[ \mathrm{Im}\mathbb{H}:=\{b\i+c\j+d\k \; | \; (b,c,d) \in \R^3\} \]
is naturally identified to $\R^3$, and given a unit quaternion $z \in U(\mathbb{H}) \simeq \SS^3$, it acts on $\mathrm{Im}\mathbb{H} \simeq \R^3$ in the following way:
\[ h \in \mathrm{Im}\mathbb{H} \simeq \R^3 \mapsto R_z(h):=zh\bar{z} \in  \mathrm{Im}\mathbb{H} \simeq \R^3. \]
The fact that $R_z(h)$ is indeed in $\mathrm{Im}\mathbb{H}$ follows from a simple computation. It is clear that $||R_z(h)||=||h||$, hence $R_z$ is an isometry. It is moreover a direct isometry: indeed, $R_1$ is nothing but the identity, and since $\SS^3$ is path-connected, $R_z$ is in the connected component of the identity. This means that $R_z \in SO_3$. We have thus found a map
\[ z \in \SS^3 \mapsto R_z \in SO_3 \]
which is a homomorphism, surjective and 2-to-1 ($R_{-z}=R_z$). So this is a covering map, and as $\SS^3$ is simply connected, this identifies $\SS^3$ with the spin group $\mathrm{Spin}_3$ and the above map is identified with the canonical projection $\Pi_3 : \mathrm{Spin}_3 \rightarrow SO_3$. In matrix notations, this map can be defined by
\begin{equation*}
\Pi_3(a+b\i+c\j+d\k)=\begin{pmatrix} 
  a^2+b^2-c^2-d^2    & -2ad+2bc & 2ac+2bd \\ 
   2ad+2bc  &  a^2-b^2+c^2-d^2 & -2ab+2cd \\
   -2ac+2bd & 2ab+2cd & a^2-b^2-c^2+d^2 
\end{pmatrix}.
\end{equation*}
Under the identification $\mathrm{Spin}_3 \simeq \SS^3$, the unit element $\1 \in \mathrm{Spin}_3$ is identified with the quaternion $\1=1$.

Using quaternions again, we can also identify $\mathrm{Spin}_4$ with the product $\SS^3 \times \SS^3$. Indeed, consider a pair of unit quaternions $(z_l,z_r) \in \SS^3 \times \SS^3$, it acts on $\mathbb{H} \simeq \R^4$ by
\[ q \in \mathbb{H} \simeq \R^4 \rightarrow R_{z_l,z_r}(q):=z_lq\bar{z_r}.  \]
As before, this defines a map 
\[ (z_l,z_r) \in \SS^3 \times \SS^3 \mapsto R_{z_l,z_r} \in SO_4 \]
which can be identified with the universal covering map $\Pi_4 : \mathrm{Spin}_4 \rightarrow SO_4$. One recovers the previous construction by embedding $\SS^3$ into the diagonal of $\SS^3 \times \SS^3$ and observing that $R_z=R_{z,z}$ leaves $\mathrm{Im}\mathbb{H}$ invariant. The canonical projection $\Pi_4 : \mathrm{Spin}_4 \rightarrow SO_4$ can also be given in matrix notations, even though here this is much more cumbersome:
\begin{equation*}
\Pi_4(a_l+b_l\i+c_l\j+d_l\k,a_r+b_r\i+c_r\j+d_r\k)=
\begin{pmatrix}
C_1 & C_2 & C_3 & C_4
\end{pmatrix}
\end{equation*}
where the columns $C_i$, for $1 \leq i \leq 4$, are given by
\begin{equation*}
C_1=
\begin{pmatrix} 
a_la_r+b_lb_r+c_lc_r+d_ld_r  \\ 
-a_lb_r+b_la_r-c_ld_r+d_lc_r  \\
-a_lc_r+b_ld_r+c_la_r-d_lb_r \\
-a_ld_r-b_lc_r+c_lb_r+d_la_r 
\end{pmatrix}
\quad C_2=
\begin{pmatrix} 
a_lb_r-b_la_r-c_ld_r+d_lc_r  \\ 
a_la_r+b_lb_r-c_lc_r-d_ld_r  \\
a_ld_r+b_lc_r+c_lb_r+d_la_r  \\
-a_lc_r+b_ld_r-c_la_r+d_lb_r  
\end{pmatrix}
\end{equation*}
\begin{equation*}
C_3=
\begin{pmatrix} 
a_lc_r+b_ld_r-c_la_r-d_lb_r  \\ 
-a_ld_r+b_lc_r+c_lb_r-d_la_r  \\
a_la_r-b_lb_r+c_lc_r-d_ld_r   \\
a_lb_r+b_la_r+c_ld_r+d_lc_r  
\end{pmatrix}
\quad C_4=
\begin{pmatrix} 
a_ld_r-b_lc_r+c_lb_r-d_la_r \\ 
a_lc_r+b_ld_r+c_la_r+d_lb_r \\
-a_lb_r-b_la_r+c_ld_r+d_lc_r \\
a_la_r-b_lb_r-c_lc_r+d_ld_r
\end{pmatrix}.
\end{equation*}
As before, under the identification $\mathrm{Spin}_4 \simeq \SS^3 \times \SS^3$, the unit element $\1 \in \mathrm{Spin}_4$ is identified with the pair of quaternions $(\1,\1)=(1,1)$.

\section{Signed permutation matrices}\label{s22}

Let $S_{n+1}$ be the group of permutations on the set of $n+1$ elements $\{1, \dots, n+1\}$. An inversion of a permutation $\pi \in S_{n+1}$ is a pair $(i,j) \in \{1, \dots, n+1\}^2$ such that $i<j$ and $\pi(i)>\pi(j)$.
The number of inversions of a permutation $\pi \in S_{n+1}$, that we denote by $\#\mathrm{inv}(\pi)$, is by definition the cardinal of the set of inversions. It is well-known that the number of inversion is at most $n(n+1)/2$, and that this number of inversions is only reached by the permutation $\rho \in S_{n+1}$ defined by $\rho(i)=n+2-i$ for all $i \in \{1, \dots, n+1\}$. In other words, $\rho$ is the product of transpositions
\[ \rho=(1\;n+1)(2\; n) ... \in S_{n+1}. \]

Now a matrix $P \in M_{n+1}$ is a permutation matrix if each column and each row of $P$ contains exactly one entries equal to $1$, and the others entries are zero. Permutation matrices forms a finite sub-group of $O_{n+1}$. There is an obvious isomorphism between the group of permutation matrices and $S_{n+1}$: to a permutation $\pi \in S_{n+1}$ we can associated a permutation matrix $P_\pi=(p_{i,j})$ where

\[ P_\pi(e_i)=e_{\pi(i)}, \] where $e_i$ denotes the $i$-th vector of the canonical basis of $\R^{n+1}$.

In particular, both groups have cardinal $(n+1)!$. Given a permutation matrix $P$, it can be written $P=P_\pi$ for a unique $\pi \in S_{n+1}$ and we define the number of inversions $\#\mathrm{inv}(P)$ to be the number of inversions of the associated permutation $\#\mathrm{inv}(\pi)$.  

More generally, a signed permutation matrix is a matrix for which each column and each row contains exactly one entries equal to $1$ or $-1$, and the others entries are zero. The set of signed permutation matrices will be denoted by $B_{n+1}$: it is easy to see that this is a finite sub-group of $O_{n+1}$, and its cardinal is equal to $2^{n+1}(n+1)!$. Given a signed permutation matrix $P$, let $abs(P)$ be the associated permutation matrix obtained by dropping the signs (put differently, the entries of $abs(P)$ are the absolute values of the entries of $P$). This defines a map $P \mapsto \pi$, where $abs(P)=P_\pi$, which is a homomorphism from $B_{n+1}$ to $S_{n+1}$, and we can define the number of inversions $\#\mathrm{inv}(P)$ as before. 

The group of signed permutation matrices of determinant one is $B_{n+1}^+=B_{n+1} \cap SO_{n+1}$, and it has a cardinal equal to $2^{n}(n+1)!$. 

\section{Bruhat decomposition}\label{s23}

Let us denote by $Up^+_{n+1}$ the group of upper triangular matrices with positive diagonal entries.

\begin{definition}
Given $Q \in SO_{n+1}$, we define the \emph{Bruhat cell} $\mathrm{Bru}_Q$ as the set of matrices $UQU' \in SO_{n+1}$, where $U$ and $U'$ belong to $Up^+_{n+1}$. 
\end{definition}

In other words, two matrices $Q \in SO_{n+1}$ and $Q' \in SO_{n+1}$ belong to the same Bruhat cell if and only if there exist $U$ and $U'$ in $Up^+_{n+1}$ such that $Q'=UQU'$. It is easy to see that given $Q \in SO_{n+1}$ and $Q' \in SO_{n+1}$, $\mathrm{Bru}_Q$ and $\mathrm{Bru}_{Q'}$ are either equal or disjoint. In fact, Bruhat cells can be considered as the orbits of a group action. First observe that if $Q' \in \mathrm{Bru}_Q$, then we can write $Q'=UQU'$ for some $U$ and $U'$ in $Up^+_{n+1}$ which are in general not unique. But if we further require that $U$ belongs to $Up^1_{n+1}$, the group of upper triangular matrices with diagonal entries equal to one, then $U'$ is uniquely defined.

\begin{definition}
The map $B : Up^1_{n+1} \times SO_{n+1} \mapsto SO_{n+1}$ defined by
\[ B(U,Q)=UQU' \]
where $U'$ is the unique matrix in $Up^+_{n+1}$ such that $UQU' \in SO_{n+1}$ is called the \emph{Bruhat action}. 
\end{definition}

The Bruhat action is clearly a group action of $Up^1_{n+1}$ on $SO_{n+1}$, and its orbits coincide with the Bruhat cells. 

It is well-known that each Bruhat cell contains a unique signed permutation matrix $P \in B_{n+1}^+$, hence two Bruhat cells associated to two different signed permutation matrix are disjoint. We summarize this in the following theorem.

\begin{theorem}[Bruhat decomposition for $SO_{n+1}$]\label{Bruhat1}
We have the decomposition
\[ SO_{n+1}=\bigsqcup_{P \in B_{n+1}^+}\mathrm{Bru}_P. \]
\end{theorem}   

Therefore there are $2^n(n+1)!$ different Bruhat cells. Each Bruhat cell $\mathrm{Bru}_P$ is diffeomorphic to $\R^{\#\mathrm{inv}(P)}$, hence they are open if and only if they have maximal dimension, that is if they correspond to the permutation $\rho$ we previously defined by $\rho(i)=n+2-i$ for all $i \in \{1, \dots, n+1\}$. Taking into account signs, there are $2^n$ open Bruhat cells which are
\[ \mathrm{Bru}_{DP_\rho}, \quad D \in \mathrm{Diag}_{n+1}^+ \]
where $P_\rho$ is the permutation matrix associated to $\rho$, and $\mathrm{Diag}_{n+1}^+ \subset B_{n+1}^+$ is the subgroup consisting of diagonal matrices with entries $\pm 1$ and determinant $1$. Observe that $\mathrm{Diag}_{n+1}^+$ is isomorphic to $(\Z/2\Z)^n$ and has a cardinal equal to $2^n$. The union of these top-dimensional cells 
\[ \bigsqcup_{D \in \mathrm{Diag}_{n+1}^+ }\mathrm{Bru}_{DP_\rho} \subset SO_{n+1} \]
is therefore a dense subset of $SO_{n+1}$. 

Given a matrix $Q \in SO_{n+1}$, it will be useful to know how to find in practice the unique signed permutation matrix $P \in B_{n+1}^+$ such that $Q \in \mathrm{Bru}_P$. Let us briefly recall, following~\cite{SS12}, an algorithm that produces this signed permutation matrix. Start with the first column of $Q$, and look for the non-zero lowest entry, for instance assume it is $Q_{i,1}$. Up to multiplying $Q$ by a diagonal matrix $D \in Up_{n+1}^+$ we may assume that $Q_{i,1}=\pm 1$. Then we can perform row operations on $Q$ to clean the first column above row $i$; in this way we find $U_1 \in Up_{n+1}^+$ such that $Q_1=U_1Q$ satisfies $Q_{1}e_1=\pm e_i$. Then perform column operations on $Q_1$ to clean row $i$ to the right of the first column, that is obtain $U_2 \in Up_{n+1}^+$ such that $Q_2=Q_1U_2$ satisfies $e_i^\top Q_2=\pm {} e_1^\top$. Repeating the process for each column we eventually find a signed permutation matrix $P$ such that $P=U_1QU_2$ for some $U_1,U_2 \in Up_{n+1}^+$, that is $P \in B_{n+1}^+$ such that $Q \in \mathrm{Bru}_P$.   

\medskip

The Bruhat decomposition of $SO_{n+1}$ can be lifted to the universal double cover $\Pi_{n+1} : \mathrm{Spin}_{n+1} \rightarrow SO_{n+1}$. Let us define the following sub-group of $\mathrm{Spin}_{n+1}$: 
\[ \tilde{B}_{n+1}^+:=\Pi_{n+1}^{-1}(B_{n+1}^+). \]
The cardinal of $\tilde{B}_{n+1}^+$ is twice the cardinal of $B_{n+1}^+$, that is $2^{n+1}(n+1)!$. 

\begin{definition}
Given $z \in \mathrm{Spin_{n+1}}$ we define the \emph{Bruhat cell} $\mathrm{Bru}_z$ as the connected component of $\Pi_{n+1}^{-1}(\mathrm{Bru}_{\Pi_{n+1}(z)})$ which contains $z$. 
\end{definition}

It is clear,  from the definition of $\Pi_{n+1}$, that $\Pi_{n+1}^{-1}(\mathrm{Bru}_{\Pi_{n+1}(z)})$ is the disjoint union of $\mathrm{Bru}_z$ and $\mathrm{Bru}_{-z}$, where each set $\mathrm{Bru}_{z}$, $\mathrm{Bru}_{-z}$ is contractible and non empty.

Since the group $Up_{n+1}^1$ is contractible, its Bruhat action on $SO_{n+1}$ lifts to a Bruhat action on $\mathrm{Spin}_{n+1}$ that, for simplicity, we still denote by $B : Up_{n+1}^1 \times \mathrm{Spin}_{n+1} \rightarrow \mathrm{Spin}_{n+1}$. As before, the Bruhat cells on $\mathrm{Spin}_{n+1}$ are the orbits of the Bruhat action.

From Theorem~\ref{Bruhat1} we have the following result.

\begin{theorem}[Bruhat decomposition for $\mathrm{Spin}_{n+1}$]\label{Bruhat2}
We have the decomposition
\[ \mathrm{Spin}_{n+1}=\bigsqcup_{\tilde{P} \in \tilde{B}_{n+1}^+}\mathrm{Bru}_{\tilde{P}}. \]
\end{theorem}   

In $\mathrm{Spin}_{n+1}$, there are $2^{n+1}(n+1)!$ disjoint Bruhat cells. Each lifted Bruhat cell $\mathrm{Bru}_{\tilde{P}}$ is still diffeomorphic to $\R^{\#\mathrm{inv}(P)}$, where $P=\Pi_{n+1}(\tilde{P}) \in B_{n+1}^+$. Let us write
\[ \widetilde{\mathrm{Diag}}_{n+1}^+:=\Pi_{n+1}^{-1}(\mathrm{Diag}_{n+1}^+) \]
and $\pm \tilde{P}_\rho$ the two preimages of $P_\rho$ under the map $\Pi_{n+1}$. There are $2^{n+1}$ open Bruhat cells, which are given by
\[ \bigsqcup_{\tilde{D} \in \widetilde{\mathrm{Diag}}_{n+1}^+ }\mathrm{Bru}_{\tilde{D}\tilde{P}_\rho} \subset \mathrm{Spin}_{n+1} \]
and their union is dense in $\mathrm{Spin}_{n+1}$.

Let us conclude with the following definition.

\begin{definition}\label{bruhateq}
Two matrices $Q \in SO_{n+1}$ and $Q' \in SO_{n+1}$ (respectively two spins $z \in \mathrm{Spin}_{n+1}$ and $z' \in \mathrm{Spin}_{n+1}$) are said to be \emph{Bruhat equivalent} if they belong to the same Bruhat cell. 
\end{definition}


%
%

\chapter{Space of curves}

\label{chapter3}

Given an integer $n \geq 1$, we let
\[ \SS^n:=\{x\in \R^{n+1} \; | \, ||x||^2=1 \} \]
be the $n$-dimensional sphere, where $||\,.\,||$ denotes the Euclidean norm in $\R^{n+1}$. 

From now on, unless stated otherwise, we will denote by $M$ either $\SS^n$ or $\R^{n+1}$. A curve in $M$ is simply the image of a map (called parametrization) $\gamma : [0,1] \rightarrow M$; the curve is of class $C^k$, $k \in \N$ if the map $\gamma : [0,1] \rightarrow M$ is of class $C^k$. A reparametrization is a smooth diffeomorphism $\phi : [a,b] \rightarrow [0,1]$, where $[a,b]$ is a non-trivial segment; the image of the curve $\gamma \circ \phi : [a,b] \rightarrow M$ is the same as the image of the curve $\gamma : [0,1] \rightarrow M$. The derivative $\phi'(t)$ is always non-zero hence it has a constant sign for any $t \in (a,b)$: the reparametrization is said to be positive (or orientiation preserving) if the sign of $\phi'(t)$ is positive.

In the sequel, we will always identify a curve to one of its parametrizations $\gamma : [0,1] \rightarrow M$, since all the properties we will be interested in will be independent of the choice of a parametrization.

\section{Locally convex curves and generic curves}\label{s31}

Given a smooth curve $\gamma : [0,1] \rightarrow M$ and an integer $j \geq 0$, we denote by $\gamma^{(j)}(t)$ the $j^{th}$-derivative of $\gamma$ at $t \in (0,1)$; we have $\gamma^{(0)}(t)=\gamma(t)$ and we will more simply write $\gamma^{'}(t)=\gamma^{(1)}(t)$ and $\gamma^{''}(t)=\gamma^{(2)}(t)$.

Let us start with a very simple definition.

\begin{definition}
A smooth curve $\gamma:[0,1] \rightarrow M$ is an \emph{immersion} if $\gamma'(t)\neq 0$ for all $t \in (0,1)$. 
\end{definition}

Immersions are also called regular curves. This is clearly independent of the choice of a parametrization: if $\phi : [a,b] \rightarrow [0,1]$ is a reparametrization, then $(\gamma \circ \phi)'(t)=\phi'(t)\gamma'(\phi(t)) \neq 0$ for all $t \in (a,b)$. For such a curve, it is well-known that one can choose a parametrization by arc-length with the property that $||\gamma'(t)||=1$ for all $t \in (0,1)$. 

We can now introduce the main object of our study.

\begin{definition}\label{deflc}
A smooth curve $\gamma:[0,1] \rightarrow M$ is called \emph{locally convex}, or \emph{$n$-order free},  if 
\[ \mathrm{det}(\gamma(t),\gamma'(t),\gamma''(t),\dots,\gamma^{(n)}(t)) \neq 0\] 
for all $t \in (0,1)$. In particular, $\gamma(t) \neq 0 $ for all $t \in (0,1)$.
\end{definition}

Obviously, the sign of the determinant in the above definition is constant, and a locally convex curve is said to be positive (respectively negative) if this sign is positive (respectively negative). Without loss of generality, from now on we will always assume it to be positive. Let us give examples of locally convex curves on $\R^{n+1}$ and $\SS^n$.

\begin{example}\label{ex1}
Consider the curve $c :[0,1] \rightarrow \R^{n+1}$ defined by
\[ c(t)=(a_0,a_1t,a_2t^2, \dots, a_nt^n), \quad a_0, \dots, a_n>0. \]
An easy calculation shows 
\[ \mathrm{det}(c(t),c'(t),\dots,c^{(n)}(t))=\prod_{i=0}^ni!a_i\]
so that $c$ is locally convex for any $a_0, \dots, a_{n}>0$.
\end{example}

\begin{example}\cite{SS12}\label{ex2}
Consider the curve $\xi :[0,1] \rightarrow \SS^{n}$ defined as follows. For $n+1=2k$, take positive numbers $c_1,\dots, c_k$ such that $c_1^2+\cdots c_k^2=1$ and $a_1, \dots, a_k>0$ mutually distinct, and set
\[ \xi(t)=(c_1\cos(a_1t),c_1\sin(a_1t),\dots,c_k\cos(a_kt),c_k\sin(a_kt)).  \]
For $n+1=2k+1$, take positive numbers $c_0,c_1,\dots, c_k$ such that $c_0^2+c_1^2+\cdots c_k^2=1$ and set
\[ \xi(t)=(c_0,c_1\cos(a_1t),c_1\sin(a_1t),\dots,c_k\cos(a_kt),c_k\sin(a_kt)). \]
In both cases, the fact that the curve $\xi$ is locally convex follows from a simple computation. In the first case, we find
\[ \mathrm{det}(\xi(t),\xi'(t),\dots,\xi^{(n)}(t))=\left(\prod_{i=1}^kc_i^2a_i\right)\left( \prod_{1 \leq i<j \leq k}(a_i-a_j)^2(a_i+a_j)^2 \right) \]
and in the second case
\[ \mathrm{det}(\xi(t),\xi'(t),\dots,\xi^{(n)}(t))=c_0\left(\prod_{i=1}^kc_i^2a_i^3\right)\left( \prod_{1 \leq i<j \leq k}(a_i-a_j)^2(a_i+a_j)^2 \right). \]
In the case $n=3$, a locally convex curve looks like an ancient phone wire (see the Figure~\ref{fig:a} below).
\end{example}

\begin{figure}[h]
\centering
\includegraphics[scale=0.5]{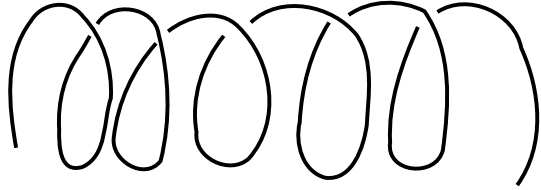}
\caption{An ancient phone wire is locally convex in $\SS^3$.}
\label{fig:a}
\end{figure}

In the sequel, we will need several elementary properties of locally convex curves in $M$. They are contained in the proposition below.

\begin{proposition}\label{proplocconv}
Let $\gamma : [0,1] \rightarrow M$ be a smooth curve. 
\begin{itemize}
\item[(i)] If $\phi : [a,b] \rightarrow [0,1]$ is a positive reparametrization, then $\gamma$ is locally convex if and only if $\gamma \circ \phi$ is locally convex. \\
\item[(ii)] If $A \in GL_{n+1}^+$, then $\gamma$ is locally convex if and only if the curve
\[ A\gamma : t \in [0,1] \mapsto A\gamma(t) \in \R^{n+1} \]
is locally convex. Moreover, if $A \in SO_{n+1}$, then $\gamma : [0,1] \rightarrow \SS^n$ is locally convex if and only $A\gamma : [0,1] \rightarrow \SS^n$ is locally convex.\\
\item[(iii)] If $g : [0,1] \rightarrow \R$ is a smooth positive function, that is $g(t)>0$ for all $t \in [0,1]$, then $\gamma$ is locally convex if and only if the curve
\[ g\gamma : t \in [0,1] \mapsto g(t)\gamma(t) \in \R^{n+1}\]
is locally convex.
\end{itemize}
\end{proposition}

Observe that the point $(i)$ of the above proposition ensures that being locally convex is independent of the choice of a parametrization, so that this is a well-defined property of the curve.

\begin{proof}
For $(i)$, a computation using the chain rule show that \[ (\gamma \circ \phi(t),(\gamma \circ \phi)'(t),\dots,(\gamma \circ \phi)^{(n)}(t))=(\gamma(\tau),\gamma'(\tau),\dots,\gamma^{(n)}(\tau))U \]
where $\tau=\phi(t)$, and where $U$ is a upper triangular matrix whose diagonal is given by $(1,\phi'(t), \dots,(\phi'(t)^n)$. This proves that $\gamma$ is locally convex if and only if $\gamma \circ \phi$ is locally convex.

For $(ii)$, the equality of matrices
\[ (A\gamma(t),(A\gamma)'(t),\dots,(A\gamma)^{(n)}(t))=A(\gamma(t),\gamma'(t),\dots,\gamma^{(n)}(t)) \]
implies that
\[\mathrm{det}(A\gamma(t),(A\gamma)'(t),\dots,(A\gamma)^{(n)}(t))=(\mathrm{det}A)\mathrm{det}(\gamma(t),\gamma'(t),\dots,\gamma^{(n)}(t))\]
and this proves $\gamma$ is locally convex if and only if $A\gamma$ is locally convex.

Finally, for $(iii)$, by using column operations one can check that 
\[ \mathrm{det}(g\gamma(t),(g\gamma)'(t),\dots,(g\gamma)^{(n)}(t))=g(t)^{n+1}\mathrm{det}(\gamma(t),\gamma'(t),\dots,\gamma^{(n)}(t)) \]
for all $t \in (0,1)$. This shows that $\gamma$ is locally convex if and only if $g\gamma$ is locally convex. 
\end{proof}

\medskip

We have the following straightforward corollary of Proposition~\ref{proplocconv}

\begin{corollary}\label{corlocconv}
Let $\gamma : [0,1] \rightarrow \R^{n+1}$ be a smooth curve.
\begin{itemize}
\item[(i)] If $\gamma(t)\neq 0$ for all $t \in [0,1]$, then $\gamma$ is locally convex if and only if the curve
\[ t \in [0,1] \mapsto \frac{\gamma(t)}{||\gamma(t)||} \in \SS^n \]
is locally convex. \\
\item[(ii)] If $\gamma_1(t)>0 $ for all $t \in [0,1]$, then $\gamma$ is locally convex if and only if the curve
\[ t \in [0,1] \mapsto \left(1,\frac{\gamma_2(t)}{\gamma_1(t)}, \dots,\frac{\gamma_n(t)}{\gamma_1(t)} \right) \in \R^{n+1}  \]
is locally convex.
\end{itemize}
\end{corollary}

\begin{proof}
Apply the point $(iii)$ of Proposition~\ref{proplocconv},  with $g(t)=||\gamma(t)||^{-1}$ to prove $(i)$ and with $g(t)=\gamma_1(t)^{-1}$ to prove $(ii)$.   
\end{proof}

\medskip

Even though we will be mainly interested in locally convex curves, it will be useful in the sequel to consider a larger class of curves.

\begin{definition}\label{defgeneric}
A smooth curve $\gamma:[0,1] \rightarrow M$ is called \emph{generic}, or \emph{$(n-1)$-order free}, if the vectors $\gamma(t), \gamma'(t),\gamma''(t),\dots,\gamma^{(n-1)}(t)$ are linearly independent for all $t \in (0,1)$. In particular, $\gamma(t) \neq 0 $ for all $t \in (0,1)$. 
\end{definition}

It is clear that a locally convex curve is generic in the above sense. It is also clear that Proposition~\ref{proplocconv} and Corollary~\ref{corlocconv} are valid if one replaces all instances of "locally convex" by "generic". 

In \S\ref{s33}, we will see a simple geometric characterization of generic curves and locally convex curves in the special cases $M=\SS^2$ and $M=\SS^3$. 

In the case $M=\SS^2$, since $\gamma'(t)$ is always orthogonal to $\gamma(t)$ (as one can see by differentiating the equality $\gamma(t).\gamma(t)=1$), $\gamma(t)$ and $\gamma'(t)$ are linearly independent if and only if $\gamma'(t) \neq 0$, that, is generic curves on $\SS^2$ are just immersions. For immersions on $\SS^2$ we will define the notion of geodesic curvature, and we will see that locally convex curves on $\SS^2$ are immersions with positive geodesic curvature.

In the case $M=\SS^3$, generic curves no longer coincide with immersions, but can be seen as immersions with non-zero geodesic curvature. For such a curve, one can further define the notion of geodesic torsion and we will see that locally convex curves are generic curves with positive geodesic torsion.

\section{Convex curves}\label{s32}

Next, let us introduce a special class of locally convex curves, which will be fundamental in the study of the space of all locally convex curves. 

\begin{definition}\label{defc}
A smooth curve $\gamma : [0,1] \rightarrow M$ is called \emph{globally convex} if any hyperplane $H \subseteq \mathbb{R}^{n+1}$ intersects the image of $\gamma$ in at most $n$ points, counting with multiplicity. 
\end{definition}

We need to clarify the notion of multiplicity in this definition. First, endpoints of the curve are not counted as intersections. Then, if $\gamma(t) \in H$ for some $t \in (0,1)$, the multiplicity of the intersection point $\gamma(t)$ is the smallest integer $k \geq 1$ such that
\[ \gamma^{(j)}(t) \in H, \quad 0 \leq j \leq k-1. \] 
So the multiplicity is one if $\gamma(t) \in H$ but $\gamma'(t) \notin H$, it is two if $\gamma(t) \in H$, $\gamma'(t) \in H$ but $\gamma''(t) \notin H$, and so on.

The following proposition is then obvious.

\begin{proposition}
Globally convex curves are locally convex.
\end{proposition}

\begin{proof}
Assume $\gamma$ is not locally convex. Then at some $t \in (0,1)$, the vectors $\gamma(t),\gamma'(t),\gamma''(t),\dots,\gamma^{(n)}(t)$ are linearly dependent, hence they are contained in some hyperplane of $\R^{n+1}$. But then the point $\gamma(t)$ intersects this hyperplane with a multiplicity at least $n+1$, so the curve cannot be globally convex. 
\end{proof}
 
\medskip 

Equivalently, the curve $\gamma=(\gamma_1,\dots,\gamma_{n+1})$ is globally convex if and only if given any $(a_1,\dots,a_{n+1}) \in \R^{n+1}\setminus\{0\}$, the function
\[ t \in [0,1] \mapsto a_1\gamma_1(t)+\cdots a_{n+1}\gamma_{n+1}(t) \in \R \]
has at most $n$ zeroes, counting multiplicities. For globally convex curves, we have the following analogues of Proposition~\ref{proplocconv} and Corollary~\ref{corlocconv}.

\begin{proposition}\label{propconv}
Let $\gamma : [0,1] \rightarrow M$ be a smooth curve. 
\begin{itemize}
\item[(i)] If $\phi : [a,b] \rightarrow [0,1]$ is a reparametrization, then $\gamma$ is globally convex if and only if $\gamma \circ \phi$ is globally convex. \\
\item[(ii)] If $A \in GL_{n+1}^+$, then $\gamma$ is globally convex if and only if the curve
\[ A\gamma : t \in [0,1] \mapsto A\gamma(t) \in \R^{n+1} \]
is globally convex. Moreover, if $A \in SO_{n+1}$, then $\gamma : [0,1] \rightarrow \SS^n$ is globally convex if and only if $A\gamma : [0,1] \rightarrow \SS^n$ is globally convex.\\
\item[(iii)] If $g : [0,1] \rightarrow \R$ is a smooth non-zero function, that is $g(t)\neq 0$ for all $t \in [0,1]$, then $\gamma$ is globally convex if and only if the curve
\[ g\gamma : t \in [0,1] \mapsto g(t)\gamma(t) \in \R^{n+1}\]
is globally convex.
\end{itemize}
\end{proposition}

\begin{corollary}\label{corconv}
Let $\gamma : [0,1] \rightarrow \R^{n+1}$ be a smooth curve.
\begin{itemize}
\item[(i)] If $\gamma(t)\neq 0$ for all $t \in [0,1]$, then $\gamma$ is globally convex if and only if the curve
\[ t \in [0,1] \mapsto \frac{\gamma(t)}{||\gamma(t)||} \in \SS^n \]
is globally convex. \\
\item[(ii)] If $\gamma_1(t)>0 $ for all $t \in [0,1]$, then $\gamma$ is globally convex if and only if the curve
\[ t \in [0,1] \mapsto \left(1,\frac{\gamma_2(t)}{\gamma_1(t)}, \dots,\frac{\gamma_n(t)}{\gamma_1(t)} \right) \in \R^{n+1}  \]
is globally convex.
\end{itemize}
\end{corollary}

The proofs of Proposition~\ref{propconv} and Corollary~\ref{corconv} are analogous to the proofs of Proposition~\ref{proplocconv} and Corollary~\ref{corlocconv}.

For simplicity, from now on globally convex curves will be simply referred to as convex curves. 

Here are some examples of convex curves, related to the examples~\ref{ex1} and~\ref{ex2} of locally convex curves.

\begin{example}
Let $b_0$, $b_1$ $\in \R$. The curve $c :[b_0,b_1] \rightarrow \R^{n+1}$ defined by
\[ c(t)=(a_0,a_1t,a_2t^2, \dots, a_nt^n), \quad a_0, \dots, a_n>0 \]
is convex. Indeed, it is given by a polynomial of degree $n$ which has then at most $n$ roots counted with multiplicities: it follows that any hyperplane of $\R^{n+1}$ intersects the image of $c$ in at most $n$ points, counting with multiplicity. 
\end{example}

\begin{example}\label{exclosed}
For $n=2k$, the curve $\xi :[-\pi,\pi] \rightarrow \SS^{n}$ defined by
\[ \xi(t)=(c_0,c_1\cos t,c_1 \sin t, \dots, c_k\cos (kt),c_k \sin (kt)), \]
where $c_0,c_1, \cdots, c_k$ are positive numbers and $c_0^2+c_1^2+\cdots + c_k^2=1$, is a closed convex curve. To see this, consider the positive reparametrization
\[ t \in (-\pi,\pi) \mapsto x=\tan(t/2) \in \R.  \] 
With this change of variables, the curve $\xi$ is given by a polynomial of degree at most $n=2k$, so exactly like for the curve $c$ above, it is convex.

For $n=2k-1$, the curve $\xi :[-\pi/2,\pi/2] \rightarrow \SS^{n}$ defined by
\[ \xi(t)=(c_1\cos t,c_1 \sin t,c_2\cos (3t),c_2 \sin (3t), \dots, c_k\cos ((2k-1)t),c_k \sin ((2k-1)t)), \]
where $c_1, \cdots, c_k$ are positive numbers $c_1^2+\cdots + c_k^2=1$, is a convex curve, but it is not closed. To prove that this is convex, consider the positive reparametrization
\[ t \in (-\pi/2,\pi/2) \mapsto x=\tan(t) \in \R\]
and proceeds as before. We will explain later why there cannot exist closed globally convex curves on odd-dimensional spheres.
\end{example}

\section{Frenet frame curves}\label{s33}

Now to a locally convex curve $\gamma:[0,1] \rightarrow M$, we associate a smooth curve 
\[ \mathcal{F_{\gamma}}:[0,1] \rightarrow SO_{n+1} \] 
as follows. By assumption, the map
\[ t \in [0,1] \mapsto (\gamma(t),\gamma'(t),\dots,\gamma^{(n)}(t)) \]
takes values in $GL_{n+1}^+$. By Gram-Schmidt orthonormalization, there exist unique $\mathcal{F}_\gamma(t) \in SO_{n+1}$ and $R_\gamma(t) \in Up_{n+1}^+$ such that
\begin{equation}\label{frenet_gram}
(\gamma(t),\gamma'(t),\dots,\gamma^{(n)}(t))=\mathcal{F_{\gamma}}(t) R_\gamma(t),
\end{equation}
where we recall that $Up_{n+1}^+$ is the space of upper triangular matrices with positive diagonal entries and real coefficients. 

\begin{definition}
The curve $\mathcal{F_{\gamma}}:[0,1] \rightarrow SO_{n+1}$ defined by~\eqref{frenet_gram} is called the \emph{Frenet frame curve} of the locally convex curve $\gamma:[0,1] \rightarrow M$. 
\end{definition}

This definition does not depend on the choice of a (positive) parametrization: indeed, recall that (proof of point $(i)$ in Proposition~\ref{proplocconv}) 
\[ (\gamma \circ \phi(t),(\gamma \circ \phi)'(t),\dots,(\gamma \circ \phi)^{(n)}(t))=(\gamma(\tau),\gamma'(\tau),\dots,\gamma^{(n)}(\tau))U \]
where $\tau=\phi(t)$ is a positive reparametrization, and where $U$ is a upper triangular matrix whose diagonal is given by $(1,\phi'(t), \dots,(\phi'(t)^n)$. So $U \in Up_{n+1}^+$, which implies that $\mathcal{F_{\gamma \circ \phi}}(t)=\mathcal{F_{\gamma}}(\tau)$. Let us also remark that the Frenet frame curve $\mathcal{F}_\gamma$ uniquely determines the curve $\gamma$.

One can still define a Frenet frame for generic curves which are not necessarily locally convex. Indeed, one can apply Gram-Schmidt to the linearly independent vectors $\gamma(t),\gamma'(t),\dots,\gamma^{(n-1)}(t)$ to obtain $n$ orthonormal vectors $u_0(t),u_1(t), \dots, u_{n-1}(t)$. Then, there is a unique vector $u_n(t)$ for which $u_0(t),u_1(t), \dots, u_{n-1}(t), u_n(t)$ is an positive orthonormal basis. We may thus set
\begin{equation}\label{frenet_gram2}
\mathcal{F_{\gamma}}(t)=(u_0(t),u_1(t), \dots, u_{n-1}(t), u_n(t)) \in SO_{n+1}
\end{equation}
and make the following more general definition. 

\begin{definition}
The curve $\mathcal{F_{\gamma}}:[0,1] \rightarrow SO_{n+1}$ defined by~\eqref{frenet_gram2} is called the \emph{Frenet frame curve} of the generic curve $\gamma:[0,1] \rightarrow M$. 
\end{definition}

Clearly, the latter definition coincides with the former when $\gamma$ is locally convex.

\medskip

Let us look at the special case $M=\SS^2$. As we already explained, generic curves are just immersions. Let us denote by $\mathbf{t}_\gamma(t)$ the unit tangent vector of $\gamma$ at the point $\gamma(t)$, that is
\[ \mathbf{t}_\gamma(t):=\frac{\gamma'(t)}{||\gamma'(t)||} \in \SS^2, \] 
and by $\mathbf{n}_\gamma(t)$ be the unit normal vector of $\gamma$ at the point $\gamma(t)$, that is 
\[ \mathbf{n}_\gamma(t):= \gamma(t) \times \mathbf{t}_\gamma(t) \]
where $\times$ is the cross-product in $\R^3$. We then have
\[ \mathcal{F_{\gamma}}(t)=(\gamma(t),\mathbf{t}_\gamma(t),\mathbf{n}_\gamma(t)) \in SO_3 \]
where $\mathbf{t}_\gamma(t)$ is the unit tangent and $\mathbf{n}_\gamma(t)$ the unit normal we defined above.  
The geodesic curvature $\kappa_\gamma(t)$ is by definition
\[ \kappa_\gamma(t):=\mathbf{t}_\gamma'(t)\cdot\mathbf{n}_\gamma(t) \]
where $\cdot$ is the Euclidean inner product. Here's a geometric definition of locally convex curves in $\SS^2$.

\begin{proposition}
A generic curve $\gamma:[0,1] \rightarrow \SS^2$ is locally convex if and only if $\kappa_\gamma(t)> 0$ for all $t \in (0,1)$. 
\end{proposition}

\begin{proof}
Without loss of generality, we may assume that $\gamma$ is parametrized by arc-length. Then $\mathbf{t}_\gamma(t)=\gamma'(t)$, and since $||\gamma(t)||=1$, $\gamma'(t)$ is orthogonal to $\gamma(t)$. Moreover, $\mathbf{n}_\gamma(t)= \gamma(t) \times \gamma'(t)$ and $\kappa_\gamma(t)=\gamma''(t)\cdot\mathbf{n}_\gamma(t)$. A simple computation gives
\[ \gamma''(t)=-\gamma(t)+\kappa_\gamma(t)\mathbf{n}_\gamma(t) \]
and hence we have the equality
\[ (\gamma(t),\gamma'(t),\gamma''(t))=\mathcal{F}_\gamma(t)R_\gamma(t) \]
with
\[ 
R_\gamma(t)=
\begin{pmatrix}
1 & 0 & -1 \\
0 & 1 & 0 \\
0 & 0 & \kappa_\gamma(t)
\end{pmatrix}. \]
Since $\mathcal{F}_\gamma(t)$ has determinant $1$, we have
\[ \mathrm{det}(\gamma(t),\gamma'(t),\gamma''(t))=\mathrm{det}R_\gamma(t)=\kappa_\gamma(t) \]
and this proves the statement.
\end{proof}

\medskip

Next let us consider the case $M=\SS^3$. Assume that $\gamma$ is generic, that is $\gamma(t),\gamma'(t),\gamma''(t)$ are linearly independent, so that its Frenet frame $\mathcal{F}_\gamma(t)$ can be defined, and let $e_1,e_2,e_3,e_4$ the canonical basis of $\R^4$. It is clear that
\[ \mathcal{F}_\gamma(t)e_1=\gamma(t), \quad \mathcal{F}_\gamma(t)e_2=\mathbf{t}_\gamma(t)=\frac{\gamma'(t)}{||\gamma'(t)||}.  \]
Let us now define the unit normal $\mathbf{n}_\gamma(t)$ and binormal $\mathbf{b}_\gamma(t)$ by the formulas
\[ \mathbf{n}_\gamma(t)=\mathcal{F}_\gamma(t)e_3, \quad \mathbf{b}_\gamma(t)=\mathcal{F}_\gamma(t)e_4  \]
so that  
\[ \mathcal{F_{\gamma}}(t)=(\gamma(t),\mathbf{t}_\gamma(t),\mathbf{n}_\gamma(t),\mathbf{b}_\gamma(t)) \in SO_4. \]
The geodesic curvature $\kappa_\gamma(t)$ is still defined by
\[ \kappa_\gamma(t):=\mathbf{t}_\gamma'(t)\cdot\mathbf{n}_\gamma(t) \]
but we further define the geodesic torsion $\tau_{\gamma}(t)$ by
\[ \tau_\gamma(t):=-\mathbf{b}_\gamma'(t)\cdot\mathbf{n}_\gamma(t). \]
It is clear that the geodesic curvature is never zero. We can then characterize locally convex curves in $\SS^3$.

\begin{proposition}
A generic curve $\gamma:[0,1] \rightarrow \SS^3$ is locally convex if and only if $\tau_\gamma(t)> 0$ for all $t \in (0,1)$. 
\end{proposition}

\begin{proof}
Without loss of generality, we may assume that $\gamma$ is parametrized by arc-length. Then, as before, $\mathbf{t}_\gamma(t)=\gamma'(t)$ and $\kappa_\gamma(t)=\gamma''(t)\cdot\mathbf{n}_\gamma(t)$. Also
\[ \gamma''(t)=-\gamma(t)+\kappa_\gamma(t)\mathbf{n}_\gamma(t) \]
and hence
\[ \gamma'''(t)=-\gamma'(t)+\kappa_\gamma(t)'\mathbf{n}_\gamma(t)+\kappa_\gamma(t)\mathbf{n}_\gamma'(t). \]
Since $\mathbf{b}_\gamma(t)\cdot \mathbf{n}_\gamma(t)=0$, we have
\[ \tau_\gamma(t)=-\mathbf{b}_\gamma'(t)\cdot \mathbf{n}_\gamma(t)=\mathbf{b}_\gamma(t)\cdot \mathbf{n}'_\gamma(t). \]
One then easily computes
\[ \gamma'''(t)\cdot\gamma(t)=0, \]
\[ \gamma'''(t)\cdot\gamma'(t)=-1-\kappa_\gamma(t)^2, \]
\[ \gamma'''(t)\cdot\mathbf{n}_\gamma(t)=\kappa_\gamma'(t), \]
\[ \gamma'''(t)\cdot\mathbf{b}_\gamma(t)=\kappa_\gamma(t)\tau_\gamma(t). \]
So we have the equality
\[ (\gamma(t),\gamma'(t),\gamma''(t),\gamma'''(t))=\mathcal{F}_\gamma(t)R_\gamma(t) \]
with
\[ 
R_\gamma(t)=
\begin{pmatrix}
1 & 0 & -1 & 0 \\
0 & 1 & 0 & -1-\kappa_\gamma(t)^2\\
0 & 0 & \kappa_\gamma(t) & \kappa_\gamma'(t) \\
0 & 0 & 0 & \kappa_\gamma(t)\tau_\gamma(t).
\end{pmatrix}. \]
Since $\mathcal{F}_\gamma(t)$ has determinant $1$ and $\kappa_\gamma(t)$ is never zero, we have
\[ \mathrm{det}(\gamma(t),\gamma'(t),\gamma''(t),\gamma'''(t))=\mathrm{det}R_\gamma(t)=\kappa_\gamma(t)^2\tau_\gamma(t) \]
and this proves the statement.
\end{proof}

\section{Spaces $\mathcal{L}\SS^n$, $\mathcal{L}\SS^n(Q)$, $\mathcal{L}\SS^n(z)$ and $\mathcal{G}\SS^n$, $\mathcal{G}\SS^n(Q)$, $\mathcal{G}\SS^n(z)$}\label{s34}

We can finally define the space of curves we will be interested in. Let us first start with locally convex curves (the case of generic curves will be analogous).

\begin{definition}\label{LSn}
We define ${\mathcal{L}\SS^{n}}$ to be the set of all locally convex curves $\gamma : [0,1] \rightarrow \SS^{n}$ such that $\mathcal{F}_\gamma(0)=I$, where $I \in SO_{n+1}$ is the identity matrix.
\end{definition}

Hence ${\mathcal{L}\SS^{n}}$ is the set of all locally convex curves with standard initial (Frenet) frame and an arbitrary final frame. Given $Q_0 \in SO_{n+1}$, one may consider the space ${\mathcal{L}\SS^{n}}_{Q_0}$ to be the set of all locally convex curves in $\gamma : [0,1] \rightarrow \SS^{n}$ such that $\mathcal{F}_\gamma(0)=Q_0$. As the proposition below, this is not more general.

\begin{proposition}
The space ${\mathcal{L}\SS^{n}}_{Q_0}$ is homeomorphic to the space $SO_{n+1} \times {\mathcal{L}\SS^{n}}$.
\end{proposition}

\begin{proof}
Given $Q_0 \in SO_{n+1}$ and $\gamma \in {\mathcal{L}\SS^{n}}$, the curve $Q_0\gamma$ is locally convex (Proposition~\ref{proplocconv}, $(ii)$) and clearly $\mathcal{F}_{Q_0\gamma}(0)=Q_0$. This defines a map
\[ (Q_0,\gamma) \in SO_{n+1} \times {\mathcal{L}\SS^{n}} \mapsto Q_0\gamma \in {\mathcal{L}\SS^{n}}_{Q_0} \]
which is clearly continuous, whose inverse
\[ \gamma \in {\mathcal{L}\SS^{n}}_{Q_0} \mapsto (Q_0,Q_0^{-1}\gamma) \in SO_{n+1} \times {\mathcal{L}\SS^{n}} \]
is also continuous. Therefore ${\mathcal{L}\SS^{n}}_{Q_0}$ and $SO_{n+1} \times {\mathcal{L}\SS^{n}}$ are homeomorphic. 
\end{proof}

\medskip

This proves there is no loss of generality in studying the space ${\mathcal{L}\SS^{n}}$. We can also define a subset of ${\mathcal{L}\SS^{n}}$, in which the final frame is also fixed.

\begin{definition}\label{LSnQ}
For $Q \in SO_{n+1}$, we define ${\mathcal{L}\SS^{n}}(Q)$ as the subset of ${\mathcal{L}\SS^{n}}$ for which $\mathcal{F}_\gamma(1)=Q$.
\end{definition}

Finally, let us recall that $\Pi_{n+1} : \mathrm{Spin}_{n+1} \rightarrow SO_{n+1}$ is the double universal cover, $n \geq 2$. We denote by $\mathbf{1}$ the identity element in $\mathrm{Spin}_{n+1}$, and by $-\mathbf{1}$ the unique non-trivial element in $\mathrm{Spin}_{n+1}$ such that $\Pi_{n+1}(-\mathbf{1})=I$. The Frenet frame curve $\mathcal{F_{\gamma}}:[0,1] \rightarrow SO_{n+1}$ can then be uniquely lifted to a curve $\tilde{\mathcal{F}_{\gamma}}:[0,1] \rightarrow \mathrm{Spin}_{n+1}$ such that $\mathcal{F_{\gamma}}=\Pi_{n+1} \circ \tilde{\mathcal{F}_{\gamma}}$ and $\tilde{\mathcal{F}_{\gamma}}(0)=\mathbf{1}$. 

\begin{definition}\label{LSnz}
For $z \in \mathrm{Spin}_{n+1}$, we define ${\mathcal{L}\SS^{n}}(z)$ as the subset of ${\mathcal{L}\SS^{n}}(\Pi_{n+1}(z))$ for which $\tilde{\mathcal{F}_{\gamma}}(1)=z$.
\end{definition}

It is clear to see from the definitions that ${\mathcal{L}\SS^{n}}(\Pi_{n+1}(z))$ is the disjoint union of ${\mathcal{L}\SS^{n}}(z)$ and ${\mathcal{L}\SS^{n}}(-z)$. 

\medskip

Replacing ``locally convex" by ``generic", we can define in the same way ${\mathcal{G}\SS^{n}}$, ${\mathcal{G}\SS^{n}}(Q)$ and ${\mathcal{G}\SS^{n}}(z)$. 

\begin{definition}\label{IS2}
We define ${\mathcal{G}\SS^{n}}$ to be the set of all generic curves $\gamma : [0,1] \rightarrow \SS^{n}$ such that $\mathcal{F}_\gamma(0)=I$, where $I \in SO_{n+1}$ is the identity matrix. For $Q \in SO_{n+1}$ (respectively for $z \in\mathrm{Spin}_{n+1}$) we define ${\mathcal{G}\SS^{n}}(Q)$ (respectively ${\mathcal{G}\SS^{n}}(z)$) as the subset of ${\mathcal{G}\SS^{n}}$ for which $\mathcal{F}_\gamma(1)=Q$ (respectively $\tilde{\mathcal{F}}_\gamma(1)=z$).
\end{definition}

\medskip

The study of the spaces $\mathcal{L}\SS^{n}(Q)$, when $Q$ varies in $SO_{n+1}$, or more generally $\mathcal{L}\SS^{n}(z)$, when $z$ varies in $\mathrm{Spin}_{n+1}$, is the main topic in this thesis.

Exactly as in~\cite{SS12}, we can already reduce the study of these spaces to a finite number of spaces. Recall that we defined a Bruhat action $B : Up_{n+1}^1 \times SO_{n+1} \rightarrow SO_{n+1}$ in Chapter~\ref{chapter2}, \S\ref{s23}, where $B(U,Q)$ is the Gram-Schmidt orthonormalization of the matrix $UQ$, for $U \in Up_{n+1}^1$ and $Q \in  SO_{n+1}$. This action induces an action of $Up_{n+1}^1$ on $\mathcal{L}\SS^{n}$. Indeed, given $U \in Up_{n+1}^1$ and $\gamma \in \mathcal{L}\SS^{n}$, let us define the curve $B(U,\gamma)$ by
\[ B(U,\gamma)(t)=B(U,\mathcal{F}_\gamma(t))e_1. \]  
This curve can also be written as
\[ B(U,\gamma)(t)=\frac{U\gamma(t)}{||U\gamma(t)||} \]
and so by Proposition~\ref{proplocconv}, $(ii)$, and Corollary~\ref{corlocconv}, $(i)$, $B(U,\gamma)$ is locally convex. Moreover, by construction
\[ \mathcal{F}_{B(U,\gamma)}(t)=B(U,\mathcal{F}_\gamma(t)) \] 
and hence $B(U,\gamma) \in \mathcal{L}\SS^{n}$ (since $B(U,I)=I$) and for $Q \in SO_{n+1}$ (respectively $z \in \mathrm{Spin}_{n+1}$), if $\gamma \in \mathcal{L}\SS^{n}(Q)$ (respectively $\gamma \in \mathcal{L}\SS^{n}(z)$), then $\gamma \in \mathcal{L}\SS^{n}(B(U,Q))$ (respectively $\gamma \in \mathcal{L}\SS^{n}(B(U,z))$). If $Q \in SO_{n+1}$ and $Q' \in SO_{n+1}$ are Bruhat equivalent, there exists $U \in Up_{n+1}^1$ such that $B(U,Q)=Q'$ and hence $B(U^{-1},Q')=Q$, we have a continuous map
\[ \gamma \in \mathcal{L}\SS^{n}(Q) \mapsto B(U,\gamma) \in \mathcal{L}\SS^{n}(Q') \] 
which has a continuous inverse
\[ \gamma \in \mathcal{L}\SS^{n}(Q') \mapsto B(U^{-1},\gamma) \in \mathcal{L}\SS^{n}(Q). \]
The same construction can be done if $z \in \mathrm{Spin}_{n+1}$ and $z' \in \mathrm{Spin}_{n+1}$ are Bruhat equivalent, hence we have the following proposition.

\begin{proposition}\label{bruhathomeo}
If $Q \in SO_{n+1}$ and $Q' \in SO_{n+1}$ (respectively $z \in \mathrm{Spin}_{n+1}$ and $z' \in \mathrm{Spin}_{n+1}$) are Bruhat equivalent, then the spaces $\mathcal{L}\SS^{n}(Q)$ and $\mathcal{L}\SS^{n}(Q')$ (respectively $\mathcal{L}\SS^{n}(z)$ and $\mathcal{L}\SS^{n}(z')$) are homeomorphic. 
\end{proposition}

\section{Jacobian and quasi-Jacobian curves}\label{s35}

In this section, we introduce another class of curves, namely Jacobian curves and the related holonomic curves, that are deeply connected to locally convex curves. We will also introduce quasi-Jacobian curves and quasi-holonomic curves which will play the same role for generic curves. 

On the one hand, these curves appear in the study of linear homogeneous ordinary differential equations, and constitute a motivation for the study of locally convex curves. On the other hand, they will be used to provide a convenient topology on the space of locally convex curves (see \S\ref{s36}).

To start with, consider a linear homogeneous ordinary differential equation of order $n+1$, that is
\begin{equation}\label{lode}
y^{(n+1)}(t)+a_n(t)y^{(n)}(t)+\cdots+a_0(t)y(t)=0
\end{equation} 
where the functions $a_i$, for $0 \leq i \leq n$, are smooth functions on the interval $[0,1]$. Letting
\[ Y(t)=
\begin{pmatrix}
y(t) \\
y'(t) \\
\vdots \\
y^{(n)}(t)
\end{pmatrix}
 \]
the equation~\eqref{lode} is easily seen to be equivalent to 
\begin{equation}\label{lode2}
Y'(t)=A(t)Y(t)
\end{equation}
where $A(t)$ is the companion matrix
\[ A(t)=
\begin{pmatrix}
0 & 1 & 0 & 0 & \ldots & 0  \\
0 & 0 & 1 & 0 & \ldots & 0 \\
0 & 0 & 0 & 1 &  & 0 \\
\vdots & \vdots  &  \vdots & & \ddots &  \\
0 & 0 & 0 & 0 & \vdots & 1  \\
-a_0(t) & -a_1(t) & -a_2(t) & \ldots &  & -a_n(t)  \\
\end{pmatrix}.
 \]
The space of solutions of~\eqref{lode2} is a vector space of dimension $n+1$: a fundamental set of solutions is a particular choice of basis of this space. Choosing such a set of solutions
\[ Y_1(t)=
\begin{pmatrix}
y_1(t) \\
y_1'(t) \\
\vdots \\
y_1^{(n)}(t)
\end{pmatrix},
\quad
Y_2(t)=
\begin{pmatrix}
y_2(t) \\
y_2'(t) \\
\vdots \\
y_2^{(n)}(t)
\end{pmatrix},
\quad
\dots,
\quad
Y_{n+1}(t)=
\begin{pmatrix}
y_{n+1}(t) \\
y_{n+1}'(t) \\
\vdots \\
y_{n+1}^{(n)}(t)
\end{pmatrix}
 \] 
we have the matrix equality
\[(Y_1'(t), Y_2'(t), \dots , Y_{n+1}'(t))=A(t)(Y_1(t), Y_2(t), \dots , Y_{n+1}(t))\]
and the vectors $Y_1(t), Y_2(t), \dots , Y_{n+1}(t)$ are linearly independent, that is
\[ \mathrm{det}(Y_1(t), Y_2(t), \dots , Y_{n+1}(t)) \neq 0. \]
Hence if we define
\[ \gamma(t)=(y_1(t), y_2(t), \dots, y_{n+1}(t)) \]
the curve $\gamma$ is locally convex (positive or negative depending on the sign of the determinant). So for any set of fundamental solutions of~\eqref{lode} one can associate a locally convex curve: conversely one can show that any locally convex curve is of this form.

\medskip

Now we will be interested in characterizing the Frenet frame curve associated to a locally convex curve. Consider a curve
\[ \Gamma : [0,1] \rightarrow SO_{n+1} \]
and define its logarithmic derivative $\Lambda(t)$ by
\[ \Lambda(t)=(\Gamma(t))^{-1}\Gamma'(t) \]
that is 
\[ \Gamma'(t)=\Gamma(t)\Lambda(t). \]
Since $\Gamma$ takes values in $SO_{n+1}$, $\Lambda$ takes values in its Lie algebra, that is $\Lambda(t)$ is a skew-symmetric matrix for all $t \in [0,1]$. 

When $\Gamma=\mathcal{F}_\gamma$ is the Frenet frame curve of a locally convex curve, its logarithmic derivative $\Lambda(t)$ is not an arbitrary skew-symmetric matrix. For instance, if $\gamma : [0,1] \rightarrow \SS^2$ is locally convex, then 
\[ \mathcal{F_{\gamma}}(t)=(\gamma(t),\mathbf{t}_\gamma(t),\mathbf{n}_\gamma(t)) \in SO_3 \]
and by simple computations one obtains
\begin{equation}\label{logderives2}
\Lambda_\gamma(t)=(\mathcal{F}_\gamma(t))^{-1}\mathcal{F}_\gamma'(t)=
\begin{pmatrix}
0 & -||\gamma'(t)|| & 0 \\
||\gamma'(t)|| & 0 & -||\gamma'(t)||\kappa_\gamma(t) \\
0 & ||\gamma'(t)||\kappa_\gamma(t) & 0
\end{pmatrix}.
\end{equation}
In the same way, if $\gamma : [0,1] \rightarrow \SS^3$ is locally convex, then 
\[ \mathcal{F_{\gamma}}(t)=(\gamma(t),\mathbf{t}_\gamma(t),\mathbf{n}_\gamma(t),\mathbf{b}_\gamma(t)) \in SO_4 \]
and one gets
\begin{equation}\label{logderives3}
\Lambda_\gamma(t)=
\begin{pmatrix}
0 & -||\gamma'(t)|| & 0  & 0 \\
||\gamma'(t)|| & 0 & -||\gamma'(t)||\kappa_\gamma(t) & 0  \\
0 & ||\gamma'(t)||\kappa_\gamma(t) & 0 & -||\gamma'(t)||\tau_\gamma(t) \\
0 & 0 & ||\gamma'(t)||\tau_\gamma(t) & 0
\end{pmatrix}.
\end{equation}
This is in fact a general phenomenon. Let us define the set $\mathfrak{J}$ of tridiagonal skew-symmetric matrices with positive subdiagonal entries, that is matrices of the form
\[ \begin{pmatrix}
0 & -c_1 & 0 & \ldots & 0 \\
c_1 & 0 & -c_2 &  & 0 \\
 & \ddots & \ddots & \ddots & 0 \\
0 & 0 & c_{n-1} & 0 & -c_n  \\
0 & 0 &  0 & c_n & 0
\end{pmatrix}, \quad c_1>0, \dots, c_n>0. \]
Then we make the following definition.

\begin{definition}\label{jacobian}
A curve $\Gamma : [0,1] \rightarrow SO_{n+1}$ is \emph{Jacobian} if its logarithmic derivative $\Lambda(t)=(\Gamma(t))^{-1}\Gamma'(t)$ belongs to $\mathfrak{J}$ for all $t \in [0,1]$. 
\end{definition}

The interest of this definition is that Jacobian curves characterize Frenet frame curves of locally convex curves. Indeed, we have the following proposition.

\begin{proposition}\label{propjacobian}
Let $\Gamma : [0,1] \rightarrow SO_{n+1}$ be a smooth curve with $\Gamma(0)=I$. Then $\Gamma$ is Jacobian if and only if there exists $\gamma \in \mathcal{L}\SS^n$ such that $\mathcal{F}_\gamma=\Gamma$.
\end{proposition}

This is exactly the content of Lemma 2.1 in~\cite{SS12}, to which we refer for a proof. Hence there is a one-to-one correspondence between locally convex curves in $\mathcal{L}\SS^n$ and Jacobian curves starting at the identity: if $\gamma \in \mathcal{L}\SS^n$, its Frenet frame curve is such a Jacobian curve, and conversely, if $\Gamma$ is a Jacobian curve with $\Gamma(0)=I$, then if we define $\gamma_\Gamma$ by setting $\gamma_\Gamma(t)=\Gamma(t)e_1$ where $e_1$ denotes the first vector of the canonical basis of $\R^{n+1}$, $\gamma_\Gamma \in \mathcal{L}\SS^n$.

Now consider a smooth curve $\Lambda : [0,1] \rightarrow \mathfrak{J}$. Then $\Lambda$ is the logarithmic derivative of a Jacobian curve $\Gamma : [0,1] \rightarrow SO_{n+1}$ if and only if $\Gamma$ solves 
\[ \Gamma'(t)=\Gamma(t)\Lambda(t). \]
If $\Gamma$ solves the above equation, then so does $Q\Gamma$, for $Q \in SO_{n+1}$, since the logarithmic derivative of $\Gamma$ and $Q\Gamma$ are equal. But the initial value problem 
\[ \Gamma'(t)=\Gamma(t)\Lambda(t), \quad \Gamma(0)=I \]
has a unique solution. We can state the following proposition, which follows from Proposition~\ref{propjacobian}.

\begin{proposition}\label{unique}
Given a curve $\Lambda : [0,1] \rightarrow \mathfrak{J}$, there is a unique curve $\gamma \in \mathcal{L}\SS^n$ such that $\Lambda_\gamma(t)=\mathcal{F}_\gamma(t)^{-1}\mathcal{F}'_\gamma(t)=\Lambda(t)$.
\end{proposition} 

Consider the locally convex curve $\xi : [0,1] \rightarrow \SS^n$ defined in Example~\ref{ex1}. It is easy to see that the logarithmic derivative $\Lambda_\xi(t)$ is constant. From what we explained, any other curve which has constant logarithmic derivative has to be of the form $Q\xi$, for some $Q \in SO_{n+1}$. More precisely, given any matrix $\Lambda \in \mathfrak{J}$, the map
\[ \Gamma_\Lambda(t)=\exp(t\Lambda) \in SO_{n+1} \]
is a Jacobian curve whose logarithmic derivative is constant equal to $\Lambda$. The curve $\gamma_\Lambda$ defined by $\gamma_\Lambda(t)=\Gamma_\Lambda(t)e_1$ is then locally convex, and there exists $Q \in SO_{n+1}$ such that $\gamma_\Lambda=Q\xi$.

\medskip

Now the Frenet frame curve $\mathcal{F}_\gamma : [0,1] \rightarrow SO_{n+1}$ of $\gamma \in \mathcal{L}\SS^n$ can be lifted to a curve
\[ \tilde{\mathcal{F}}_\gamma : [0,1] \rightarrow \mathrm{Spin}_{n+1}, \]
that is $\mathcal{F}_\gamma = \tilde{\mathcal{F}}_\gamma \circ \Pi_{n+1}$ where $\Pi_{n+1} : \mathrm{Spin}_{n+1}\rightarrow SO_{n+1}$ is the universal cover projection. Such a lifted Frenet frame curve $\tilde{\mathcal{F}}_\gamma$ is thus characterized by the following definition.

\begin{definition}\label{holonomic}
A curve $\tilde{\Gamma} : [0,1] \rightarrow \mathrm{Spin}_{n+1}$ is \emph{holonomic} if the projected curve $\Gamma=\tilde{\Gamma} \circ \Pi_{n+1}$ is a Jacobian curve.
\end{definition}    

Let us further add the following two definitions.

\begin{definition}\label{globjacobian}
A curve $\Gamma : [0,1] \rightarrow SO_{n+1}$ is \emph{globally Jacobian} if the curve 
\[ \gamma(t)=\Gamma(t)e_1 \]
is (globally) convex.
\end{definition}   

\begin{definition}\label{globholonomic}
A curve $\tilde{\Gamma} : [0,1] \rightarrow \mathrm{Spin}_{n+1}$ is \emph{globally holonomic} if the projected curve $\Gamma=\tilde{\Gamma} \circ \Pi_{n+1}$ is a globally Jacobian curve.
\end{definition}   

\medskip

To conclude, we can also characterize the Frenet frame curve associated to a generic curve. Let us define the set $\mathfrak{Q}$ of tridiagonal skew-symmetric matrices of the form
\[ \begin{pmatrix}
0 & -c_1 & 0 & \ldots & 0 \\
c_1 & 0 & -c_2 &  & 0 \\
 & \ddots & \ddots & \ddots & 0 \\
0 & 0 & c_{n-1} & 0 & -c_n  \\
0 & 0 &  0 & c_n & 0
\end{pmatrix}, \quad c_1>0, \dots,c_{n-1}>0, c_n \in \R. \]
Clearly, $\mathfrak{J}$ is contained in $\mathfrak{Q}$ and we have the following definition and proposition which are analogous to Definition~\ref{jacobian} and Proposition~\ref{propjacobian}.

\begin{definition}\label{jacobian2}
A curve $\Gamma : [0,1] \rightarrow SO_{n+1}$ is \emph{quasi-Jacobian} if its logarithmic derivative $\Lambda(t)=(\Gamma(t))^{-1}\Gamma'(t)$ belongs to $\mathfrak{Q}$ for all $t \in [0,1]$. 
\end{definition}

\begin{proposition}\label{propjacobian2}
Let $\Gamma : [0,1] \rightarrow SO_{n+1}$ be a smooth curve with $\Gamma(0)=I$. Then $\Gamma$ is quasi-Jacobian if and only if there exists $\gamma \in \mathcal{G}\SS^n$ such that $\mathcal{F}_\gamma=\Gamma$.
\end{proposition} 

Let us also make the following definition, which is analogous to Definition~\ref{holonomic}. 

\begin{definition}\label{holonomic2}
A curve $\tilde{\Gamma} : [0,1] \rightarrow \mathrm{Spin}_{n+1}$ is \emph{quasi-holonomic} if the projected curve $\Gamma=\tilde{\Gamma} \circ \Pi_{n+1}$ is a quasi-Jacobian curve.
\end{definition}    

\section{Topology on the space of locally convex curves and generic curves}\label{s36}

Let us equip the spaces $\mathcal{L}\SS^n$ and $\mathcal{G}\SS^n$ with a topology which will be convenient for our purposes. Let us start with the case of $\mathcal{L}\SS^n$ (the case of $\mathcal{G}\SS^n$ will be analogous). 

Perhaps the most natural choice would be to consider the $C^k$-topology, for an integer $k \geq n$ (the choice of the particular $k \geq n$ is not important). When choosing this topology, the corresponding topological space will be simply denoted by $\mathcal{L}\SS^n$.  However, we will have to make concatenations of curves, and this will result in curves which are non-smooth at some points. So it will be useful to assume that our curves have just some weak regularity (that they are smooth almost everywhere, but not necessarily everywhere). 

Following~\cite{Sal13} and~\cite{SS12}, we will equip the space $\mathcal{L}\SS^n$ (or more precisely a larger space $\hat{\mathcal{L}}\SS^n$) with a Hilbert manifold structure (and therefore a topology), using the correspondence between Frenet frame curves of locally convex curves and Jacobian curves.

For $\gamma \in \mathcal{L}\SS^n$, its Frenet frame curve is Jacobian and hence its logarithmic derivative is of the form
\[ \Lambda_\gamma(t)=\begin{pmatrix}
0 & -c_1(t) & 0 & \ldots & 0 \\
c_1(t) & 0 & -c_2(t) &  & 0 \\
 & \ddots &   & \ddots & 0 \\
0 & 0 & c_{n-1}(t) &  & -c_n(t)  \\
0 & 0 &  0 & c_n(t) & 0
\end{pmatrix} \]
for some functions $c_i : [0,1] \rightarrow (0,+\infty)$, $1 \leq i \leq n$. We thus have a correspondence:
\begin{equation}\label{corr1}
\gamma \in \mathcal{L}\SS^n \longleftrightarrow (c_1, \dots, c_n) \in C([0,1],(0,+\infty))^n=C([0,1],(0,+\infty)^n).
\end{equation}
By definition, the Frenet frame curve $\mathcal{F}_\gamma$ solves the initial value problem
\begin{equation}\label{ivp}
\mathcal{F}_\gamma'(t)=\mathcal{F}_\gamma(t)\Lambda_\gamma(t), \quad \mathcal{F}_\gamma(0)=I.
\end{equation}
If the functions $c_i$ are continuous, it is well-known that~\eqref{ivp} can be solved, but~\eqref{ivp} can be solved even in low-regularity: assuming that $c_i$ are in $L^2$, that is
\[ \int_{0}^1c_i^2(t)dt < +\infty, \]
then~\eqref{ivp} has a unique solution (see for instance~\cite{CL84}, page 63). So let us consider $L^2([0,1],(0,+\infty))$ the set of square-integrable functions $f: [0,1] \rightarrow (0,+\infty)$, so that under the correspondence~\eqref{corr1}, $\mathcal{L}\SS^n$ can be identified to a subset $(L^2([0,1],(0,+\infty))^n$. However $L^2([0,1],(0,+\infty))$ is not a vector space, so this does not define a Hilbert manifold structure. To solve this issue, consider the diffeomorphism
\[ \varphi : (0,+\infty) \rightarrow \R, \quad \varphi(x)=x-1/x. \] 
For any function $f : [0,1] \rightarrow (0,+\infty)$, we have a function $\hat{f} : [0,1] \rightarrow \R$ defined by $\hat{f}=\varphi \circ f$. The specific choice of $\varphi$ implies that
\[ f \in L^2([0,1],(0,+\infty)) \Leftrightarrow \hat{f} \in L^2([0,1]) \]
where $L^2([0,1])=L^2([0,1],\R)$ is the Hilbert space of square-integrable functions. The correspondence~\eqref{corr1} enables us to define the space $\hat{\mathcal{L}}\SS^n$ by the following correspondence 
\begin{equation}\label{corr2}
\gamma \in \hat{\mathcal{L}}\SS^n \longleftrightarrow (\hat{c}_1, \dots, \hat{c}_n) \in (L^2([0,1]))^n=L^2([0,1],\R^n).
\end{equation}
This defines a (trivial) Hilbert manifold structure on $\hat{\mathcal{L}}\SS^n$, and hence a topology. The spaces $\hat{\mathcal{L}}\SS^n(Q)$, for $Q \in SO_{n+1}$ and $\hat{\mathcal{L}}\SS^n(z)$, for $z \in \mathrm{Spin}_{n+1}$ are defined by further requiring that the final Frenet frame (lifted Frenet frame) is equal to $Q$ (equal to $z$). These are closed subsets of $\hat{\mathcal{L}}\SS^n$, and they are in fact closed submanifolds of finite codimension (see~\cite{Sal13} for the case $n=2$; the general case will be similar).  

\medskip

Observe that $\gamma \in \hat{\mathcal{L}}\SS^n$ is not a smooth curve everywhere (though it is smooth almost everywhere) and hence the drawback of this definition is that the concept of local convexity does not make sense at any point. On the other hand, this space $\hat{\mathcal{L}}\SS^n$ will allow us to make concatenations without being bothered by the loss of smoothness at some points. Considering a locally convex curve as being an element in $\mathcal{L}\SS^n$ or in $\hat{\mathcal{L}}\SS^n$ has both advantages and drawback. Fortunately, we have the following result.

\begin{proposition}\label{2topology}
The set $\mathcal{L}\SS^n$ (with the $C^k$ topology, $k \geq n$) is a dense subset of $\hat{\mathcal{L}}\SS^n$. Moreover, the inclusion $\mathcal{L}\SS^n \hookrightarrow \hat{\mathcal{L}}\SS^n$ induces a map $\mathcal{L}\SS^n \rightarrow \hat{\mathcal{L}}\SS^n$ of topological spaces which is a homotopy equivalence. In particular, these spaces are contractible.
\end{proposition}

The first part of this proposition is obvious (it follows from the density of smooth functions in the space of square-integrable functions). Using this, and Proposition 2.1 of~\cite{Sal13}, the second part of the statement follows.

Hence in the sequel, we will identify the topological spaces $\mathcal{L}\SS^n$ and $\hat{\mathcal{L}}\SS^n$, and both will be simply denoted by $\mathcal{L}\SS^n$. In practice, this means that we will consider that our locally convex curves are smooth, but in the construction, we will not be bothered by the loss of smoothness due to juxtaposition of curves. In the same way, we identify $\hat{\mathcal{L}}\SS^n(Q)$ and $\hat{\mathcal{L}}\SS^n(z)$ with $\mathcal{L}\SS^n(Q)$ and $\mathcal{L}\SS^n(z)$, for $Q \in SO_{n+1}$ and $z \in \mathrm{Spin}_{n+1}$. 

\medskip 

To conclude, let us consider the space of generic curves $\mathcal{G}\SS^n$. The only thing we used to define $\hat{\mathcal{L}}\SS^n$ was Proposition~\ref{propjacobian} that allows us to identify Frenet frame curves of locally convex curves with Jacobian curves, that is, curves whose logarithmic derivatives belongs to $\mathfrak{J}$. But recall that Proposition~\ref{propjacobian2} allows us to identify Frenet frame curves of generic curves with quasi-Jacobian curves, that is, curves whose logarithmic derivatives belong to $\mathfrak{Q}$; we can thus define in a similar manner the topological space $\hat{\mathcal{G}}\SS^n$. Proposition~\ref{2topology} also holds in this case, and hence we will also identify $\mathcal{G}\SS^n$, equipped with the $C^k$ topology for $k \geq n$, with $\hat{\mathcal{G}}\SS^n$.


%
%

\chapter{Statement of the results}
\label{chapter4}
Given $z \in \mathrm{Spin}_{n+1}$, the topology of the spaces $\mathcal{G}\SS^n(z)$ is well understood. The difficult question is the topology of the spaces $\mathcal{L}\SS^n(z)$. This chapter is devoted to recall some previous results on this direction and the statement of our main results.

\section{Known results for any $n \geq 2$}\label{s41}

Recall that $\mathcal{L}\SS^n(I)$ is the space of locally convex curves in $\SS^n$ whose initial and final Frenet frame are equal to the identity matrix in $SO_{n+1}$.

\begin{theorem}[M. Z. Shapiro, \cite{Sha93}]\label{thmsha}
The space $\mathcal{L}\SS^n(I)$ has $3$ connected components if $n$ is even, and $2$ if $n$ is odd. 
\end{theorem}

The fact that $\mathcal{L}\SS^n(I)$ has $3$ connected components when $n$ is even is related to the existence of closed convex curves on any even-dimensional sphere (such examples were given in Example~\ref{exclosed}, Chapter~\ref{chapter3}, \S\ref{s32}). Let us recall also the following result.

\begin{theorem}[M. Z. Shapiro, \cite{Sha93}, Anisov, \cite{Ani98}]
\label{thmani}
The space $\mathcal{L}\SS^n(z)$ has exactly two connected components if there exist convex curves in $\mathcal{L}\SS^n(z)$, and one otherwise. If $\mathcal{L}\SS^n(z)$ has two connected components, one is made of convex curves, and this component is contractible.   
\end{theorem}

To formulate precisely the results of~\cite{SS12}, we need some notations. For a positive integer $m$ and an integer $s \in \Z$ such that $|s|\leq m$ and $s \equiv m$ (mod $2$), let
\[ M^m_s=\mathrm{diag}(-1, \dots, -1, 1, \dots, 1) \in O_{n+1} \]
be the diagonal matrix of size $m$, whose first $(m-s)/2$ entries are equal to $-1$, and last $(m+s)/2$ entries are equal to $1$. The trace and the signature of $M_s^m$ are both equal to $s$, and $M^m_s \in SO_{n+1}$ if and only if $s\equiv m$ (mod $4$). When $M^m_s \in SO_{n+1}$, let us denote by $\pm w_s^m \in \mathrm{Spin}_{n+1}$ its two preimages. The two preimages of the identity matrix $I \in SO_{n+1}$ will be denoted by $\mathbf{1}$ and $-\mathbf{1}$. The first result of~\cite{SS12} can be stated as follows. 

\begin{theorem}[Saldanha-B. Shapiro, \cite{SS12}]\label{thmss1}
For any $n \geq 2$, any $Q \in SO_{n+1}$ and any $z \in \mathrm{Spin}_{n+1}$ one has
\begin{itemize}
\item[(i)] Each space $\mathcal{L}\SS^n(Q)$ is homeomorphic to one of the spaces $\mathcal{L}\SS^n(M_s^{n+1})$, where $|s|\leq n+1$ and $s \equiv n+1$ (mod $4$). 
\item[(ii)] For $n$ even, each space $\mathcal{L}\SS^n(z)$ is homeomorphic to one of the spaces $\mathcal{L}\SS^n(\mathbf{1})$, $\mathcal{L}\SS^n(-\mathbf{1})$, $\mathcal{L}\SS^n(w_{n-3}^{n+1})$, $\mathcal{L}\SS^n(w_{n-7}^{n+1})$, $\dots$, $\mathcal{L}\SS^n(w_{-n+5}^{n+1})$, $\mathcal{L}\SS^n(w_{-n+1}^{n+1})$.  
\item[(iii)] For $n$ odd, each space $\mathcal{L}\SS^n(z)$ is homeomorphic to one of the spaces $\mathcal{L}\SS^n(\mathbf{1})$, $\mathcal{L}\SS^n(-\mathbf{1})$, $\mathcal{L}\SS^n(w_{n-3}^{n+1})$, $\mathcal{L}\SS^n(w_{n-7}^{n+1})$, $\dots$, $\mathcal{L}\SS^n(w_{-n+3}^{n+1})$, $\mathcal{L}\SS^n(w_{-n-1}^{n+1})$, $\mathcal{L}\SS^n(-w_{-n-1}^{n+1})$.  
\end{itemize}
\end{theorem}

The most general question one can ask at this point is the following one.

\begin{question}\label{mainquestion}
Determine the topology of the spaces appearing in Theorem~\ref{thmss1}.
\end{question}

This is certainly a very difficult problem in general, so we may also ask the following question.

\begin{question}\label{mainquestion0}
Are the topological spaces appearing in Theorem~\ref{thmss1} pairwise non-homeomorphic?
\end{question}

The next result of~\cite{SS12} gives some information on the topology of these spaces. Let us define $\Omega SO_{n+1}(Q)$ (respectively $\Omega \mathrm{Spin}_{n+1}(z)$) to be the space of all continuous curves $\alpha : [0,1] \rightarrow SO_{n+1}$ (respectively $\alpha : [0,1] \rightarrow \mathrm{Spin}_{n+1}$) with $\alpha(0)=I$ and $\alpha(1)=Q$ (respectively $\alpha(0)=\mathbf{1}$ and $\alpha(1)=z$). Using the Frenet frame, we define the following \emph{Frenet frame injections}
\[ \mathcal{F}_{[Q]} : \gamma \in \mathcal{L}\SS^n(Q) \mapsto  \mathcal{F}_\gamma \in \Omega SO_{n+1}(Q)   \]
and
\[ \tilde{\mathcal{F}}_{[z]} : \gamma \in \mathcal{L}\SS^n(z) \mapsto \tilde{\mathcal{F}}_\gamma \in \Omega \mathrm{Spin}_{n+1}(z). \]
It is well-known that different values of $Q \in SO_{n+1}$ (respectively $z \in \mathrm{Spin}_{n+1}$) give rises to homeomorphic spaces  $\Omega SO_{n+1}(Q)$ (respectively $\Omega \mathrm{Spin}_{n+1}(z)$), therefore we can drop $Q$ (respectively $z$) from the notations and write $\Omega SO_{n+1}$ (respectively $\Omega \mathrm{Spin}_{n+1}$).

\begin{theorem}[Saldanha-B. Shapiro, \cite{SS12}]\label{thmss2}
For any $n \geq 2$, any $Q \in SO_{n+1}$ and any $z \in \mathrm{Spin}_{n+1}$, the Frenet frame injections $\mathcal{F}_{[Q]}$ and $\tilde{\mathcal{F}}_{[z]}$ induce surjective homomorphisms between homotopy and homology groups with real coefficients: for any $k \in \N$, the induced homomorphisms
\[ \pi_k(\mathcal{L}\SS^n(Q)) \rightarrow \pi_k(\Omega SO_{n+1}), \quad H_k(\mathcal{L}\SS^n(Q),\R) \rightarrow H_k(\Omega SO_{n+1},\R)  \]
and 
\[ \pi_k(\mathcal{L}\SS^n(z)) \rightarrow \pi_k(\Omega \mathrm{Spin}_{n+1}), \quad H_k(\mathcal{L}\SS^n(z),\R) \rightarrow H_k(\Omega \mathrm{Spin}_{n+1},\R)  \]
are surjective. Moreover, for $|s|\leq 1$, the Frenet frame injections $\mathcal{F}_{[M_s^{n+1}]}$ and $\tilde{\mathcal{F}}_{[w_s^{n+1}]}$ are weak homotopy equivalence, and therefore we have homeomorphisms
\[ \mathcal{L}\SS^n(M_s^{n+1}) \approx \Omega SO_{n+1}, \quad \mathcal{L}\SS^n(w_s^{n+1}) \approx \Omega \mathrm{Spin}_{n+1}.  \] 
\end{theorem}  

Let us recall that a map is a weak homotopy equivalence if it induces isomorphisms between homotopy groups, and that two Hilbert manifolds are weakly homotopically equivalent if and only if they are homeomorphic (see \cite{BH70}).

Observe that Theorem~\ref{thmss2} gives an answer to Question~\ref{mainquestion} for only one space; in particular it does not give any answer to Question~\ref{mainquestion0}. However, the above result will allow us to ask another interesting question. 

First let us remark that the Frenet frame injections are still defined for generic curves, and it follows from the work of Hirsch, Smale and Gromov (\cite{Sma59b}, \cite{Hir59} and \cite{Gro86}) that in this case, they always define homeomorphisms. Let us state this result properly.

\begin{theorem}\label{HSG}
For any $n \geq 2$, any $Q \in SO_{n+1}$ and any $z \in \mathrm{Spin}_{n+1}$, the Frenet frame injections defined by
 \[ \mathcal{F}_{[Q]} : \gamma \in \mathcal{G}\SS^n(Q) \mapsto \mathcal{F}_\gamma \in \Omega SO_{n+1}(Q) \simeq \Omega SO_{n+1} \]
and
\[ \tilde{\mathcal{F}}_{[z]} : \gamma \in \mathcal{G}\SS^n(z) \mapsto \tilde{\mathcal{F}}_\gamma \in \Omega \mathrm{Spin}_{n+1}(z) \simeq \Omega \mathrm{Spin}_{n+1}
 \]
are homotopy equivalences.
\end{theorem}

This theorem, combined with Theorem~\ref{thmss2}, immediately implies the following result.

\begin{theorem}\label{topo}
For any $n \geq 2$, any $Q \in SO_{n+1}$ and any $z \in \mathrm{Spin}_{n+1}$, the inclusions 
\[ \mathcal{L}\SS^n(Q) \subset \mathcal{G}\SS^n(Q), \quad \mathcal{L}\SS^n(z) \subset \mathcal{G}\SS^n(z) \]
induce surjective homomorphisms between homotopy and homology groups: for any $k \in \N$, the induced homomorphisms
\[ \pi_k(\mathcal{L}\SS^n(Q)) \rightarrow \pi_k(\mathcal{G}\SS^n(Q)), \quad H_k(\mathcal{L}\SS^n(Q),\R) \rightarrow H_k(\mathcal{G}\SS^n(Q), \R)  \]
and 
\[ \pi_k(\mathcal{L}\SS^n(z)) \rightarrow \pi_k(\mathcal{G}\SS^n(z)), \quad H_k(\mathcal{L}\SS^n(z),\R) \rightarrow H_k(\mathcal{G}\SS^n(z),\R)  \]
are surjective. Moreover, for $|s|\leq 1$, the inclusions
\[ \mathcal{L}\SS^n(M_{s}^{n+1}) \subset \mathcal{G}\SS^n(M_{s}^{n+1}), \quad \mathcal{L}\SS^n(w_{s}^{n+1}) \subset \mathcal{G}\SS^n(w_{s}^{n+1}) \]
are homotopy equivalences.
\end{theorem}

At this point, the following question, related to Question~\ref{mainquestion0}, is very natural.

\begin{question}\label{mainquestion1}
For any $n \geq 2$, any $Q \in SO_{n+1}$ and any $z \in \mathrm{Spin}_{n+1}$, are the inclusions 
\[ \mathcal{L}\SS^n(Q) \subset \mathcal{G}\SS^n(Q), \quad \mathcal{L}\SS^n(z) \subset \mathcal{G}\SS^n(z) \]
homotopy equivalence?
\end{question}

\section{The case $n=2$}\label{s42}

In this section and the next one, we will see what these results mean in the special case $n=2$ and $n=3$, recalling that $\mathrm{Spin}_{3}$ and $\mathrm{Spin}_{4}$ can be identified respectively to $\SS^3$ and $\SS^3 \times \SS^3$, where $\SS^3$ is considered as the unit sphere of $\H$. In what follows, we write $\mathbf{1},\mathbf{i},\mathbf{j},\mathbf{k}$ for the canonical basis of $\H$ viewed as vector space over $\R$.

Let's recall that the Theorem ~\ref{thmsha}, in the case $n=2$, has been known since Little's paper.
\begin{theorem}[Little,  \cite{Lit70}, with our notation]\label{thmlit}
The space $\mathcal{L}\SS^2(I)$ has $3$ connected components:
one connected component is $\mathcal{L}\SS^2(\mathbf{1})$ and the other two connected components are
$\mathcal{L}\SS^2(-\mathbf{1})_c$ and $\mathcal{L}\SS^2(-\mathbf{1})_n$, whose union is $\mathcal{L}\SS^2(-\mathbf{1}),$
where $\mathcal{L}\SS^2(-\mathbf{1})_c$ is the component associated to convex curves and $\mathcal{L}\SS^2(-\mathbf{1})_n$ the component associated to non-convex curves.
\end{theorem}

\begin{figure}[H]
\centering
\includegraphics[scale=0.5]{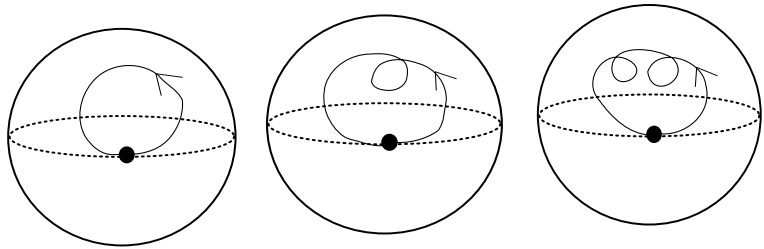}
\caption{Curves in $\mathcal{L}\SS^2(-\mathbf{1})_c, \; \mathcal{L}\SS^2(\mathbf{1})$ and $\mathcal{L}\SS^2(\mathbf{-1})_n$.}
\label{Little}
\end{figure}

As in the Theorem ~\ref{thmss1}, in the case $n=2$, let $s \in \Z$ such that $|s|\leq 3$ and $s\equiv 3$ (mod $4$). We have $s=-1$ or $s=3$ and
\[ M_{-1}^3=
\begin{pmatrix}
-1 & 0 & 0 \\
0 & -1 & 0 \\
0 & 0 & 1
\end{pmatrix}, \quad
 M_{3}^3=I=
\begin{pmatrix}
1 & 0 & 0 \\
0 & 1 & 0 \\
0 & 0 & 1
\end{pmatrix}.
 \]
Then $\pm w_{3}^3$ can be identified with $\pm \mathbf{1}$ and $\pm w_{-1}^3$ can be identified with $\pm \mathbf{i}$. Therefore, theorem~\ref{thmss1} implies that $\mathcal{L}\SS^2(Q)$ is homeomorphic to $\mathcal{L}\SS^2(M_{-1}^3)$ or $\mathcal{L}\SS^2(I)$, and $\mathcal{L}\SS^2(z)$ is homeomorphic to $\mathcal{L}\SS^2(\mathbf{1})$, $\mathcal{L}\SS^2(-\mathbf{1}),$ or $\mathcal{L}\SS^2(\mathbf{i})$. Theorem~\ref{thmss2} gives the topology of two of these spaces, namely
\[ \mathcal{L}\SS^2(M_{-1}^3) \simeq \Omega SO_3, \quad \mathcal{L}\SS^2(\mathbf{i}) \simeq \Omega \SS^3.  \]
The cohomology groups of theses spaces are well-known; for instance (see \cite{BT82})
\begin{equation}\label{homloop}
H^k(\Omega \SS^3,\R)=
\begin{cases}
0, & k \; \mathrm{odd}\\
\R, & k \; \mathrm{even}.
\end{cases}
\end{equation}
Clearly, we have
\[ \mathcal{L}\SS^2(I)=\mathcal{L}\SS^2(\mathbf{1}) \sqcup \mathcal{L}\SS^2(-\mathbf{1}) \]
and from Theorem~\ref{thmlit} the space $\mathcal{L}\SS^2(I)$ has $3$ connected components: $\mathcal{L}\SS^2(\mathbf{1})$, $\mathcal{L}\SS^2(-\mathbf{1})_c$,  and $\mathcal{L}\SS^2(-\mathbf{1})_n$. From Theorem~\ref{thmani}, the component $\mathcal{L}\SS^2(-\mathbf{1})_c$ is contractible. 

Let us go back to the Question~\ref{mainquestion0} in this special case $n=2$. The spaces $\mathcal{L}\SS^2(\mathbf{1})$ and $\mathcal{L}\SS^2(-\mathbf{1})$ are clearly not homeomorphic since they don't have the same number of connected components. Since the Frenet frame injection is always homotopically surjective, the cohomology of both spaces contains at least the cohomology of $\Omega \SS^3$. In~\cite{Sal09I} and \cite{Sal09II}, Saldanha proved the following result.

\begin{theorem}[Saldanha, \cite{Sal09I},\cite{Sal09II}]\label{thmsal0}
The spaces $\mathcal{L}\SS^2(\mathbf{1})$ and $\mathcal{L}\SS^2(-\mathbf{1})_n$ are simply connected, and we have, for any integer $k \geq 1$, the inequalities
\[ \mathrm{dim}\;H^{2k}(\mathcal{L}\SS^2((-1)^k\mathbf{1}), \R) \geq 1, \quad \mathrm{dim}\;H^{2k}(\mathcal{L}\SS^2((-1)^{k+1}\mathbf{1}), \R) \geq 2 \]
with equality for $k=1$.
\end{theorem}

Comparing this result with the cohomology groups of $\Omega \SS^3$, one obtains an answer to Question~\ref{mainquestion1}: the inclusions
\[ \mathcal{L}\SS^2(I) \subset \mathcal{G}\SS^2(I) \]
and
\[ \mathcal{L}\SS^2(\mathbf{1}) \subset \mathcal{G}\SS^2(\mathbf{1}), \quad \mathcal{L}\SS^2(-\mathbf{1}) \subset \mathcal{G}\SS^2(-\mathbf{1}) \]
are not homotopy equivalences.

This obviously gives also a partial answer to Question~\ref{mainquestion0}: $\mathcal{L}\SS^2(\mathbf{1})$ is not homeomorphic to $\mathcal{L}\SS^2(\mathbf{i}) \simeq \Omega \SS^3$, and $\mathcal{L}\SS^2(-\mathbf{1})$ is not homeomorphic to $\mathcal{L}\SS^2(\mathbf{i}) \simeq \Omega \SS^3$. From this, it follows also that the spaces $\mathcal{L}\SS^2(I)$ and $\mathcal{L}\SS^2(M_{-1}^3) \simeq \Omega SO_3$ are non-homeomorphic. 

However, this does not answer Question~\ref{mainquestion0}, since it does not prove that $\mathcal{L}\SS^2(\mathbf{1})$ is not homeomorphic to $\mathcal{L}\SS^2(-\mathbf{1})$. The following much stronger result was later proved by Saldanha. Recall that given two pointed topological space $X$ and $Y$, we denote by $X \vee Y$ the wedge product of $X$ and $Y$. The following theorem was proved in~\cite{Sal13}.

\begin{theorem}[Saldanha, \cite{Sal13}]\label{thmsal}
We have the following homotopy equivalences
\[ \mathcal{L}\SS^2(\mathbf{1}) \approx (\Omega \SS^3) \vee \SS^2 \vee \SS^6 \vee \SS^{10} \vee \cdots, \quad \mathcal{L}\SS^2(-\mathbf{1}) \approx (\Omega \SS^3) \vee \SS^0 \vee \SS^4 \vee \SS^8 \vee \cdots.  \]
\end{theorem}

In particular, this gives a complete answer to Question~\ref{mainquestion}, and a fortiori to Question~\ref{mainquestion0}. Since the cohomology groups of spheres are well-known, here's a straightforward corollary.

\begin{corollary}[Saldanha, \cite{Sal13}]\label{corsal}
For any even integer $k \geq 1$, we have
\begin{equation*}
H^k(\mathcal{L}\SS^2(\mathbf{1}),\R)=
\begin{cases}

\R^2, & 4|(k+2) \\
\R, & 4|k
\end{cases}
\end{equation*}
and
\begin{equation*}
H^k(\mathcal{L}\SS^2(-\mathbf{1}),\R)=H^k(\mathcal{L}\SS^2(-\mathbf{1})_n,\R)=
\begin{cases}
\R, & 4|(k+2) \\
\R^2, & 4|k.
\end{cases}
\end{equation*}
\end{corollary}

\section{The case $n=3$}\label{s43}

Let's recall that the $\mathrm{Spin}_{4}$ can be identified with $\SS^3 \times \SS^3$. From the definition \ref{LSnz}, given $z=(z_l,z_r) \in \SS^3 \times \SS^3$ we will denote $\mathcal{L}\SS^3(z_l,z_r)$ the space of locally convex curves in $\SS^3$ with the initial and final lifted Frenet frame $(\mathbf{1},\mathbf{1})$ and $(z_l,z_r),$ i.e.,
\[ \mathcal{L}\SS^3(z_l,z_r)=\{\gamma: [0,1] \rightarrow \SS^3 \; | \, \tilde{\mathcal{F}}_\gamma(0)=(\mathbf{1},\mathbf{1}) \; \text{and} \; \tilde{\mathcal{F}}_\gamma(1)=(z_l,z_r) \}. \]  

Let $s \in \Z$ such that $|s|\leq 4$ and $s\equiv 4$ (mod $4$). We have $s=-4$, $s=0$ or $s=4$ and
\[ M_{-4}^4=-I=
\begin{pmatrix}
-1 & 0 & 0 & 0\\
0 & -1 & 0 & 0 \\
0 & 0 & -1 & 0 \\
0 & 0 & 0 & -1
\end{pmatrix}, \quad
M_{0}^4=
\begin{pmatrix}
-1 & 0 & 0 & 0\\
0 & -1 & 0 & 0 \\
0 & 0 & 1 & 0 \\
0 & 0 & 0 & 1
\end{pmatrix}, 
\]
and
\[
 M_{4}^4=I=
\begin{pmatrix}
1 & 0 & 0 & 0\\
0 & 1 & 0 & 0 \\
0 & 0 & 1 & 0 \\
0 & 0 & 0 & 1
\end{pmatrix}.
 \]
Then $\pm w_{-4}^4$ can be identified with $\pm(\mathbf{1},-\mathbf{1})$, $\pm w_{0}^4$ can be identified with $\pm (\mathbf{i},-\mathbf{i})$ and $\pm w_{4}^4$ can be identified with $ \pm(\mathbf{1},\mathbf{1})$. Theorem~\ref{thmss1} implies that $\mathcal{L}\SS^3(Q)$ is homeomorphic to one of these three spaces 
\begin{equation}\label{3space}
\mathcal{L}\SS^3(-I), \quad \mathcal{L}\SS^3(M_{0}^4),  \quad  \mathcal{L}\SS^3(I)
\end{equation}
and $\mathcal{L}\SS^3(z)$ is homeomorphic to one of these five spaces:
\begin{equation}\label{5space}
\quad \mathcal{L}\SS^3(\mathbf{1},-\mathbf{1}), \quad \mathcal{L}\SS^3(-\mathbf{1},\mathbf{1}), \quad \mathcal{L}\SS^3(\mathbf{i},-\mathbf{i}), \quad  
\mathcal{L}\SS^3(\mathbf{1},\mathbf{1}), \quad \mathcal{L}\SS^3(-\mathbf{1},-\mathbf{1}).
\end{equation}
Again, Theorem~\ref{thmss2} gives the topology of two of these spaces, namely
\[ \mathcal{L}\SS^3(M_{0}^4) \simeq \Omega SO_4, \quad \mathcal{L}\SS^3(\mathbf{i},-\mathbf{i}) \simeq \Omega (\SS^3 \times \SS^3) \simeq \Omega \SS^3 \times \Omega\SS^3.  \]
As before, using~\eqref{homloop} and the K\"{u}nneth formula, one can determine the cohomology groups of these spaces
\begin{equation}\label{loopspace}
H^j(\Omega (\SS^3 \times \SS^3),\R)=H^j(\Omega \SS^3 \times \Omega\SS^3,\R)=
\begin{cases}
0, & j \; \mathrm{odd} \\
\R^{l+1}, & j=2l, \; l\in\N.
\end{cases}
\end{equation}

Questions~\ref{mainquestion}, ~\ref{mainquestion0} and~\ref{mainquestion1} in this case $n=3$ are open. In the next section we will state our main results, which will give in particular a partial answer to Questions~\ref{mainquestion0} and~\ref{mainquestion1}.

\section{Results of the thesis}\label{s44}

Let us now describe our main results. 

Our first result, which will be a crucial tool in this thesis, allows us to identify the space $\mathcal{G}\SS^3(z_l,z_r)$ with a certain subspace of the cartesian product $\mathcal{G}\SS^2(z_l)\times \mathcal{G}\SS^2(z_r)$. Here's a precise statement.  

\begin{Main}\label{th0}
There is a homeomorphism between the space of $\gamma \in \mathcal{G}\SS^3(z_l,z_r)$ and the space of pair of curves 
\[ (\gamma_l,\gamma_r) \in \mathcal{G}\SS^2(z_l)\times \mathcal{G}\SS^2(z_r) \]
satisfying the condition
\begin{equation}\label{cond0}\tag{$*$}
||\gamma_l'(t)||=||\gamma_r'(t)||, \quad \kappa_{\gamma_l}(t)>\kappa_{\gamma_r}(t), \quad t \in [0,1].
\end{equation}
Moreover, given $\gamma \in \mathcal{G}\SS^3(z_l,z_r)$, the curves $\gamma_l$ and $\gamma_r$ satisfy
\[ ||\gamma_l'(t)||=||\gamma_r'(t)||=||\gamma'(t)||\kappa_{\gamma}(t), \quad \kappa_{\gamma_l}(t)=\frac{\tau_\gamma(t)+1}{\kappa_\gamma(t)}, \quad \kappa_{\gamma_r}(t)=\frac{\tau_\gamma(t)-1}{\kappa_\gamma(t)} \]
and given a pair of curves $(\gamma_l,\gamma_r) \in \mathcal{G}\SS^2(z_l) \times \mathcal{G}\SS^2(z_r)$ satisfying~\eqref{cond0}, the curve $\gamma$ satisfies
\[ ||\gamma'(t)||=\frac{||\gamma_l'(t)||(\kappa_{\gamma_l}(t)-\kappa_{\gamma_r}(t))}{2}, \quad \kappa_\gamma(t)=\frac{2}{\kappa_{\gamma_l}(t)-\kappa_{\gamma_r}(t)},\] 
\[\tau_\gamma(t)=\frac{\kappa_{\gamma_l}(t)+\kappa_{\gamma_r}(t)}{\kappa_{\gamma_l}(t)-\kappa_{\gamma_r}(t)}. \] 
\end{Main}

In other words, every generic curve $\gamma \in \mathcal{G}\SS^3(z_l,z_r)$ can be decomposed as a pair of generic curves (immersions) $\gamma_l \in \mathcal{G}\SS^2(z_l)$ and $\gamma_r \in \mathcal{G}\SS^2(z_r)$.

In the special case where $\gamma \in \mathcal{L}\SS^3(z_l,z_r)$, more can be said. Indeed, we not only have $\kappa_{\gamma_l}(t)>\kappa_{\gamma_r}(t)$ but also $\kappa_{\gamma_l}(t)>-\kappa_{\gamma_r}(t)$: since $\gamma$ has positive torsion, 
\[ \kappa_{\gamma_l}(t)+\kappa_{\gamma_r}(t)=\frac{2\tau_\gamma(t)}{\kappa_\gamma(t)}>0. \]
Therefore $\kappa_{\gamma_l}(t)>|\kappa_{\gamma_r}(t)|\geq 0$, which means that $\gamma_l$ is locally convex, and so we obtain the following corollary.  

\begin{Main}\label{th1}
There is a homeomorphism between the space of $\gamma \in \mathcal{L}\SS^3(z_l,z_r)$ and the space of pairs of curves 
\[ (\gamma_l,\gamma_r) \in \mathcal{L}\SS^2(z_l)\times \mathcal{G}\SS^2(z_r) \]
satisfying the condition
\begin{equation}\label{cond}\tag{$**$}
||\gamma_l'(t)||=||\gamma_r'(t)||, \quad \kappa_{\gamma_l}(t)>|\kappa_{\gamma_r}(t)|, \quad t \in [0,1].
\end{equation}
Moreover, given $\gamma \in \mathcal{L}\SS^3(z_l,z_r)$, the curves $\gamma_l$ and $\gamma_r$ satisfy
\[ ||\gamma_l'(t)||=||\gamma_r'(t)||=||\gamma'(t)||\kappa_{\gamma}(t), \quad \kappa_{\gamma_l}(t)=\frac{\tau_\gamma(t)+1}{\kappa_\gamma(t)}, \quad \kappa_{\gamma_r}(t)=\frac{\tau_\gamma(t)-1}{\kappa_\gamma(t)} \]
and given a pair of curves $(\gamma_l,\gamma_r) \in \mathcal{L}\SS^2(z_l) \times \mathcal{G}\SS^2(z_r)$ satisfying~\eqref{cond}, the curve $\gamma$ satisfy
\[ ||\gamma'(t)||=\frac{||\gamma_l'(t)||(\kappa_{\gamma_l}(t)-\kappa_{\gamma_r}(t))}{2}, \quad \kappa_\gamma(t)=\frac{2}{\kappa_{\gamma_l}(t)-\kappa_{\gamma_r}(t)},\] 
\[\tau_\gamma(t)=\frac{\kappa_{\gamma_l}(t)+\kappa_{\gamma_r}(t)}{\kappa_{\gamma_l}(t)-\kappa_{\gamma_r}(t)}. \] 
\end{Main}

In other words, every locally convex curve $\gamma \in \mathcal{L}\SS^3(z_l,z_r)$ can be decomposed as a pair of curves $\gamma_l$ and $\gamma_r$, where $\gamma_l \in \mathcal{L}\SS^2(z_l)$ is locally convex and $\gamma_r \in \mathcal{G}\SS^2(z_r)$ is an immersion. A locally convex curve in $\SS^3$ is rather hard to understand from a geometrical point of view; Theorem~\ref{th1} allows us to see such a curve as a pair of curves in $\SS^2$, a situation where one can use geometrical intuition.

Theorem~\ref{th1}, which follows directly from Theorem~\ref{th0}, will allow us to give many examples of locally convex curves in $\SS^3$. We will see later (Proposition~\ref{leftconvex} in Chapter~\ref{chapter8}, \S\ref{s82}) that if $\gamma \in \mathcal{L}\SS^3(z_l,z_r)$ is such that its left part $\gamma_l \in \mathcal{L}\SS^2(z_l)$ is convex, then $\gamma$ is convex, but in general the converse is false. 

Yet there are still some spaces in which one can characterize completely the convexity of $\gamma$ by looking at its left part.    

The first space in which we have such a characterization is $\mathcal{L}\SS^3(-\1,\k)$.

\begin{Main}\label{th2}
A curve $\gamma \in \mathcal{L}\SS^3(-\1,\k)$ is convex if and only if its left part $\gamma_l \in \mathcal{L}\SS^2(-\1)$ is convex.
\end{Main}

The second space is $\mathcal{L}\SS^3(\1,-\1)$; in fact we will see later (Proposition~\ref{spaces} in Chapter~\ref{chapter7}, \S\ref{s73}) that $\mathcal{L}\SS^3(\1,-\1)$ is homeomorphic to $\mathcal{L}\SS^3(-\1,\k)$. However, for this space, we can only give a necessary condition for a curve to be convex, even though we believe that this condition is also sufficient.

\begin{Main}\label{th3}
Assume that $\gamma \in \mathcal{L}\SS^3(\1,-\1)$ is convex. Then its left part $\gamma_l \in \mathcal{L}\SS^2(\1)$ is contained in an open hemisphere and its rotation number is equal to $2$.
\end{Main}

We refer to Chapter~\ref{chapter6}, \S\ref{s64} for a precise definition of a curve contained in an open hemisphere and its rotation number.

\medskip

To state our next result, we need more definitions. Let us consider $\gamma \in \mathcal{L}\SS^n$, $\mathcal{F}_\gamma : [0,1] \rightarrow SO_{n+1}$ its Frenet frame curve and $\tilde{\mathcal{F}}_\gamma : [0,1] \rightarrow SO_{n+1}$ the lifted Frenet frame curve. Let us introduce the definition of convex and stably convex matrices $Q \in SO_{n+1}$ and spin $z \in \mathrm{Spin}_{n+1}$.

\begin{definition}
A matrix $Q \in SO_{n+1}$ (respectively a spin $z \in \mathrm{Spin}_{n+1}$) is called \emph{convex} if there exists a convex arc $\gamma : [t_0,t_1] \rightarrow \SS^n$, $0 \leq t_0 < t_1 \leq 1$ such that $(\mathcal{F}_\gamma(t_0))^{-1}\mathcal{F}_\gamma(t_1)=Q$ (respectively $(\tilde{\mathcal{F}}_\gamma(t_0))^{-1}\tilde{\mathcal{F}}_\gamma(t_1)=z$). 

If the convex arc $\gamma : [t_0,t_1] \rightarrow \SS^n$ can be extended to a convex arc $\gamma : [t_0,t_1+\varepsilon] \rightarrow \SS^n$ for some $\varepsilon>0$, then $Q \in SO_{n+1}$ (respectively $z \in \mathrm{Spin}_{n+1}$) is called \emph{stably convex}, and the convex arc is also called \emph{stably convex}.  
\end{definition}

Stably convex matrices clearly form an open set; in fact we will see that they correspond to one open Bruhat cell. Convex matrices which are not stably convex are the disjoint union by Bruhat cells of lower dimension, and they characterize Frenet frames at which convexity is lost. 

In general, convex matrices belong to the disjoint union of $(n+1)!$ Bruhat cells, and convex but not stably convex matrices belong to the disjoint union of $(n+1)!-1$ Bruhat cells of lower dimension. The same thing holds true for convex spins. For $n=2$, these convex Bruhat cells are easily determined (there are $6$ of them) and their knowledge plays an important role in understanding the topology of the space of locally convex curves in $\SS^2$. For $n=3$, using the geometric understanding we obtained from Theorem~\ref{th1}, we will have the following result. 

\begin{Main}\label{th4}
The explicit list of the $24$ convex matrices in $B_4^+ \subset SO_4$ and the $24$ convex spins in $\tilde{B}_4^+ \subset  \mathrm{Spin}_{4} \simeq \SS^3 \times \SS^3$ is given in~\S\ref{s72}.
\end{Main}

We will also determine explicitly the list of open Bruhat cells in which the curve immediately enters after loosing convexity, both in the case of matrices and spins. 

Even though we will not use this explicit list of convex matrices and convex spins, we believe that this is an important step towards an answer to Question~\ref{mainquestion} in the case $n=3$. 

\medskip

Finally, let us come back to our main question, that is Question~\ref{mainquestion}. Using Theorem~\ref{th1} and the work of Saldanha in the case $n=2$, we will prove the following result.  

\begin{Main}\label{th5}
For any even integer $j \geq 1$, we have
\[ \mathrm{dim} \; H^j(\mathcal{L}\SS^3(-\mathbf{1},\1),\R) \geq 1 + \mathrm{dim} \; H^j(\mathcal{G}\SS^3 (-\mathbf{1},\1),\R), \quad 4|(j+2),  \]
\[ \mathrm{dim} \; H^j(\mathcal{L}\SS^3(\mathbf{1},-\1),\R) \geq 1 + \mathrm{dim} \; H^j(\mathcal{G}\SS^3 (\mathbf{1},-\1),\R), \quad 4|j .\]

\end{Main}

Comparing this Theorem with~\eqref{loopspace} and recalling that
\[ \mathcal{L}\SS^3(\mathbf{i},-\mathbf{i}) \simeq \Omega (\SS^3 \times \SS^3) \simeq \Omega \SS^3 \times \Omega\SS^3, \]
we immediately obtain the following theorem, which gives a partial answer to Questions~\ref{mainquestion0} and~\ref{mainquestion1}.

\begin{Main}\label{th6}
The inclusions
\[ \mathcal{L}\SS^3(-I) \subset \mathcal{G}\SS^3(-I) \]
and
\[ \mathcal{L}\SS^3(-\mathbf{1},\1) \subset \mathcal{G}\SS^3(-\mathbf{1},\1), \quad \mathcal{L}\SS^3(\mathbf{1},-\1) \subset \mathcal{G}\SS^3(\mathbf{1},-\1) \]
are not homotopy equivalences. Therefore $\mathcal{L}\SS^3(-I)$ is not homeomorphic to $\mathcal{L}\SS^3(M_0^4)$, $\mathcal{L}\SS^3(-\mathbf{1},\1)$ is not homeomorphic to $\mathcal{L}\SS^3(\mathbf{i},-\mathbf{i})$ and $\mathcal{L}\SS^3(\mathbf{1},-\1)$ is not homeomorphic to $\mathcal{L}\SS^3(\mathbf{i},-\mathbf{i})$. 
\end{Main}


%
%

\chapter{Some operations on locally convex curves}
\label{chapter5}

In this Chapter, following~\cite{SS12} we describe several operations on the space of locally convex curves.

\section{Time reversal}\label{s51}

Given a locally convex curve $\gamma \in \mathcal{L}\SS^n$, we would like to be able to reverse time. A naive idea would be to consider the curve $t \mapsto \gamma(1-t)$: clearly such a curve could be negative locally convex and its initial Frenet frame is not necessarily the identity. So we will have to do something more elaborate.

Let $J_+=\mathrm{diag}(1,-1,1,-1, \dots) \in O_{n+1}$; its determinant is equal to $(-1)^{n(n+1)/2}$. We define a map
\[ \mathbf{TR} : SO_{n+1} \rightarrow SO_{n+1}, \quad Q \mapsto \mathbf{TR}(Q)=J_+{} Q ^\top J_+ \]
This is a well-defined map which is an anti-automorphism: that is $\mathbf{TR}(QQ')=\mathbf{TR}(Q')\mathbf{TR}(Q)$ for any $Q \in SO_{n+1}$ and $Q' \in SO_{n+1}$. One can easily check that this anti-automorphism lifts to an anti-automorphism 
\[ \mathbf{TR} : \mathrm{Spin}_{n+1} \rightarrow \mathrm{Spin}_{n+1}. \]
These two maps preserves the subgroups $\mathrm{Diag}_{n+1}^+ \subset B_{n+1}^+ \subset SO_{n+1}$ and $\widetilde{\mathrm{Diag}}_{n+1}^+ \subset \tilde{B}_{n+1}^+ \subset \mathrm{Spin}_{n+1}$. The action of $\mathbf{TR}$ on $B_{n+1}^+$ is simple: $\mathbf{TR}(Q)$ is obtained from $Q \in B_{n+1}^+$ by transposition and change of sign of all entries for which $i+j$ is odd.

We can finally define the operation of time reversal.

\begin{definition}\label{deftime} 
Given $\gamma \in \mathcal{L}\SS^n(Q)$, we define its \emph{time reversal} $\gamma^{\mathbf{TR}}$ by
\[ \gamma^{\mathbf{TR}}(t)=J_+ {} Q^\top \gamma(1-t). \]
\end{definition}

With this definition, the following proposition was proved in~\cite{SS12}.

\begin{proposition}\label{proptime} 
For any $\gamma \in \mathcal{L}\SS^n(Q)$, we have $\gamma^{\mathbf{TR}} \in \mathcal{L}\SS^n(\mathbf{TR}(Q))$. Moreover,
\[ \mathcal{F}_{\gamma^{\mathbf{TR}}}(t)=J_+ {} Q ^\top \mathcal{F}_\gamma(1-t)J_+, \quad \Lambda_{\gamma^{\mathbf{TR}}}(t)=\Lambda_\gamma(1-t) \]
and time reversal yields homeomorphisms
\[ \mathcal{L}\SS^n(Q) \approx \mathcal{L}\SS^n(\mathbf{TR}(Q)), \quad \mathcal{L}\SS^n(z) \approx \mathcal{L}\SS^n(\mathbf{TR}(z)) \]
for any $Q \in SO_{n+1}$ and any $z \in \mathrm{Spin}_{n+1}$.
\end{proposition}

\section{Arnold duality}\label{s52}

Let us start by defining the \emph{Arnold matrix} $A \in B_{n+1}^+$: it is the anti-diagonal matrix with entries $(A)_{i,n+2-i}=(-1)^{i+1}$. In terms of the permutation matrix $P_\rho$ we defined in Chapter~\ref{chapter2}, section~\S\ref{s23} and the matrix $J_+ \in O_{n+1}$ defined in the previous section~\S\ref{s51}, we have
\[ A=J_+P_\rho. \] 
For $n=2$ and $n=3$, we have, respectively,
\[
A=
\begin{pmatrix}
0 & 0 & 1 \\
0 & -1 & 0 \\
1 & 0 & 0 
\end{pmatrix},
\quad
A=
\begin{pmatrix}
0 & 0 & 0 & 1 \\
0 & 0 & -1 & 0 \\
0 & 1 & 0 & 0 \\
-1 & 0 & 0 & 0 
\end{pmatrix}.
\]
Observe that $A ^\top =A$ if and only if $n+1$ is odd. We can define a map
\[ \mathbf{AD} : SO_{n+1}  \rightarrow SO_{n+1}, \quad Q  \mapsto \mathbf{AD}(Q)={} A ^\top Q A  \]
which is an automorphism that preserves $B_{n+1}^+$. The action of $B_{n+1}^+$ is also clear: $\mathbf{AD}(Q)$ is obtained from $Q$ by a "half-turn" (the $(i,j)$-entry goes to the $(n+2-i,n+2-j)$ entry) and then change the signs of all entries with $i+j$ odd.

This automorphism lifts to an automorphism 
\[ \mathbf{AD} : z \in \mathrm{Spin}_{n+1} \mapsto \mathbf{AD}(z) \in \mathrm{Spin}_{n+1} \]
that preserves $\tilde{B}_{n+1}^+$.

We can now define the operation of Arnold duality.

\begin{definition}\label{defarnold} 
Given $\gamma \in \mathcal{L}\SS^n(Q)$, we define its \emph{Arnold dual} $\gamma^{\mathbf{AD}}$ by
\[ \gamma^{\mathbf{AD}}(t)=\mathbf{AD}(\mathcal{F}_\gamma(t))e_1. \]
\end{definition}

As before, the following proposition was proved in~\cite{SS12}.

\begin{proposition}\label{proparnold} 
For any $\gamma \in \mathcal{L}\SS^n(Q)$, we have $\gamma^{\mathbf{AD}} \in \mathcal{L}\SS^n(\mathbf{AD}(Q))$. Moreover,
\[ \mathcal{F}_{\gamma^{\mathbf{AD}}}(t)=\mathbf{AD}(\mathcal{F}_\gamma(t)), \quad \Lambda_{\gamma^{\mathbf{AD}}}(t)=\mathbf{AD}(\Lambda_\gamma(t)) \]
and Arnold duality yields homeomorphisms
\[ \mathcal{L}\SS^n(Q) \approx \mathcal{L}\SS^n(\mathbf{AD}(Q)), \quad \mathcal{L}\SS^n(z) \approx \mathcal{L}\SS^n(\mathbf{AD}(z)) \]
for any $Q \in SO_{n+1}$ and any $z \in \mathrm{Spin}_{n+1}$.
\end{proposition}

\section{Chopping operations}\label{s53}

The last operation we want to define is very different from the two previous one. Roughly speaking, given a locally convex curve $\gamma$, we would like to chop off a small tip at the end (respectively at the beginning) of the curve; these two related operations will be called respectively \emph{negative chopping} and \emph{positive chopping}. The chopping operation described in~\cite{SS12} corresponds to what we call here negative chopping; we found it convenient to also introduce the positive chopping operation. 

Let us start with the naive geometric description of the negative chopping. For $\gamma \in \mathcal{L}\SS^n$ and a small $\varepsilon>0$, we define 
\[ \mathbf{chop_\varepsilon^-}(\gamma)(t)=\gamma((1-\varepsilon)t). \]
The image of $\mathbf{chop_\varepsilon^-}(\gamma)$ is nothing but the image of $\gamma$, from which we have removed the small piece $\gamma((1-\varepsilon,1])$. A simple computation gives
\[ \mathcal{F}_{\mathbf{chop_\varepsilon^-}(\gamma)}(t)=\mathcal{F}_\gamma((1-\varepsilon)t), \quad \Lambda_{\mathbf{chop_\varepsilon^-}(\gamma)}(t)=(1-\varepsilon)\Lambda_\gamma((1-\varepsilon)t)  \]
so, in particular, we have 
\[ \mathcal{F}_{\mathbf{chop_\varepsilon^-}(\gamma)}(1)=\mathcal{F}_\gamma(1-\varepsilon). \]
The obvious problem here is that if $\varepsilon$ is fixed and $\gamma$ varies arbitrarily, we have no control on the final Frenet frame $\mathcal{F}_{\mathbf{chop_\varepsilon^-}(\gamma)}(1)$ of the curve $\mathbf{chop_\varepsilon^-}(\gamma)$.

To fix this, we will not consider individual final Frenet frame $\mathcal{F}_{\mathbf{chop_\varepsilon^-}(\gamma)}(1)$ but look at their Bruhat cells instead.

To explain this, let us start with some algebraic and combinatorics notions. Given a signed permutation matrix $Q \in B_{n+1}^+$, let us denote $\mathbf{NE}(Q,i,j)$ the number of non-zero entries in the northeast quadrant of $Q$. Formally, given $(i,j)$ such that $(Q)_{(i,j)} \neq 0$, $\mathbf{NE}(Q,i,j)$ is the number of pairs $(i',j')$ with $i'<i$ and $j'>j$ such that $(Q)_{(i',j')} \neq 0$. The number of non-zero entries in the southwest quadrant of $Q$ can be defined by
\[ \mathbf{SW}(Q,i,j)=\mathbf{NE}(Q ^\top,j,i) \]
and it is easy to check that 
\[ \mathbf{NE}(Q,i,j)-\mathbf{SW}(Q,i,j)=i-j. \]
Let us further define
\[ \delta_i^-(Q)=(Q)_{(i,j)}(-1)^{\mathbf{NE}(Q,i,j)} \]
where $j$ is the only index such that $(Q)_{(i,j)} \neq 0$, and the diagonal matrix
\[ \Delta^-(Q)=\mathrm{diag}(\delta_1^-(Q),\delta_2^-(Q), \dots, \delta_n^-(Q)). \]
An elementary computation shows that
\[ \mathrm{det}(\Delta^-(Q))=\mathrm{det}(Q). \]
We have just defined a map
\[ \Delta^- : B_{n+1}^+ \rightarrow \mathrm{Diag}_{n+1}^+ \]     
that we can extend to a map
\[ \Delta^- : SO_{n+1} \rightarrow \mathrm{Diag}_{n+1}^+ \]
by setting $\Delta^-(Q)=\Delta^-(Q')$ whenever $Q$ and $Q'$ are Bruhat equivalent. Since $\Delta^-(\Delta^-(Q))=Q$, we obtain a partition of $SO_{n+1}$ into $2^n$ disjoint classes
\[ SO_{n+1}=\bigsqcup_{Q \in \mathrm{Diag}_{n+1}^+}(\Delta^-)^{-1}(Q). \]
For $Q \in \mathrm{Diag}_{n+1}^+$, one can check that $\Delta^-(QQ')=Q\Delta^-(Q')$ so that $(\Delta^-)^{-1}(Q)$ is a fundamental domain for the action of $\mathrm{Diag}_{n+1}^+$ on $SO_{n+1}$ by multiplication on the left.   

Let $A \in B_{n+1}^+$ be the Arnold matrix which was defined in~\S\ref{s52}. Since $\Delta^-(A)=I$ is the identity matrix, we have $\Delta^-(QA)=Q$ for any $Q \in \mathrm{Diag}_{n+1}^+$.

\begin{definition}\label{chop}
Given $Q \in SO_{n+1}$, we define its \emph{negative chopping} $\mathbf{chop}^-(Q) \in B_{n+1}^+$ by
\[ \mathbf{chop^-}(Q)=\Delta^-(Q)A. \]
\end{definition}

So for any $Q \in SO_{n+1}$, the Bruhat cell of $\mathbf{chop^-}(Q)$ is an open set (that is, it has maximal dimension), which is dense in $(\Delta^-)^{-1}(\Delta^-(Q))=(\mathbf{chop}^-)^{-1}(\mathbf{chop}^-(Q))$.

The link between the (naive) geometric chopping operation and this algebraic operation is contained in the following proposition, proved in~\cite{SS12}.

\begin{proposition}\label{propchop}
For any $Q \in SO_{n+1}$ and any $\gamma \in \mathcal{L}\SS^n(Q)$, there exists $\tilde{\varepsilon}>0$ such that for all $t \in (1-\tilde{\varepsilon},1)$, $\mathcal{F}_\gamma(t)\in \mathrm{Bru}_{\mathbf{chop}^-(Q)}$.
In other words, for any $\gamma \in \mathcal{L}\SS^n$ and all $\varepsilon \in (0,\tilde{\varepsilon})$, $\mathcal{F}_{\mathbf{chop_\varepsilon}^-(\gamma)}(1)$ is Bruhat equivalent to $\mathbf{chop}^-(\mathcal{F}_\gamma(1))$.  
\end{proposition}

This proposition allows us to compute quite easily the signed permutation matrix $\mathbf{chop}^-(Q)$. For instance, consider the matrix $Q \in SO_4$ defined by 
\[  
Q=
\begin{pmatrix}
0 & -1 & 0 & 0 \\
0 & 0 & 0 & -1 \\
-1 & 0 & 0 & 0 \\
0 & 0 & 1 & 0
\end{pmatrix}
\]
and take an arbitrary $\gamma \in \mathcal{L}\SS^n(Q)$. By a Taylor expansion, for a small negative $h=t-1$ (that is for $t<1$ close to $1$), we have that $\mathcal{F}_\gamma(h)$ is close to
\[  
\begin{pmatrix}
-h & -1 & 0 & 0 \\
-h^3/6 & -h^2/2 & -h & -1 \\
-1 & 0 & 0 & 0 \\
h^2/2 & h & 1 & 0
\end{pmatrix}.
\]
Using the algorithm described in Chapter~\ref{chapter2}, \S\ref{s23}, the signed permutation matrix $\mathbf{chop}^-(Q)$ can be found from the approximation of $\mathcal{F}_\gamma(h)$ as follows: all the signs of the determinants of the $k\times k$ south-west blocks of the approximation of $\mathcal{F}_\gamma(h)$, for $1 \leq k \leq 4$, should coincide with the signs of the determinants of the south-west blocks of $\mathbf{chop}^-(Q)$. The first south-west block is nothing but $h^2/2$, and it has positive sign, so the $(4,1)$-entry of $\mathbf{chop}^-(Q)$ is equal to one. The second south-west block is  
 \[  
\begin{pmatrix}
-1 & 0  \\
h^2/2 & h 
\end{pmatrix}
\]
and its determinant is $-h$, hence the sign is positive (since $h$ is negative), therefore the $(3,2)$-entry of $\mathbf{chop}^-(Q)$ should be equal to minus one (so that the determinant of the second south-west block of $\mathbf{chop}^-(Q)$ is equal to one). The third south-west block is   
\[  
\begin{pmatrix}

-h^3/6 & -h^2/2 & -h  \\
-1 & 0 & 0  \\
h^2/2 & h & 1 
\end{pmatrix}
\]
and its determinant is equal to $h^2/2$, hence it is positive and so the $(2,3)$-entry of $\mathbf{chop}^-(Q)$ should be equal to one. Finally, since $\mathbf{chop}^-(Q)$ should have determinant one, the $(4,1)$-entry of $\mathbf{chop}^-(Q)$ should be equal to minus one, and we have found that   
\[  
\mathbf{chop}^-(Q)=
\begin{pmatrix}
0 & 0 & 0 & -1 \\
0 & 0 & 1 & 0 \\
0 & -1 & 0 & 0 \\
1 & 0 & 0 & 0
\end{pmatrix}.
\]

\medskip

The role of the negative chopping operation is now clear. Consider a smooth Jacobian curve
\[ \Gamma : (-\varepsilon,\varepsilon) \rightarrow SO_{n+1}. \]
In view of Proposition~\ref{propjacobian} (Chapter~\ref{chapter3}, \S\ref{s35}) this is the same thing as looking at the Frenet frame curve of a locally convex curve, except that the interval of the definition is not necessarily $[0,1]$ and we do not require the initial Frenet frame to be the identity.

If $\Gamma(0)$ belongs to a open (that is, top-dimensional) Bruhat cell, then for $\varepsilon$ sufficiently small, $\Gamma(t)$ belongs to the same open cell. But $\Gamma(0)$ belongs to a lower-dimensional cell, then $\mathbf{chop}^-(\Gamma(0))$ will tell us in which open Bruhat cell $\Gamma(t)$, for small $t<0$, belongs to. One would like to know also in which open Bruhat cell $\Gamma(t)$, for small $t>0$, will belong to, and this is why we introduce the definition of positive chopping.  

As before, let us start with the naive geometric description of the positive chopping. This time, we consider $\gamma : [0,1] \rightarrow SO_{n+1}$ and we assume that $\mathcal{F}_\gamma(0)=Q \in SO_{n+1}$ and $\mathcal{F}_\gamma(1)=I \in SO_{n+1}$; let us denote by $\mathcal{L}\SS^n(Q;I)$ this space of curves. Given a small $\varepsilon>0$, we define 
\[ \mathbf{chop_\varepsilon^+}(\gamma)(t)=\gamma((1-\varepsilon)t+\varepsilon). \]
The image of $\mathbf{chop_\varepsilon^+}(\gamma)$ is nothing but the image of $\gamma$, from which we have removed the small piece $\gamma([0,\varepsilon))$. A simple computation gives
\[ \mathcal{F}_{\mathbf{chop_\varepsilon^+}(\gamma)}(t)=\mathcal{F}_\gamma((1-\varepsilon)t+\varepsilon), \quad \Lambda_{\mathbf{chop_\varepsilon^+}(\gamma)}(t)=(1-\varepsilon)\Lambda_\gamma((1-\varepsilon)t+\varepsilon)  \]
so, in particular, we have 
\[ \mathcal{F}_{\mathbf{chop_\varepsilon^+}(\gamma)}(0)=\mathcal{F}_\gamma(\varepsilon). \]
From the algebraic and combinatorial point of view, the positive chopping operation is similar to the negative one. Given a signed permutation matrix $Q \in B_{n+1}^+$, we define
\[ \delta_i^+(Q)=(Q)_{(i,j)}(-1)^{\mathbf{SW}(Q,i,j)} \]
where $j$ is the only index such that $(Q)_{(i,j)} \neq 0$, and the diagonal matrix
\[ \Delta^+(Q)=\mathrm{diag}(\delta_1^+(Q),\delta_2^+(Q), \dots, \delta_n^+(Q)). \]
As before, we have
\[ \mathrm{det}(\Delta^+(Q))=\mathrm{det}(Q). \]
This defines a map
\[ \Delta^+ : B_{n+1}^+ \rightarrow \mathrm{Diag}_{n+1}^+ \]     
that we can extend to a map
\[ \Delta^+ : SO_{n+1} \rightarrow \mathrm{Diag}_{n+1}^+ \]
by setting $\Delta^+(Q)=\Delta^+(Q')$ whenever $Q$ and $Q'$ are Bruhat equivalent. Since $\Delta^+(\Delta^+(Q))=Q$, we obtain a partition of $SO_{n+1}$ into $2^n$ disjoint classes
\[ SO_{n+1}=\bigsqcup_{Q \in \mathrm{Diag}_{n+1}^+}(\Delta^+)^{-1}(Q). \]
For $Q \in \mathrm{Diag}_{n+1}^+$, one still have $\Delta^+(QQ')=Q\Delta^+(Q')$.   

Let $A \in B_{n+1}^+$ be the Arnold matrix and $A ^\top \in B_{n+1}^+$ its transpose. Since $\Delta^+(A ^\top)=I$ is the identity matrix, we have $\Delta^+(A ^\top Q)=Q$ for any $Q \in \mathrm{Diag}_{n+1}^+$.

\begin{definition}\label{chopplus}
Given $Q \in SO_{n+1}$, we define its \emph{positive chopping} $\mathbf{chop}^+(Q) \in B_{n+1}^+$ by
\[ \mathbf{chop^+}(Q)=\Delta^+(Q) A ^\top. \]
\end{definition}

As before, for any $Q \in SO_{n+1}$, the Bruhat cell of $\mathbf{chop^+}(Q)$ is an open set which is dense in $(\Delta^+)^{-1}(\Delta^+(Q))=(\mathbf{chop}^+)^{-1}(\mathbf{chop}^+(Q))$.

Then we have a statement similar to Proposition~\ref{propchop}.

\begin{proposition}\label{propchop2}
For any $Q \in SO_{n+1}$ and any $\gamma \in \mathcal{L}\SS^n(Q;I)$, there exists $\tilde{\varepsilon}>0$ such that for all $t \in (0,\tilde{\varepsilon})$, $\mathcal{F}_\gamma(t)\in \mathrm{Bru}_{\mathbf{chop}^+(Q)}$.
In other words, for any locally convex curve $\gamma : [0,1] \rightarrow \SS^n$ and all $\varepsilon \in (0,\tilde{\varepsilon})$, $\mathcal{F}_{\mathbf{chop_\varepsilon}^+(\gamma)}(0)$ is Bruhat equivalent to $\mathbf{chop}^+(\mathcal{F}_\gamma(0))$.  
\end{proposition}

One can compute the signed matrix $\mathbf{chop}^+(Q)$ in the same way as we did for $\mathbf{chop}^-(Q)$. Indeed, consider again $Q \in SO_4$ defined by 
\[  
Q=
\begin{pmatrix}
0 & -1 & 0 & 0 \\
0 & 0 & 0 & -1 \\
-1 & 0 & 0 & 0 \\
0 & 0 & 1 & 0
\end{pmatrix}
\]
and take an arbitrary $\gamma \in \mathcal{L}\SS^n(Q;I)$. By a Taylor expansion, for a small positive $h$, we have that $\mathcal{F}_\gamma(h)$ is close to
\[  
\begin{pmatrix}
-h & -1 & 0 & 0 \\
-h^3/6 & -h^2/2 & -h & -1 \\
-1 & 0 & 0 & 0 \\
h^2/2 & h & 1 & 0
\end{pmatrix}.
\]
Proceeding exactly as we did to compute $\mathbf{chop}^-(Q)$ but recalling that now $h$ is positive, one finds 
\[  
\mathbf{chop}^+(Q)=
\begin{pmatrix}
0 & 0 & 0 & -1 \\
0 & 0 & -1 & 0 \\
0 & 1 & 0 & 0 \\
1 & 0 & 0 & 0
\end{pmatrix}.
\]

\medskip

In fact, this positive chopping operation can be recovered from the negative chopping operation and a modified time reversal operation. The time reversal operation introduced in~\ref{s51} is not convenient here since it fixes the initial Frenet frame. For $Q \in SO_{n+1}$ and $\gamma \in \mathcal{L}\SS^n(Q;I)$, let us define
\[ \mathbf{TR^*}(Q)=J_+QJ_+; \quad \gamma^{\mathbf{TR^*}}(t)=J_+\gamma(1-t). \] 
We have
\[ \mathcal{F}_{\gamma^{\mathbf{TR^*}}}(t)=J_+\mathcal{F}_\gamma(t)J_+=\mathbf{TR^*}(\mathcal{F}_\gamma(t)) \]
so $\gamma^{\mathbf{TR^*}} \in \mathcal{L}\SS^n(\mathbf{TR^*}(Q))$. Let us apply the naive negative chopping operation on $\gamma^{\mathbf{TR^*}}$ to obtain the curve $\mathbf{chop_\varepsilon}^-(\gamma^{\mathbf{TR^*}})$: by Proposition~\ref{propchop}, the final Frenet frame of this curve is Bruhat equivalent to $\mathbf{chop}^-(\mathbf{TR^*}(Q))$. Let us now apply again $\mathbf{TR^*}$ again to the curve $\mathbf{chop_\varepsilon}^-(\gamma^{\mathbf{TR^*}})$: we obtain a curve $(\mathbf{chop_\varepsilon}^-(\gamma^{\mathbf{TR^*}}))^{\mathbf{TR^*}}$ whose initial Frenet frame is Bruhat equivalent to $\mathbf{TR^*}(\mathbf{chop}^-(\mathbf{TR^*}(Q)))$. Geometrically, it is clear that 
\[ \mathcal{F}_{\mathbf{chop_\varepsilon^+}(\gamma)}(0)=\mathbf{TR^*}(\mathcal{F}_{\mathbf{chop_\varepsilon^-}(\gamma^{\mathbf{TR^*}})}(1)).\]
Algebraically, this amounts to the following equality:
\begin{eqnarray*}
\mathbf{TR^*}(\mathbf{chop}^-(\mathbf{TR^*}(Q))) & = & \mathbf{TR^*}(\mathbf{chop}^-(J_+QJ_+)) \\
& = & \mathbf{TR}(\Delta^-(J_+QJ_+)A) \\
& = & J_+\Delta^-(J_+QJ_+)AJ_+ \\
& = & \Delta^-(J_+QJ_+)J_+AJ_+ \\
& = & \Delta^-(J_+QJ_+) A ^\top \\
& = & \Delta^+(Q) A ^\top \\
& = & \mathbf{chop}^+(Q).
\end{eqnarray*}
Indeed, in the fourth equality we simply used the fact $J_+$ and $\Delta^-(J_+QJ_+)$ commutes since they are both diagonal, in the fifth equality we used the fact $J_+AJ_+={} A ^\top$, and in the sixth equality we used the relation
\[ \Delta^+(Q)=\Delta^-(J_+QJ_+) \]
that can be easily checked.

\medskip

We would like now to lift these constructions to the spin group. However, unlike the time reversal (which defined an anti-automorphism) and the Arnold duality (which defined an automorphism), the maps $\Delta^+$ and $\Delta^-$ are not group homomorphisms and hence can not be lifted directly to the spin group. Yet we can still use the geometric characterizations contained in Proposition~\ref{propchop} and Proposition~\ref{propchop2} to define negative and positive chopping for spins.  

\begin{definition}\label{defchop}
For any $z \in \mathrm{Spin}_{n+1}$ and any $\gamma \in \mathcal{L}\SS^n(z)$, we define $\mathbf{chop}^-(z)$ as the unique element in $\tilde{B}_{n+1}^+$ for which there exists $\tilde{\varepsilon}>0$ such that for all $t \in (1-\tilde{\varepsilon},1)$, $\tilde{\mathcal{F}}_\gamma(t)\in \mathrm{Bru}_{\mathbf{chop}^-(z)}$.  
\end{definition}

The fact that this is well-defined (that is, it is independent of the choice of the curve in $\mathcal{L}\SS^n(z)$) follows directly from Proposition~\ref{propchop}. 

Let us define $\mathcal{L}\SS^n(z;\1)$ as the set of locally convex curves with initial lifted Frenet frame $\tilde{\mathcal{F}}_\gamma(0)=z$ and final lifted Frenet frame $\tilde{\mathcal{F}}_\gamma(1)=\1$. 

\begin{definition}\label{defchop2}
For any $z \in \mathrm{Spin}_{n+1}$ and any $\gamma \in \mathcal{L}\SS^n(z;\1)$, we define $\mathbf{chop}^+(z)$ as the unique element in $\tilde{B}_{n+1}^+$ for which there exists $\tilde{\varepsilon}>0$ such that for all $t \in (0,\tilde{\varepsilon})$, $\tilde{\mathcal{F}}_\gamma(t)\in \mathrm{Bru}_{\mathbf{chop}^+(z)}$.  
\end{definition} 

As before, this is well-defined in view of Proposition~\ref{propchop2}. Hence we have well-defined maps
\[ \mathbf{chop}^{\pm} : \mathrm{Spin}_{n+1} \rightarrow \tilde{B}_{n+1}^+ \]
and hence we can define
\[ \a=\mathbf{chop}^{-}(\1) \in \mathrm{Spin}_{n+1}, \quad \bar{\a}=\mathbf{chop}^{+}(\1) \in \mathrm{Spin}_{n+1}. \]
In particular, $\a$ projects down to $A$, and $\bar{\a}$ projects down to $A ^\top$; the other two elements that projects down to $A$ and $A ^\top$ will be denoted respectively by $-\a$ and $-\bar{\a}$. We may then define the maps
\[ \Delta^\pm : \mathrm{Spin}_{n+1} \rightarrow \widetilde{\mathrm{Diag}}^+_{n+1} \]
for $z \in \mathrm{Spin}_{n+1}$ by the relations
\[ \mathbf{chop}^-(z)=\Delta^-(z)\a, \quad \mathbf{chop}^+(z)={} \bar{\a} \Delta^+(z).  \]

To conclude, let us mention the following important result proved in~\cite{SS12}.

\begin{proposition}\label{chophomeo}
For any $Q \in SO_{n+1}$ and any $z \in \mathrm{Spin}_{n+1},$ we have homeomorphisms
\[ \mathcal{L}\SS^n(Q) \approx \mathcal{L}\SS^n(\mathbf{chop}^-(Q)) \approx \mathcal{L}\SS^n(\Delta^-(Q))  \]
and 
\[ \mathcal{L}\SS^n(z) \approx \mathcal{L}\SS^n(\mathbf{chop}^-(z)) \approx \mathcal{L}\SS^n(\Delta^-(z)).  \]
\end{proposition}

We can also state an analogous statement for the positive chopping operation, using the spaces $\mathcal{L}\SS^n(Q;I)$ and $\mathcal{L}\SS^n(z;\mathbf{1})$.

\begin{proposition}\label{chophomeo2}
For any $Q \in SO_{n+1}$ and any $z \in \mathrm{Spin}_{n+1}$, we have homeomorphisms
\[ \mathcal{L}\SS^n(Q;I) \approx \mathcal{L}\SS^n(\mathbf{chop}^+(Q);I) \approx \mathcal{L}\SS^n(\Delta^-(Q);I)  \]
and 
\[ \mathcal{L}\SS^n(z;\mathbf{1}) \approx \mathcal{L}\SS^n(\mathbf{chop}^-(z);\mathbf{1}) \approx \mathcal{L}\SS^n(\Delta^-(z);\mathbf{1}).  \]
\end{proposition}


%
%

\chapter{Decomposition of generic curves in $\SS^3$}

\label{chapter6}

The goal of this chapter is to prove Theorem~\ref{th0}, which states that a generic curve in $\SS^3$ can be decomposed as a pair of immersions in $\SS^2$. When restricted to locally convex curves, this gives Theorem~\ref{th1} which states that a locally convex curve in $\SS^3$ can be decomposed as a pair of curves in $\SS^2$, one of which is locally convex and the other which is an immersion. This theorem will be proved in~\S\ref{s61}, and in~\S\ref{s62} we will give many examples illustrating this general phenomenon for locally convex curves. In~\S\ref{s63} and~\S\ref{s64}, we will use this theorem, as well as the examples illustrating this theorem, to prove Theorem~\ref{th2} and Theorem~\ref{th3} that respectively characterizes convexity in the space $\mathcal{L}\SS^3(-\1,\k)$ and gives a necessary condition for convexity in the space $\mathcal{L}\SS^3(\1,-\1)$.

\section{Proof of Theorem~\ref{th0} and Theorem~\ref{th1}}\label{s61}

Consider $\gamma \in \mathcal{G}\SS^3$ and its associated Frenet and lifted Frenet frame curve
\[ \mathcal{F}_\gamma : [0,1] \rightarrow SO_4, \quad \tilde{\mathcal{F}}_\gamma : [0,1] \rightarrow \SS^3 \times \SS^3. \]
These are respectively quasi-Jacobian and quasi-holonomic curves, and we recall that any quasi-Jacobian and quasi-holonomic curves are of this form. Hence characterizing generic curves in $\SS^3$ is the same as characterizing quasi-holonomic curves
\[ \tilde{\Gamma} : [0,1] \rightarrow \SS^3 \times \SS^3. \]
Recall that the Lie algebra of $\SS^3$, viewed as the group of unit quaternions, is the vector space of imaginary quaternions
\[ \mathrm{Im}\mathbb{H}:=\{b\i+c\j+d\k \; | \; (b,c,d) \in \R^3\} \]
and hence the Lie algebra of $\SS^3 \times \SS^3$ is the product $\mathrm{Im}\mathbb{H} \times \mathrm{Im}\mathbb{H}$. The logarithmic derivative of $\tilde{\Gamma}$ belongs to the Lie algebra of $\SS^3 \times \SS^3$, that is
\[ \Lambda_{\tilde{\Gamma}}(t)=\tilde{\Gamma}(t)^{-1}\tilde{\Gamma}'(t) \in \mathrm{Im}\mathbb{H} \times \mathrm{Im}\mathbb{H}, \quad t \in [0,1].  \]   
In the proposition below, we characterize the subset of $\mathrm{Im}\mathbb{H} \times \mathrm{Im}\mathbb{H}$ to which the logarithmic derivative of a quasi-holonomic curve belongs. Let \[ \tilde{\mathfrak{Q}}:=\{(b_l\i+d\k,b_r\i+d\k) \in \mathrm{Im}\mathbb{H} \times \mathrm{Im}\mathbb{H} \; | \; (b_l,b_r,d) \in \R^3, \; b_l>b_r, \; d>0 \}. \]

\begin{proposition}\label{propth1}
Let $\tilde{\Gamma} : [0,1] \rightarrow \SS^3 \times \SS^3$ be a smooth curve with $\tilde{\Gamma}(0)=(\1,\1)$. Then $\tilde{\Gamma}$ is quasi-holonomic if and only its logarithmic derivative 
\[ \Lambda_{\tilde{\Gamma}}(t) \in \tilde{\mathfrak{Q}}, \quad t \in [0,1]. \]
Moreover, if 
\[ \Lambda_{\tilde{\Gamma}}(t)=(b_l(t)\i+d(t)\k,b_r(t)\i+d(t)\k) \in \tilde{\mathfrak{Q}}, \quad t \in [0,1], \]
then 
\[ b_l(t)-b_r(t)=||\gamma'(t)||, \quad 2d(t)=||\gamma'(t)||\kappa_\gamma(t), \quad b_l(t)+b_r(t)=||\gamma'(t)||\tau_\gamma(t)\]
where the curve $\gamma : [0,1] \rightarrow \SS^3$ is defined by
\[ \gamma(t)=(\Pi_4 \circ \tilde{\Gamma}(t))e_1. \]
\end{proposition}

\begin{proof}
By definition, $\tilde{\Gamma}$ is quasi-holonomic if and only if the projected curve
\[ \Gamma=\Pi_4 \circ \tilde{\Gamma} : [0,1] \rightarrow SO_4 \]
is quasi-Jacobian, and by definition, $\Gamma$ is quasi-Jacobian if only if its logarithmic derivative belongs to the subset $\mathfrak{Q}$ of matrices of the form
\[ \begin{pmatrix}
0 & -c_1 & 0 & 0  \\
c_1 & 0 & -c_2 & 0 \\
0 & c_2 & 0 & -c_3  \\
0 & 0 & c_3 & 0 
\end{pmatrix}, \quad c_1>0, c_2>0, c_3 \in \R. \]
By the chain rule we have
\[ \Gamma'(t)=(D_{\tilde{\Gamma}(t)}\Pi_4) \tilde{\Gamma}'(t) \]
hence
\[ \Lambda_{\Gamma}(t)=\Gamma(t)^{-1}\Gamma'(t)=\Gamma(t)^{-1}(D_{\tilde{\Gamma}(t)}\Pi_4) \tilde{\Gamma}'(t)=\Gamma(t)^{-1}(D_{\tilde{\Gamma}(t)}\Pi_4)\tilde{\Gamma}(t)\Lambda_{\tilde{\Gamma}}(t). \]
But since $\Gamma(t)^{-1}(D_{\tilde{\Gamma}(t)}\Pi_4)\tilde{\Gamma}(t)$ is the differential of $\Pi_4$ at the identity $(\1,\1)$, we obtain 
\[ \Lambda_{\Gamma}(t)=(D_{(\1,\1)}\Pi_4) \Lambda_{\tilde{\Gamma}}(t) \]
hence to prove the first part of the proposition, one needs to prove that $\mathfrak{Q}=D_{(\1,\1)}\Pi_4(\tilde{\mathfrak{Q}})$. The differential 
\[ D_{(\1,\1)}\Pi_4 : \mathrm{Im}\mathbb{H} \times \mathrm{Im}\mathbb{H} \rightarrow so_4 \]
is given by
\[ D_{(\1,\1)}\Pi_4(h_l,h_r) : z \in \H \mapsto h_lz-zh_r \in \H \]
for $(h_1,h_2) \in \mathrm{Im}\mathbb{H} \times \mathrm{Im}\mathbb{H}$. If we let 
\[ h_l=b_l\i+c_l\j+d_l\k, \quad h_r=b_r\i+c_r\j+d_r\k  \]
then 
\[ D_{(\1,\1)}\Pi_4(h_l,h_r)z=b_l\i z+c_l\j z+d_l\k z-(b_rz\i+c_rz\j+d_rz\k). \]
Let us denote by $\i_l$, $\j_l$ and $\k_l$ the matrices in $so_4$ that corresponds to left multiplication by respectively $\i$, $\j$ and $\k$; similarly we define $\i_r$, $\j_r$ and $\k_r$ the matrices in $so_4$ that corresponds to right multiplication by respectively $\bar{\i}$, $\bar{\j}$ and $\bar{\k}$. These matrices are given by 
\[      
\i_l=
\begin{pmatrix}
0 & -1 & 0 & 0 \\
+1 & 0 & 0 & 0 \\
0& 0 & 0  & -1 \\
0 & 0 & +1 & 0
\end{pmatrix},
\quad
\i_r=
\begin{pmatrix}
0 & +1 & 0 & 0 \\
-1 & 0 & 0 & 0 \\
0 & 0 & 0  & -1 \\
0 & 0 & +1 & 0
\end{pmatrix},
\]

\[
\j_l=
\begin{pmatrix}
0 & 0 & -1 & 0 \\
0 & 0 & 0 & +1 \\
+1 & 0 & 0  & 0 \\
0 & -1 & 0 & 0
\end{pmatrix},
\quad 
\j_r=
\begin{pmatrix}
0 & 0 & +1 & 0 \\
0 & 0 & 0 & +1 \\
-1 & 0 & 0  & 0 \\
0 & -1 & 0 & 0
\end{pmatrix},
\]
\[
\k_l=
\begin{pmatrix}
0 & 0 & 0 & -1 \\
0 & 0 & -1 & 0 \\
0 & +1 & 0  & 0 \\
+1 & 0 & 0 & 0
\end{pmatrix},
\quad
\k_r=
\begin{pmatrix}
0 & 0 & 0 & +1 \\
0 & 0 & -1 & 0 \\
0 & +1 & 0  & 0 \\
-1 & 0 & 0 & 0
\end{pmatrix}.
\]
We can then express $D_{(\1,\1)}\Pi_4(h_l,h_r)$ in matrix notation:
\begin{equation*} 
D_{(\1,\1)}\Pi_4(h_l,h_r)=
\begin{pmatrix}
0 & -(b_l-b_r)& -(c_l-c_r) & -(d_l-d_r) \\
b_l-b_r & 0 & -(d_l+d_r) &  -(-c_l-c_r) \\
c_l-c_r& d_l+d_r  & 0 & -(b_l+b_r)\\
d_l-d_r & -c_l-c_r & b_l+b_r  & 0
\end{pmatrix}.
\end{equation*}
From this expression, it is clear that $(h_l,h_r) \in \tilde{\mathfrak{Q}}$ if and only if $D_{(\1,\1)}\Pi_4(h_l,h_r) \in \mathfrak{Q}$. This proves the equality $\mathfrak{Q}=D_{(\1,\1)}\Pi_4(\tilde{\mathfrak{Q}})$, and hence the first part of the proposition.

Concerning the second part of the proposition, if
\[ \Lambda_{\tilde{\Gamma}}(t)=(b_l(t)\i+d(t)\k,b_r(t)\i+d(t)\k) \in \tilde{\mathfrak{Q}}, \quad t \in [0,1], \] 
then $\Lambda_{\Gamma}(t)=D_{(\1,\1)}\Pi_4(\Lambda_{\tilde{\Gamma}}(t))$ is equal to
\begin{equation*} 
\begin{pmatrix}
0 & -(b_l(t)-b_r(t))& 0 & 0 \\
b_l(t)-b_r(t) & 0 & -2d(t) &  0 \\
0 & 2d(t)  & 0 & -(b_l(t)+b_r(t))\\
0 & 0 & b_l(t)+b_r(t)  & 0
\end{pmatrix}.
\end{equation*}  
But recall (see~\eqref{logderives3}, Chapter~\ref{chapter3}, \S\ref{s35}) that we also have
\begin{equation*}
\Lambda_\Gamma(t)=\Lambda_\gamma(t)=
\begin{pmatrix}
0 & -||\gamma'(t)|| & 0  & 0 \\
||\gamma'(t)|| & 0 & -||\gamma'(t)||\kappa_\gamma(t) & 0  \\
0 & ||\gamma'(t)||\kappa_\gamma(t) & 0 & -||\gamma'(t)||\tau_\gamma(t) \\
0 & 0 & ||\gamma'(t)||\tau_\gamma(t) & 0
\end{pmatrix}
\end{equation*}
where
\[ \gamma(t)=\Gamma(t)e_1=(\Pi_4 \circ \tilde{\Gamma}(t))e_1. \]
So a simple comparison between the two expressions of $\Lambda_\Gamma(t)$ proves the second part of the proposition.
\end{proof}

\medskip

This proposition will allow us to prove Theorem~\ref{th0}.

\begin{proof}[{} of Theorem~\ref{th0}]
Let $\gamma \in \mathcal{G}\SS^3(z_l,z_r)$. Are defined its Frenet frame curve $\mathcal{F}_\gamma(t)$, its lifted Frenet frame curve $\tilde{\Gamma}(t)=\tilde{\mathcal{F}}_\gamma(t)$ and the logarithmic derivative
\[ \Lambda_{\tilde{\Gamma}}(t)=\tilde{\Gamma}(t)^{-1}\tilde{\Gamma}'(t). \]
Using Proposition~\ref{propjacobian2} from Chapter~\ref{chapter3}, \S\ref{s35}, we know that $\mathcal{F}_\gamma$ is quasi-Jacobian, hence $\tilde{\Gamma}=\tilde{\mathcal{F}}_\gamma$ is quasi-holonomic. Thus we can apply Proposition~\ref{propth1} and we can uniquely write
\[ \Lambda_{\tilde{\Gamma}}(t)=(b_l(t)\i+d(t)\k,b_r(t)\i+d(t)\k) \]
with 
\[ b_l(t)-b_r(t)=||\gamma'(t)||, \quad 2d(t)=||\gamma'(t)||\kappa_\gamma(t), \quad b_l(t)+b_r(t)=||\gamma'(t)||\tau_\gamma(t). \]
Equivalently,
\begin{equation}\label{bbd}
\begin{cases}
d(t)=||\gamma'(t)||\kappa_\gamma(t)/2, \\
b_l(t)=||\gamma'(t)||(\tau_\gamma(t)+1)/2, \\ 
b_r(t)=||\gamma'(t)||(\tau_\gamma(t)-1)/2.
\end{cases}  
\end{equation}
Let us then define the curves
\[ \tilde{\Gamma}_l : [0,1] \rightarrow \SS^3, \quad \tilde{\Gamma}_r : [0,1] \rightarrow \SS^3 \]
by
\[ \tilde{\Gamma}_l(0)=\mathbf{1}, \quad \tilde{\Gamma}_l(1)=z_l,\quad \Lambda_{\tilde{\Gamma}_l}(t)=b_l(t)\i+d(t)\k \in \mathrm{Im}\H, \]
and 
\[ \tilde{\Gamma}_r(0)=\mathbf{1}, \quad \tilde{\Gamma}_r(1)=z_r,\quad \Lambda_{\tilde{\Gamma}_r}(t)=b_r(t)\i+d(t)\k \in \mathrm{Im}\H.  \]
The curves $\tilde{\Gamma}_l$ and $\tilde{\Gamma}_r$ are uniquely defined. Let us further define
\[ \Gamma_l:=\Pi_3 \circ \tilde{\Gamma}_l : [0,1] \rightarrow SO_3, \quad \Gamma_r:=\Pi_3 \circ \tilde{\Gamma}_r : [0,1] \rightarrow SO_3 \]
where we recall that $\Pi_3 : \SS^3 \rightarrow SO_3$ is the universal cover projection. Next we want to compute the logarithmic derivative of $\Gamma_l$ and $\Gamma_r$. The differential of $\Pi_3$ at $\mathbf{1}$ can be computed exactly as we computed the differential of $\Pi_4$ at $(\mathbf{1},\mathbf{1})$ (in the proof of Proposition~\ref{propth1});
we have
\[ D_{\1}\Pi_3 : \mathrm{Im}\H \rightarrow so_3 \]
and for $h=(b\i+c\j+d\k) \in \mathrm{Im}\H$, we can write in matrix notation
\begin{equation*} 
D_{\1}\Pi_3(h)=
\begin{pmatrix}
0 & -2d & -2c  \\
2d & 0 & -2b  \\
2c & 2b  & 0 \\
\end{pmatrix}.
\end{equation*}
From this expression we obtain
\begin{equation}\label{exp1}
\Lambda_{\Gamma_l}(t)=D_{\1}\Pi_3(\Lambda_{\tilde{\Gamma}_l}(t))=
\begin{pmatrix}
0 & -2d(t) & 0  \\
2d(t) & 0 & -2b_l(t)  \\
0 & 2b_l(t)  & 0 \\
\end{pmatrix}
\end{equation}
and
\begin{equation}\label{exp2}
\Lambda_{\Gamma_r}(t)=D_{\1}\Pi_3(\Lambda_{\tilde{\Gamma}_r}(t))=
\begin{pmatrix}
0 & -2d(t) & 0  \\
2d(t) & 0 & -2b_r(t)  \\
0 & 2b_r(t)  & 0 \\
\end{pmatrix}. 
\end{equation}
From~\eqref{bbd}, we see that $d(t)>0$ and $b_l(t) \in \R$, hence $\Gamma_l$ is a quasi-Jacobian curve, and therefore if we define
\[ \gamma_l(t):=\Gamma_l(t)e_1 \]
then $\gamma_l \in \mathcal{G}\SS^2(z_l)$. Moreover, recall from~\eqref{logderives2}, Chapter~\ref{chapter3}, \S\ref{s34}, that
\begin{equation*}
\Lambda_{\Gamma_l}(t)=\Lambda_{\gamma_l}(t)=
\begin{pmatrix}
0 & -||\gamma_l'(t)|| & 0  \\
||\gamma_l'(t)|| & 0 & -||\gamma_l'(t)||\kappa_{\gamma_l}(t)  \\
0 & ||\gamma_l'(t)||\kappa_{\gamma_l}(t)  & 0 \\
\end{pmatrix}
\end{equation*}
so that comparing this with~\eqref{exp1} and recalling~\eqref{bbd}, we find
\[ ||\gamma_l'(t)||=2d(t)=||\gamma'(t)||\kappa_{\gamma}(t)\] 
and
\[ \kappa_{\gamma_l}(t)=\frac{2b_l(t)}{ ||\gamma_l'(t)||}=\frac{||\gamma_l'(t)||(\tau_\gamma(t)+1)}{ ||\gamma_l'(t)||\kappa_\gamma(t)}=\frac{\tau_\gamma(t)+1}{\kappa_\gamma(t)}. \]
Now $\Gamma_r$ is also a quasi-Jacobian curve, hence if we define
\[ \gamma_r(t):=\Gamma_r(t)e_1, \]
then $\gamma_r \in \mathcal{G}\SS^2(z_r)$, and as before, we have
\begin{equation*}
\Lambda_{\Gamma_r}(t)=\Lambda_{\gamma_r}(t)=
\begin{pmatrix}
0 & -||\gamma_r'(t)|| & 0  \\
||\gamma_r'(t)|| & 0 & -||\gamma_r'(t)||\kappa_{\gamma_r}(t)  \\
0 & ||\gamma_r'(t)||\kappa_{\gamma_r}(t)  & 0 \\
\end{pmatrix}
\end{equation*}
and
\[ ||\gamma_r'(t)||=||\gamma'(t)||\kappa_{\gamma}(t), \quad \kappa_{\gamma_r}(t)=\frac{\tau_\gamma(t)-1}{\kappa_\gamma(t)}. \]
This shows that given $\gamma \in \mathcal{L}\SS^3(z_l,z_r)$, there exists a unique pair of curves $(\gamma_l,\gamma_r)$, with $\gamma_l \in \mathcal{G}\SS^2(z_l)$ and $\mathcal{G}\SS^2(z_r)$ such that
\[ ||\gamma_l'(t)||=||\gamma_r'(t)||, \quad \kappa_{\gamma_l}(t)>\kappa_{\gamma_r}(t) \]
and moreover
\[ ||\gamma_l'(t)||=||\gamma_r'(t)||=||\gamma'(t)||\kappa_{\gamma}(t), \quad \kappa_{\gamma_l}(t)=\frac{\tau_\gamma(t)+1}{\kappa_\gamma(t)}, \quad \kappa_{\gamma_r}(t)=\frac{\tau_\gamma(t)-1}{\kappa_\gamma(t)}.  \]
This defines a map $\gamma \mapsto (\gamma_l,\gamma_r)$, which, by construction is continuous. Conversely, given a pair of curves $(\gamma_l,\gamma_r)$, with $\gamma_l \in \mathcal{G}\SS^2(z_l)$ and $\mathcal{G}\SS^2(z_r)$ such that
\[ ||\gamma_l'(t)||=||\gamma_r'(t)||, \quad \kappa_{\gamma_l}(t)>\kappa_{\gamma_r}(t), \]
by simply reversing the construction above, we can find a unique curve $\gamma \in \mathcal{G}\SS^3(z_l,z_r)$ such that
\[ \kappa_\gamma(t)=\frac{2}{\kappa_{\gamma_l}(t)-\kappa_{\gamma_r}(t)}, \] 
\[ \tau_\gamma(t)=\frac{\kappa_\gamma(t)(\kappa_{\gamma_l}(t)+\kappa_{\gamma_r}(t))}{2}=\frac{\kappa_{\gamma_l}(t)+\kappa_{\gamma_r}(t)}{\kappa_{\gamma_l}(t)-\kappa_{\gamma_r}(t)}, \] 
\[ ||\gamma'(t)||=\frac{||\gamma_l'(t)||}{\kappa_\gamma(t)}=\frac{||\gamma_l'(t)||(\kappa_{\gamma_l}(t)-\kappa_{\gamma_r}(t))}{2}.  \]
This also defines a map $(\gamma_l,\gamma_r) \mapsto \gamma$, which is also clearly continuous, and this completes the proof of the theorem.
\end{proof}

The proof of Theorem~\ref{th1} follows directly from the statement of Theorem~\ref{th0}. Alternatively, one can proceed exactly as in the proof of Theorem~\ref{th0}, replacing quasi-holonomic curves (respectively quasi-Jacobian curves) by holonomic curves (respectively Jacobian curves), using Proposition~\ref{propjacobian} instead of Proposition~\ref{propjacobian2}, and replacing $\mathfrak{Q}$ and $\tilde{\mathfrak{Q}}$ by respectively $\mathfrak{J}$ and  
\[ \tilde{\mathfrak{J}}:=\{(b_l\i+d\k,b_r\i+d\k) \in \mathrm{Im}\mathbb{H} \times \mathrm{Im}\mathbb{H} \; | \; (b_l,b_r,d) \in \R^3, \; b_l>|b_r|, \; d>0 \}. \]

\section{Examples}\label{s62}

In this section, we want to use Theorem~\ref{th1} to produce examples in the spaces we are interested in: namely
\[\mathcal{L}\SS^3(\1,\1), \quad \mathcal{L}\SS^3(-\1,-\1), \quad \mathcal{L}\SS^3(\1,-\1), \quad \mathcal{L}\SS^3(-\1,\1). \]
We will also give examples in the spaces
\[ \mathcal{L}\SS^3(-\1,-\k), \quad \mathcal{L}\SS^3(\1,\k), \quad  \mathcal{L}\SS^3(-\1,\k), \quad \mathcal{L}\SS^3(\1,-\k). \]
As we will see later in Chapter~\ref{chapter7}, \S\ref{s73}, these spaces are respectively homeomorphic to the spaces we are interested in, and they are sometimes more convenient to work with. 

These examples not only serve as an illustration of Theorem~\ref{th1}, but they will be also used throughout the thesis.

\medskip

Recall that theorem~\ref{th1} gives us a homemorphism between the space of $\gamma \in \mathcal{L}\SS^3(z_l,z_r)$ and the space of pairs $(\gamma_l,\gamma_r) \in \mathcal{L}\SS^2(z_l) \times \mathcal{G}\SS^2(z_r)$ for which
\[ ||\gamma_l'(t)||=||\gamma_r'(t)||, \quad \kappa_{\gamma_l}(t)>|\kappa_{\gamma_r}(t)|, \quad t \in [0,1]. \]
This allows us to decompose a locally convex curve in $\SS^3$ as a pair of a locally convex curve in $\SS^2$ and an immersion in $\SS^2$, with some compatibility conditions. Hence to produce examples of locally convex curve in $\SS^3$, it is enough to produce examples of such pairs. All our examples will be constructed as follows. 

For a real number $0<c \leq 2\pi$, let $\sigma_c : [0,1] \rightarrow \SS^2$ be the unique circle of length $c$, that is $||\sigma_c'(t)||=c$, with fixed initial and final Frenet frame equals to the identity. Setting $c=2\pi\sin\rho$ ($\rho \in (0,\pi/2]$ is the radius of curvature), this curve can be given by the following formula
\[ \sigma_c(t)=\cos\rho(\cos\rho,0,\sin\rho)+\sin\rho(\sin\rho\cos(2\pi t), \sin(2\pi t), -\cos\rho\cos(2\pi t)). \] 
The geodesic curvature of this curve is given by $\cot(\rho) \in [0,+\infty)$. For $c<2\pi$, $\sigma_c$ is locally convex but also convex, but for $c=2\pi$, this is a meridian curve
\[ \sigma_{2\pi}(t)=(\cos(2\pi t), \sin(2\pi t), 0) \]
which has zero geodesic curvature, so this is just an immersion. For the left part of our curves, we will use $\sigma_c$ with $c<2\pi$ and iterate it a certain number of times, and for the right part of curves, we will use $\sigma_{2\pi}$ and iterate it a certain number of times. Since the right part will always have zero geodesic curvature, the only restriction so that this pair of curves defines a locally convex curve in $\SS^3$ is the condition that their length should be equal. However, in order to realize different final lifted Frenet frame, we will have to iterate the curve $\sigma_c$ (on the left) and the curve $\sigma_{2\pi}$ (on the right) a different number of times: the equality of length will be achieved by properly choosing $c$ in each case. Given $k>0$, let us define the curve $\sigma_c^k$ as the curve $\sigma_c$ iterated $k$ times, that is 
\[ \sigma_c^k(t)=\sigma_c(kt), \quad t \in [0,1]. \]   
In the sequel, $k$ will either be an integer or half an integer. For instance, given any $m \in \N$, then
\[ \sigma_c^m \in \mathcal{L}\SS^2((-\1)^m), \; 0<c<2\pi, \quad \sigma_{2\pi}^{m} \in \mathcal{G}\SS^2((-\1)^m)  \]
and
\[ \sigma_{2\pi}^{m/2} \in \mathcal{G}\SS^2(\k^m). \]
See the Figure~\ref{fig:b} below for an illustration.

\begin{figure}[H]
\centering
\includegraphics[scale=0.5]{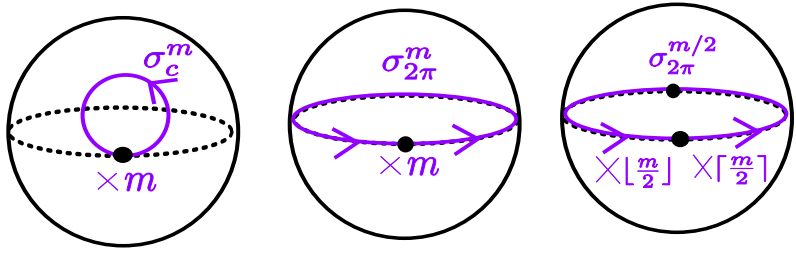}
\caption{The curves $\sigma_c^m$, $\sigma_{2\pi}^m$ and $\sigma_{2\pi}^{m/2}$.}
\label{fig:b}
\end{figure}

\begin{example}[Spaces $\mathcal{L}\SS^3((-\1)^m,\k^m)$]\label{family1}
Let us give explicit examples in the spaces $\mathcal{L}\SS^3((-\1)^m,\k^m)$, $m \geq 1$. For $m\equiv 1$, $2$, $3$ or $4$ modulo $4$, this will give examples in the spaces $\mathcal{L}\SS^3(-\1,\k)$, $\mathcal{L}\SS^3(\1,-\1)$, $\mathcal{L}\SS^3(-\1,-\k)$ and $\mathcal{L}\SS^3(\1,\1)$.

\end{example}
We want to define a curve $\gamma_1^m \in \mathcal{L}\SS^3((-\1)^m,\k^m)$ such that its left and right part are given by
\[ \gamma_{1,l}^m=\sigma_c^m \in \mathcal{L}\SS^2((-\1)^m), \quad \gamma_{1,r}^m =\sigma_{2\pi}^{m/2} \in \mathcal{G}\SS^2(\k^m). \]
To define a pair of curves, we need to choose $0<c<2\pi$ such that
\[ ||(\gamma_{1,l}^{m})'(t)||=||(\sigma_c^{m})'(t)||=cm \]
is equal to 
\[ ||(\gamma_{1,r}^{m})'(t)||=||(\sigma_{2\pi}^{m/2})'(t)||=\pi m. \]
It suffices to choose $c=\pi$ so that both curves have length equal to $\pi m$, then the geodesic curvature of $\gamma_{1,l}^m=\sigma_c^m$ is constantly equal to $\sqrt{3}$ while clearly, the geodesic curvature of $\gamma_{1,r}^m =\sigma_{2\pi}^{m/2}$ is zero.  

Let us now find explicitly the curve $\gamma_1^m$. From Theorem~\ref{th1}, we can compute
\[ ||(\gamma_1^{m})'(t)||=\frac{||(\gamma_{1,l}^m)'(t)||(\kappa_{\gamma_{1,l}}(t)-\kappa_{\gamma_{1,r}}(t))}{2}=\frac{m\pi\sqrt{3}}{2} \]
\[ \kappa_{\gamma_1^m}(t)=\frac{2}{\kappa_{\gamma_{1,l}}(t)-\kappa_{\gamma_{1,r}}(t)}=\frac{2}{\sqrt{3}},\] 
\[\tau_{\gamma_1^m}(t)=\frac{\kappa_{\gamma_{1,l}}(t)+\kappa_{\gamma_{1,r}}(t)}{\kappa_{\gamma_{1,l}}(t)-\kappa_{\gamma_{1,r}}(t)}=1. \]
Therefore the lifted logarithmic derivative and logarithmic derivative of $\gamma_1^m$ are constant and given by
\[ \tilde{\Lambda}_{\gamma_1^m}=\frac{1}{2}\left(m\sqrt{3}\pi\i+m\pi\k,m\pi\k\right) \]
and
\[  
\Lambda_{\gamma_1^m}=\frac{\pi}{2}
\begin{pmatrix}
0 & -m\sqrt{3} & 0  & 0 \\
m\sqrt{3} & 0 & -2m & 0  \\
0 & 2m & 0 & -m\sqrt{3} \\
0 & 0 & m\sqrt{3} & 0
\end{pmatrix}.
\]
The holonomic curve can be computed explicitly:
\[ \tilde{\Gamma}_{\gamma_1^m}(t)=\left(\exp\left(m\pi t\frac{\sqrt{3}\i+\k}{2}\right),\exp\left(m\pi t\frac{\k}{2}\right)\right). \]
The Jacobian curve $\Gamma_{\gamma_1^m}$ satisfies
\[ \Gamma_{\gamma_1^m}'(t)=\Gamma_{\gamma_1^m}(t)\Lambda_{\gamma_1^m}, \quad \Gamma_{\gamma_1^m}(0)=I \]
and can also be computed explicitly since it is the exponential of $\Gamma_{\gamma_1^m}$, that is
\[ \Gamma_{\gamma_1^m}(t)=\exp(t\Lambda_{\gamma_1^m}). \] 
The curve $\gamma_1^m$ is then equal to $\Gamma_{\gamma_1^m}e_1$, and we find that

\begin{eqnarray*}
\gamma_1^m(t) & = & \left(\frac{1}{4}\cos\left(\frac{3}{2}t\pi m\right)+\frac{3}{4}\cos\left(\frac{1}{2}t\pi m\right) \right.,\\
&  & \frac{\sqrt{3}}{4}\sin\left(\frac{3}{2}t\pi m\right)+\frac{\sqrt{3}}{4}\sin\left(\frac{1}{2}t\pi m\right), \\ 
&  & \frac{\sqrt{3}}{4}\cos\left(\frac{1}{2}t\pi m\right)-\frac{\sqrt{3}}{4}\cos\left(\frac{3}{2}t\pi m\right), \\
&  & \left.\frac{3}{4}\sin\left(\frac{1}{2}t\pi m\right)-\frac{1}{4}\sin\left(\frac{3}{2}t\pi m\right)\right).
\end{eqnarray*}

Below we give some illustrations in the case $m=1$ (Figure~\ref{fig:c}), $m=2$ (Figure~\ref{fig:d}), $m=3$ (Figure~\ref{fig:e}), $m=4$ (Figure~\ref{fig:f}) and $m=5$ (Figure~\ref{fig:g}).  

\begin{figure}[H]
\centering
\includegraphics[scale=0.5]{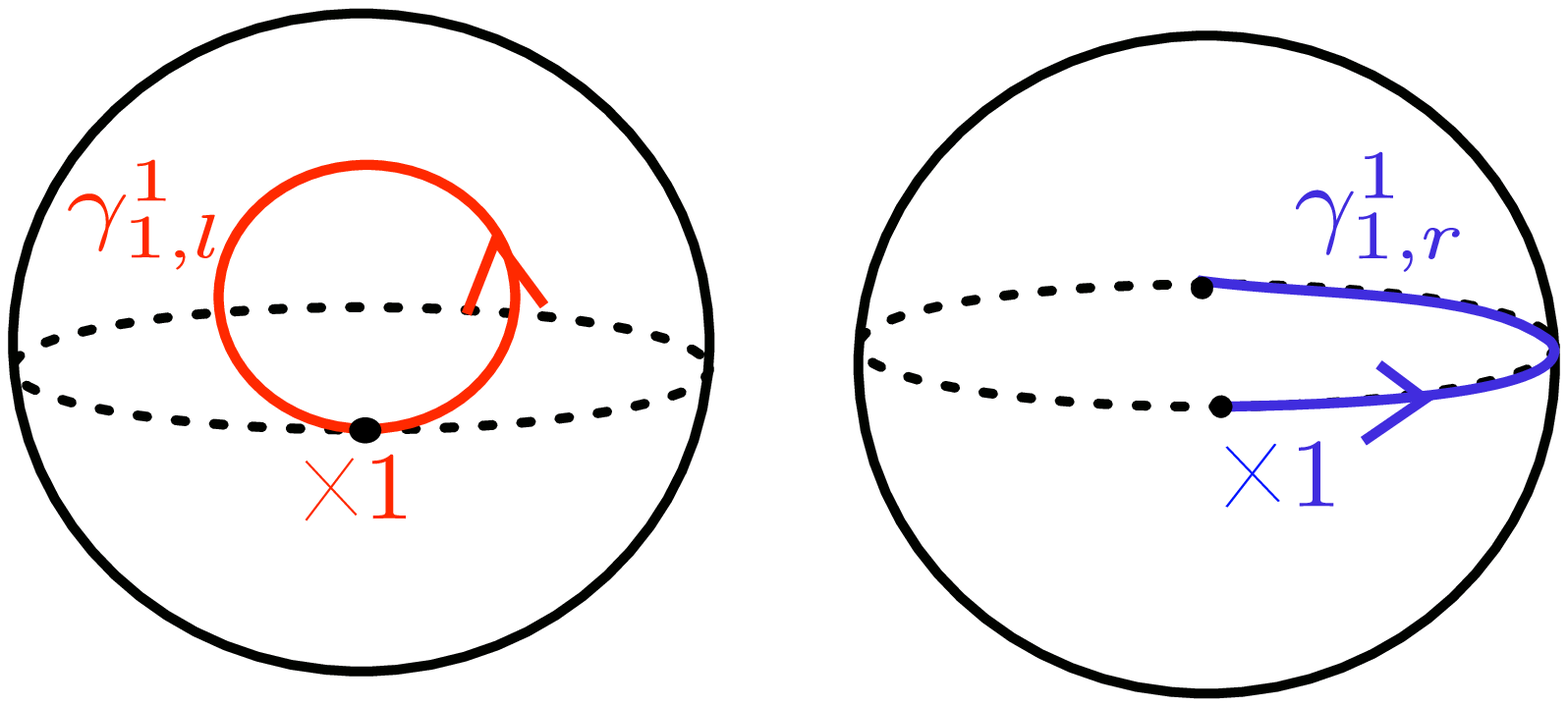}
\caption{The curve $\gamma_1^1$.}
\label{fig:c}
\end{figure}

\begin{figure}[H]
\centering
\includegraphics[scale=0.5]{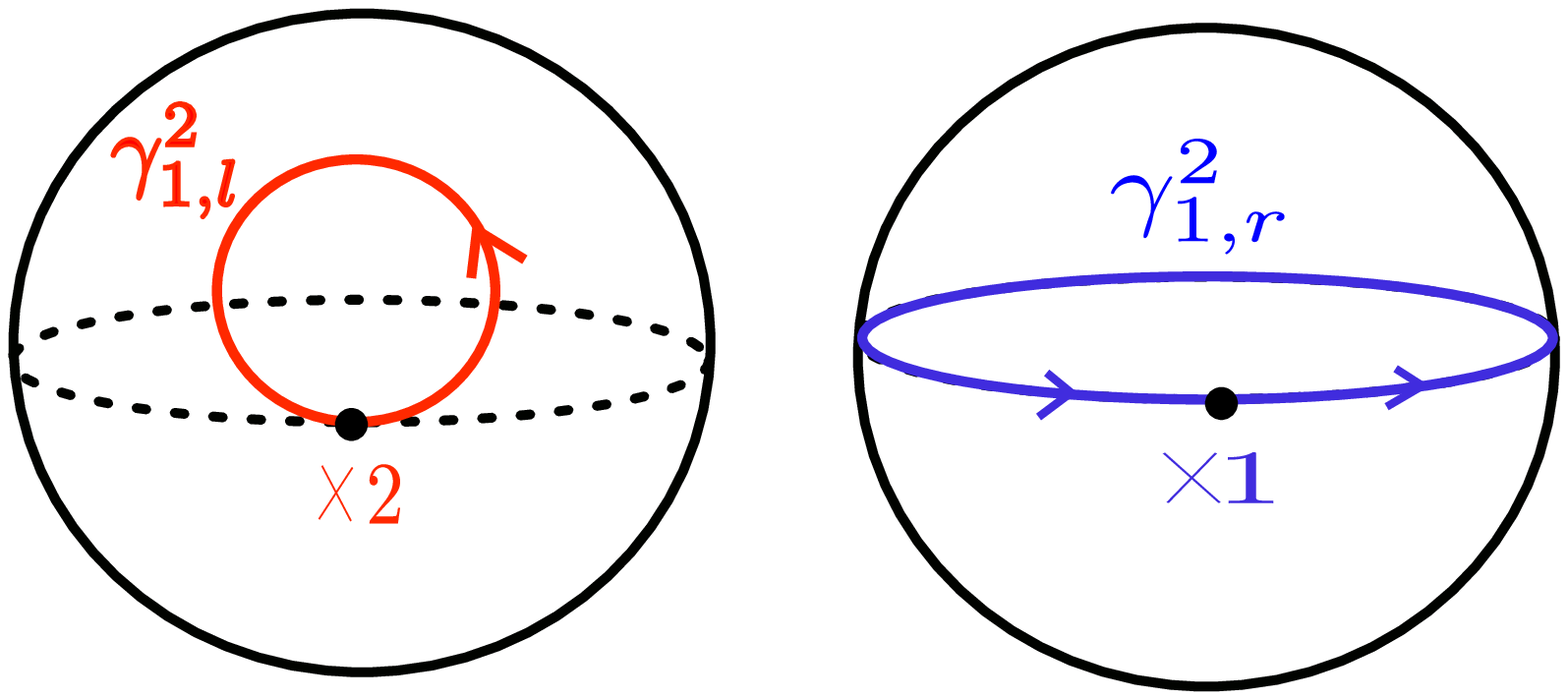}
\caption{The curve $\gamma_1^2$.}
\label{fig:d}
\end{figure}

\begin{figure}[H]
\centering
\includegraphics[scale=0.5]{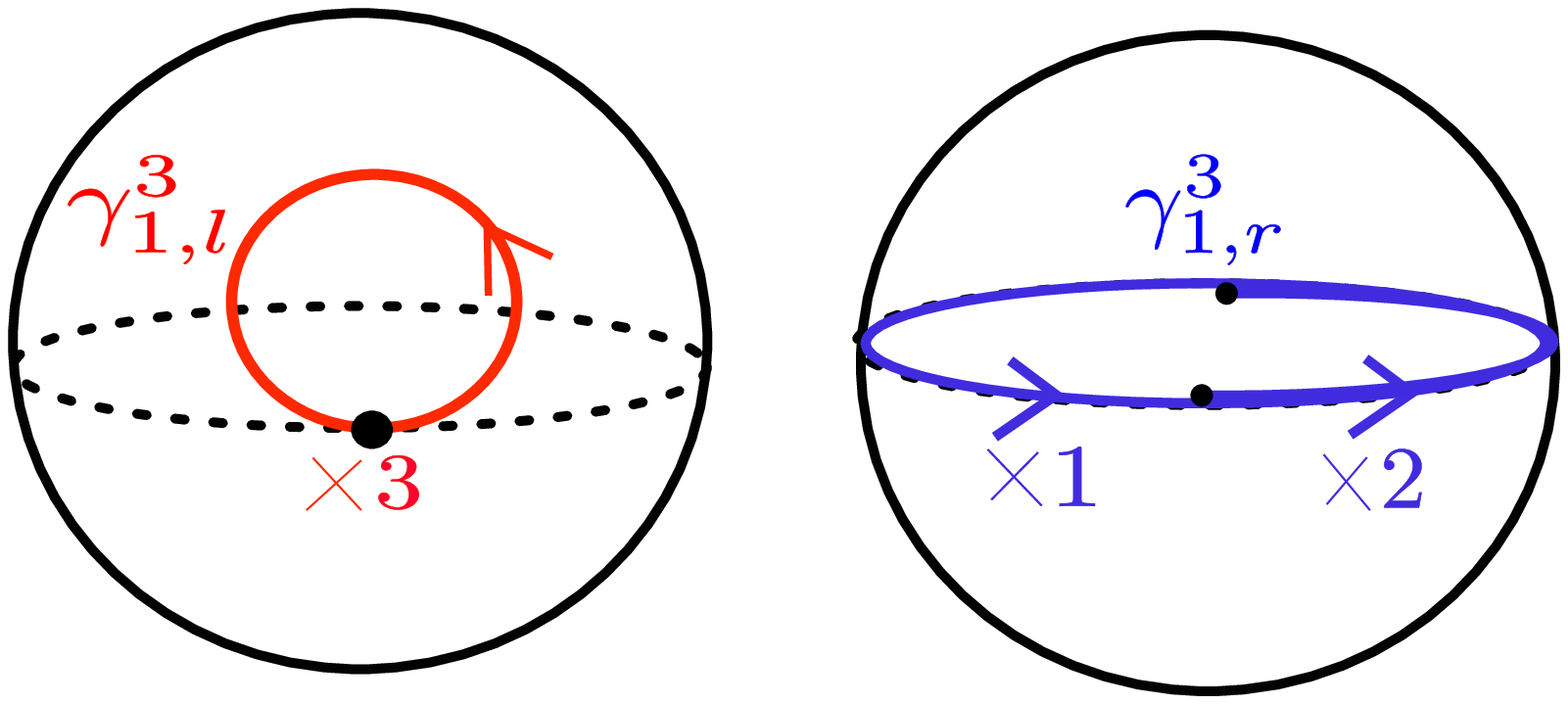}
\caption{The curve $\gamma_1^3$.}
\label{fig:e}
\end{figure}

\begin{figure}[H]
\centering
\includegraphics[scale=0.5]{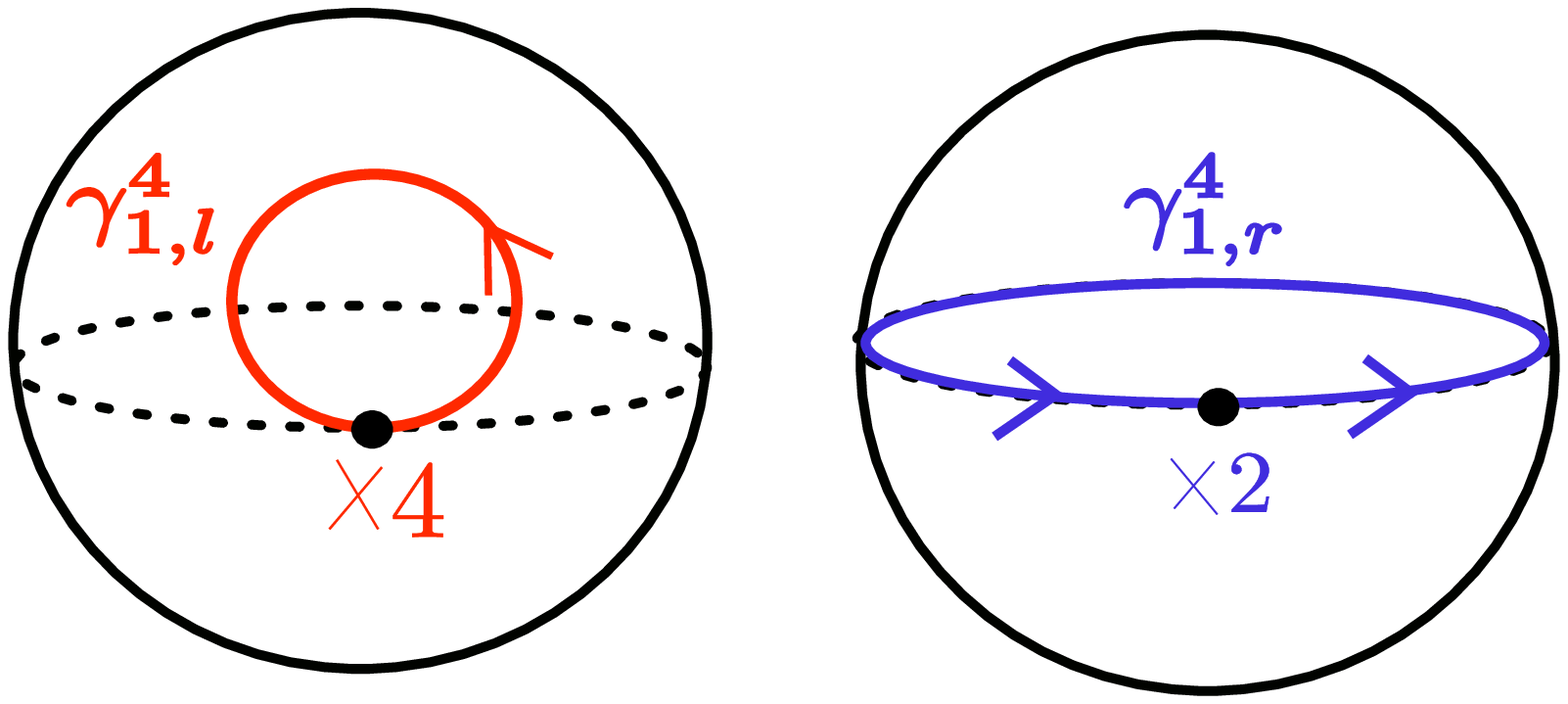}
\caption{The curve $\gamma_1^4$.}
\label{fig:f}
\end{figure}

\begin{figure}[H]
\centering
\includegraphics[scale=0.5]{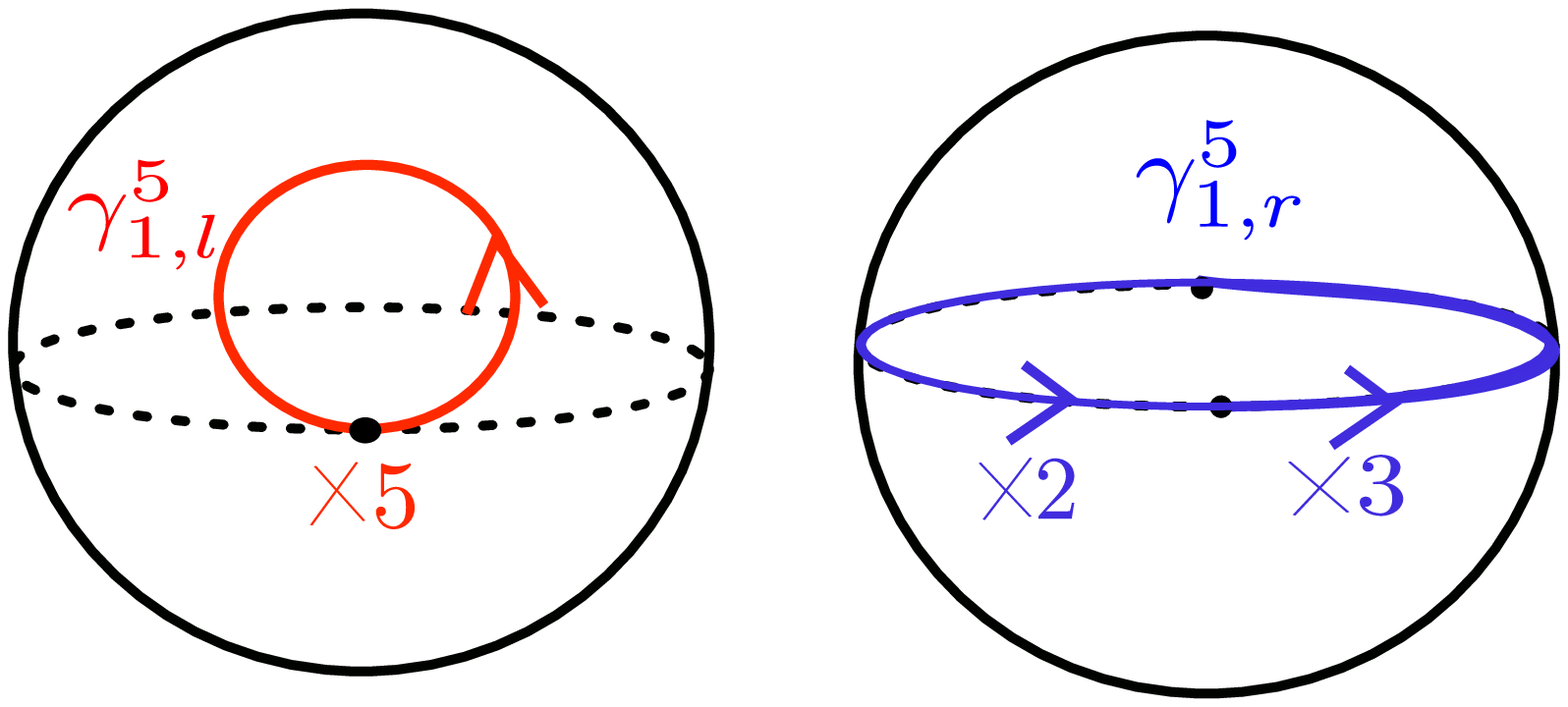}
\caption{The curve $\gamma_1^5$.}
\label{fig:g}
\end{figure}

\begin{example}[Spaces $\mathcal{L}\SS^3(\1,\k^m)$]\label{family2}
Let us give explicit examples in the spaces $\mathcal{L}\SS^3(\1,\k^m)$, $m \geq 1$. For $m\equiv 1$, $2$, $3$ or $4$ modulo $4$, this will give examples in the spaces $\mathcal{L}\SS^3(\1,\k)$, $\mathcal{L}\SS^3(\1,-\1)$, $\mathcal{L}\SS^3(\1,-\k)$ and $\mathcal{L}\SS^3(\1,\1)$.
\end{example}
We want to define a curve $\gamma_2^m \in \mathcal{L}\SS^3(\1,\k^m)$ such that its left and right part are given by
\[ \gamma_{2,l}^m=\sigma_c^{2m} \in \mathcal{L}\SS^2(\1), \quad \gamma_{2,r}^m =\sigma_{2\pi}^{m/2} \in \mathcal{G}\SS^2(\k^m). \]
To define a pair of curves, we need to choose $0<c<2\pi$ such that
\[ ||(\gamma_{2,l}^{2m})'(t)||=||(\sigma_c^{2m})'(t)||=2cm \]
is equal to 
\[ ||(\gamma_{2,r}^{m})'(t)||=||(\sigma_{2\pi}^{m/2})'(t)||=\pi m. \]
It suffices to choose $c=\pi/2$ so that both curves have length equal to $\pi m$, then the geodesic curvature of $\gamma_{2,l}^m=\sigma_c^{2m}$ is constantly equal to $\sqrt{15}$ while clearly, the geodesic curvature of $\gamma_{2,r}^m =\sigma_{2\pi}^{m/2}$ is zero.  

As before, we can find explicitly the curve $\gamma_2^m$. From Theorem~\ref{th1}, we can compute
\[ ||(\gamma_2^{m})'(t)||=\frac{||(\gamma_{2,l}^m)'(t)||(\kappa_{\gamma_{2,l}}(t)-\kappa_{\gamma_{2,r}}(t))}{2}=\frac{m\pi\sqrt{15}}{2} \]
\[ \kappa_{\gamma_2^m}(t)=\frac{2}{\kappa_{\gamma_{2,l}}(t)-\kappa_{\gamma_{2,r}}(t)}=\frac{2}{\sqrt{15}},\] 
\[\tau_{\gamma_2^m}(t)=\frac{\kappa_{\gamma_{2,l}}(t)+\kappa_{\gamma_{2,r}}(t)}{\kappa_{\gamma_{2,l}}(t)-\kappa_{\gamma_{2,r}}(t)}=1. \]
Therefore the lifted logarithmic derivative and logarithmic derivative of $\gamma_1^m$ are constant and given by
\[ \tilde{\Lambda}_{\gamma_2^m}=\frac{1}{2}\left(m\sqrt{15}\pi\i+m\pi\k,m\pi\k\right) \]
and
\[  
\Lambda_{\gamma_2^m}=\frac{\pi}{2}
\begin{pmatrix}
0 & -m\sqrt{15} & 0  & 0 \\
m\sqrt{15} & 0 & -2m & 0  \\
0 & 2m & 0 & -m\sqrt{15} \\
0 & 0 & m\sqrt{15} & 0
\end{pmatrix}.
\]
The holonomic curve can be computed explicitly:
\[ \tilde{\Gamma}_{\gamma_2^m}(t)=\left(\exp\left(2m\pi t\frac{\sqrt{15}\i+\k}{4}\right),\exp\left(m\pi t\frac{\k}{2}\right)\right). \]
The Jacobian curve $\Gamma_{\gamma_2^m}$ satisfies
\[ \Gamma_{\gamma_2^m}'(t)=\Gamma_{\gamma_2^m}(t)\Lambda_{\gamma_2^m}, \quad \Gamma_{\gamma_2^m}(0)=I \]
and can also be computed explicitly since it is the exponential of $\Gamma_{\gamma_2^m}$, that is
\[ \Gamma_{\gamma_2^m}(t)=\exp(t\Lambda_{\gamma_2^m}). \] 
The curve $\gamma_2^m$ is then equal to $\Gamma_{\gamma_2^m}e_1$, and we find that
\begin{eqnarray*}
\gamma_2^m(t) & = & \left(\frac{3}{8}\cos\left(\frac{5}{2}t\pi m\right)+\frac{5}{8}\cos\left(\frac{3}{2}t\pi m\right)\right., \\
&  & \frac{\sqrt{15}}{8}\sin\left(\frac{5}{2}t\pi m\right)+\frac{\sqrt{15}}{8}\sin\left(\frac{3}{2}t\pi m\right), \\
&  & \frac{\sqrt{15}}{8}\cos\left(\frac{3}{2}t\pi m\right)-\frac{\sqrt{15}}{8}\cos\left(\frac{5}{2}t\pi m\right), \\
&  & \left.\frac{5}{8}\sin\left(\frac{3}{2}t\pi m\right)-\frac{3}{8}\sin\left(\frac{5}{2}t\pi m\right)\right).
\end{eqnarray*}
Below we give some illustrations in the case $m=1$ (Figure~\ref{fig:h}) and $m=3$ (Figure~\ref{fig:i}).

\begin{figure}[H]
\centering
\includegraphics[scale=0.5]{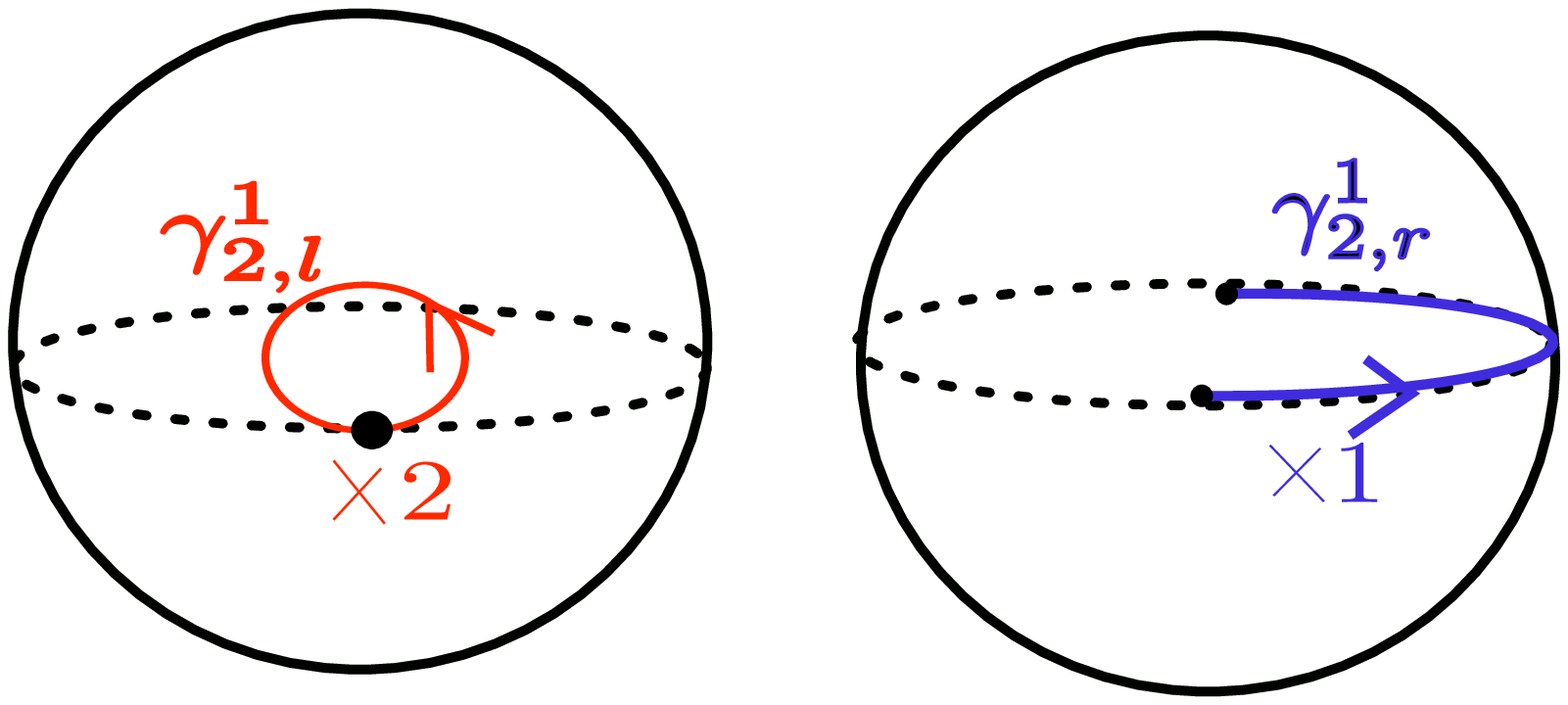}
\caption{The curve $\gamma_2^1$.}
\label{fig:h}
\end{figure}

\begin{figure}[H]
\centering
\includegraphics[scale=0.5]{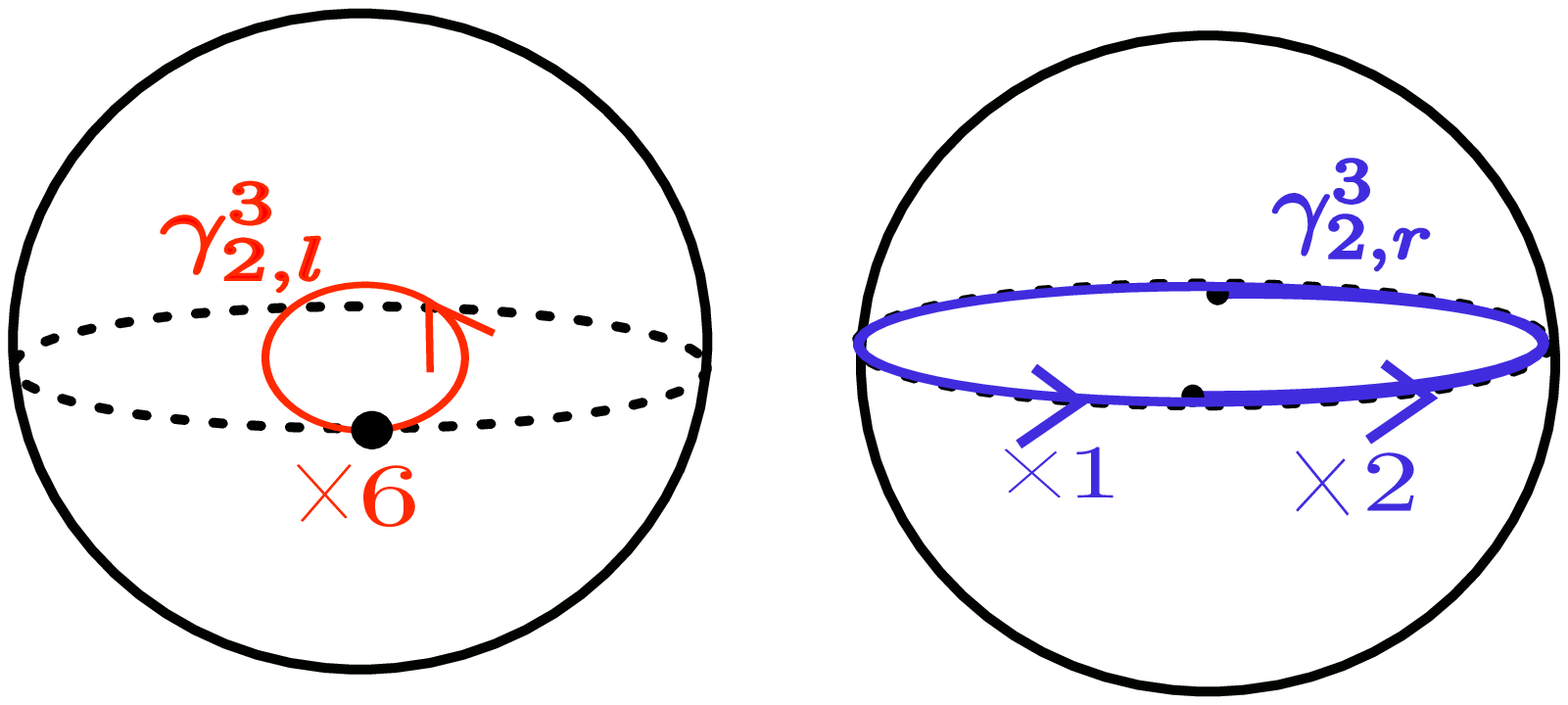}
\caption{The curve $\gamma_2^3$.}
\label{fig:i}
\end{figure}

We already gave examples of locally convex curves in the spaces $\mathcal{L}\SS^3(-\1,\k)$ and $\mathcal{L}\SS^3(-\1,-\k)$ (Example~\eqref{family1}), $\mathcal{L}\SS^3(1,\k)$ and  $\mathcal{L}\SS^3(\1,-\k)$ (Example~\eqref{family2}) and in $\mathcal{L}\SS^3(\1,-\1)$ and $\mathcal{L}\SS^3(\1,\1)$ (Example~\eqref{family1} and~\eqref{family2}). It remains to give examples in $\mathcal{L}\SS^3(-\1,-\1)$ and $\mathcal{L}\SS^3(-\1,\1)$. 

\begin{example}[Spaces $\mathcal{L}\SS^3((-\1)^{3m},(-\1)^m)$]\label{family3}
Let us give explicit examples in the spaces $\mathcal{L}\SS^3((-\1)^{3m},(-\1)^m)$, $m \geq 1$. For $m\equiv 1$, $2$ modulo $2$, this will give examples in the spaces $\mathcal{L}\SS^3(-\1,-\1)$ and $\mathcal{L}\SS^3(\1,\1)$.
\end{example}
We want to define a curve $\gamma_3^m \in \mathcal{L}\SS^3((-\1)^{3m},(-\1)^m)$ such that its left and right part are given by
\[ \gamma_{3,l}^m=\sigma_c^{3m} \in \mathcal{L}\SS^2((-\1^{3m}), \quad \gamma_{3,r}^m=\sigma_{2\pi}^{m} \in \mathcal{G}\SS^2((-\1)^m). \]
To define a pair of curves, we need to choose $0<c<2\pi$ such that
\[ ||(\gamma_{3,l}^{m})'(t)||=||(\sigma_c^{3m})'(t)||=3cm \]
is equal to 
\[ ||(\gamma_{2,r}^{m})'(t)||=||(\sigma_{2\pi}^{m})'(t)||=2\pi m. \]
It suffices to choose $c=2\pi/3$ so that both curves have length equal to $2 \pi m$, then the geodesic curvature of $\gamma_{3,l}^m=\sigma_c^{3m}$ is constantly equal to $2\sqrt{2}$ while clearly, the geodesic curvature of $\gamma_{3,r}^m =\sigma_{2\pi}^{m}$ is zero.  

As before, we can find explicitly the curve $\gamma_3^m$. From Theorem~\ref{th1}, we can compute
\[ ||(\gamma_3^{m})'(t)||=\frac{||(\gamma_{3,l}^m)'(t)||(\kappa_{\gamma_{3,l}}(t)-\kappa_{\gamma_{3,r}}(t))}{2}=2\sqrt{2}\pi m \]
\[ \kappa_{\gamma_3^m}(t)=\frac{2}{\kappa_{\gamma_{3,l}}(t)-\kappa_{\gamma_{3,r}}(t)}=\frac{1}{\sqrt{2}},\] 
\[\tau_{\gamma_3^m}(t)=\frac{\kappa_{\gamma_{3,l}}(t)+\kappa_{\gamma_{3,r}}(t)}{\kappa_{\gamma_{3,l}}(t)-\kappa_{\gamma_{3,r}}(t)}=1. \]
Therefore the lifted logarithmic derivative and logarithmic derivative of $\gamma_1^m$ are constant and given by
\[ \tilde{\Lambda}_{\gamma_3^m}=\left(2\sqrt{2}m\pi\i+m\pi\k,m\pi\k\right) \]
and
\[  
\Lambda_{\gamma_3^m}=\pi
\begin{pmatrix}
0 & -m2\sqrt{2} & 0  & 0 \\
m2\sqrt{2} & 0 & -2m & 0  \\
0 & 2m & 0 & -m2\sqrt{2} \\
0 & 0 & m2\sqrt{2} & 0
\end{pmatrix}.
\]
The holonomic curve can be computed explicitly:
\[ \tilde{\Gamma}_{\gamma_3^m}(t)=\left(\exp\left(3m\pi t\frac{2\sqrt{2}\i+\k}{3}\right),\exp\left(m\pi t \k\right)\right). \]
The Jacobian curve $\Gamma_{\gamma_3^m}$ satisfies
\[ \Gamma_{\gamma_3^m}'(t)=\Gamma_{\gamma_3^m}(t)\Lambda_{\gamma_3^m}, \quad \Gamma_{\gamma_3^m}(0)=I \]
and can also be computed explicitly since it is the exponential of $\Gamma_{\gamma_3^m}$, that is
\[ \Gamma_{\gamma_3^m}(t)=\exp(t\Lambda_{\gamma_3^m}). \] 
The curve $\gamma_3^m$ is then equal to $\Gamma_{\gamma_3^m}e_1$, and we find that
\begin{eqnarray*}
\gamma_3^m(t) & = & \left(\frac{1}{3}\cos\left(4t\pi m\right)+\frac{2}{3}\cos\left(2t\pi m\right)\right., \\
&  & \frac{\sqrt{2}}{3}\sin\left(4t\pi m\right)+\frac{\sqrt{2}}{3}\sin\left(2t\pi m\right), \\
&  & \frac{\sqrt{2}}{3}\cos\left(2t\pi m\right)-\frac{\sqrt{2}}{3}\cos\left(4t\pi m\right), \\
&  & \left.\frac{2}{3}\sin\left(2t\pi m\right)-\frac{1}{3}\sin\left(4t\pi m\right)\right).
\end{eqnarray*}
Below we give an illustration in the case $m=1$ (Figure~\ref{fig:j}).

\begin{figure}[H]
\centering
\includegraphics[scale=0.5]{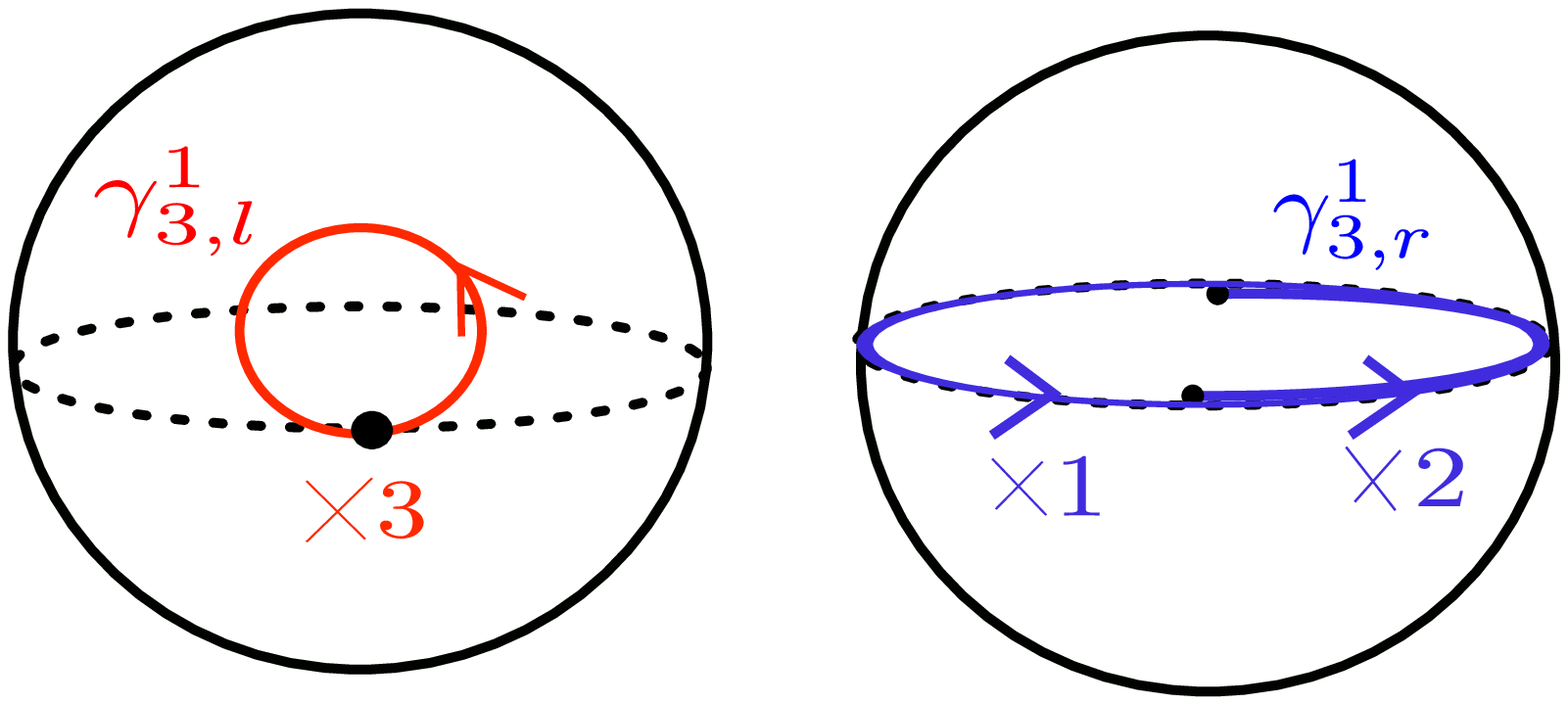}
\caption{The curve $\gamma_3^1$.}
\label{fig:j}
\end{figure}

\begin{example}[Spaces $\mathcal{L}\SS^3((-\1)^{3m},\1)$]\label{family4}
Let us give explicit examples in the spaces $\mathcal{L}\SS^3((-\1)^{3m},\1)$, $m \geq 1$. For $m\equiv 1$, $2$ modulo $2$, this will give examples in the spaces $\mathcal{L}\SS^3(-\1,\1)$ and $\mathcal{L}\SS^3(\1,\1)$.
\end{example}
We want to define a curve $\gamma_4^m \in \mathcal{L}\SS^3((-\1)^{3m},\1)$ such that its left and right part are given by
\[ \gamma_{4,l}^m=\sigma_c^{3m} \in \mathcal{L}\SS^2((-\1^{3m}), \quad \gamma_{4,r}^m =\sigma_{2\pi}^{2m} \in \mathcal{G}\SS^2(\1). \]
To define a pair of curves, we need to choose $0<c<2\pi$ such that
\[ ||(\gamma_{4,l}^{m})'(t)||=||(\sigma_c^{3m})'(t)||=3cm \]
is equal to 
\[ ||(\gamma_{4,r}^{m})'(t)||=||(\sigma_{2\pi}^{2m})'(t)||=4\pi m. \]
It suffices to choose $c=4\pi/3$ so that both curves have length equal to $4\pi m$, then the geodesic curvature of $\gamma_{4,l}^m=\sigma_c^{3m}$ is constantly equal to $\sqrt{5}/2$ while clearly, the geodesic curvature of $\gamma_{4,r}^m =\sigma_{2\pi}^{2m}$ is zero.  

As before, we can find explicitly the curve $\gamma_4^m$. From Theorem~\ref{th1}, we can compute
\[ ||(\gamma_4^{m})'(t)||=\frac{||(\gamma_{4,l}^m)'(t)||(\kappa_{\gamma_{4,l}}(t)-\kappa_{\gamma_{4,r}}(t))}{2}=\sqrt{5}\pi m \]
\[ \kappa_{\gamma_4^m}(t)=\frac{2}{\kappa_{\gamma_{4,l}}(t)-\kappa_{\gamma_{4,r}}(t)}=\frac{4}{\sqrt{5}},\] 
\[\tau_{\gamma_4^m}(t)=\frac{\kappa_{\gamma_{4,l}}(t)+\kappa_{\gamma_{4,r}}(t)}{\kappa_{\gamma_{4,l}}(t)-\kappa_{\gamma_{4,r}}(t)}=1. \]
Therefore the lifted logarithmic derivative and logarithmic derivative of $\gamma_4^m$ are constant and given by
\[ \tilde{\Lambda}_{\gamma_4^m}=\left(\sqrt{5}m\pi\i+2m\pi\k,2m\pi\k\right) \]
and
\[  
\Lambda_{\gamma_4^m}=\pi
\begin{pmatrix}
0 & m\sqrt{5} & 0  & 0 \\
m\sqrt{5} & 0 & -4m & 0  \\
0 & 4m & 0 & -m\sqrt{5} \\
0 & 0 & m\sqrt{5} & 0
\end{pmatrix}.
\]
The holonomic curve can be computed explicitly:
\[ \tilde{\Gamma}_{\gamma_4^m}(t)=\left(\exp\left(3m\pi t\frac{\sqrt{5}\i+2\k}{3}\right),\exp\left(2m\pi t \k\right)\right). \]
The Jacobian curve $\Gamma_{\gamma_4^m}$ satisfies
\[ \Gamma_{\gamma_4^m}'(t)=\Gamma_{\gamma_4^m}(t)\Lambda_{\gamma_4^m}, \quad \Gamma_{\gamma_4^m}(0)=I \]
and can also be computed explicitly since it is the exponential of $\Gamma_{\gamma_4^m}$, that is
\[ \Gamma_{\gamma_4^m}(t)=\exp(t\Lambda_{\gamma_4^m}). \] 
The curve $\gamma_4^m$ is then equal to $\Gamma_{\gamma_4^m}e_1$, and we find that
\begin{eqnarray*}
\gamma_4^m(t) & = & \left(\frac{1}{6}\cos\left(5t\pi m\right)+\frac{5}{6}\cos\left(t\pi m\right)\right., \\
&  & \frac{\sqrt{5}}{6}\sin\left(5t\pi m\right)+\frac{\sqrt{5}}{6}\sin\left(t\pi m\right), \\
&  & \frac{\sqrt{5}}{6}\cos\left(t\pi m\right)-\frac{\sqrt{5}}{6}\cos\left(5t\pi m\right), \\
&  & \left.\frac{5}{6}\sin\left(t\pi m\right)-\frac{1}{6}\sin\left(5t\pi m\right)\right).
\end{eqnarray*}
Below we give an illustration in the case $m=1$ (Figure~\ref{fig:k}).

\begin{figure}[H]
\centering
\includegraphics[scale=0.5]{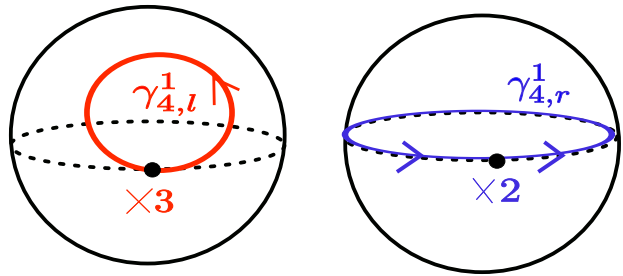}
\caption{The curve $\gamma_4^1$.}
\label{fig:k}
\end{figure}

\section{Proof of Theorem~\ref{th2} }\label{s63}

In this section, we give the proof of Theorem~\ref{th2} which characterizes convexity in the spaces $\mathcal{L}\SS^3(-\1,\k)$ by just looking at the left part of the curve.

\begin{proof}[{} of Theorem~\ref{th2}]
Recall that we want to prove that a curve $\gamma \in \mathcal{L}\SS^3(-\1,\k)$ is convex if and only its left part $\gamma_l \in \mathcal{L}\SS^2(-\1)$ is convex.

It is clear that $\mathcal{L}\SS^2(-\1)$ contains convex curves; the curve $\sigma_c$, for $0 < c <2\pi$ defined in~\S\ref{s62} is convex, since it intersects any hyperplane of $\R^3$ (or equivalently any great circle) in exactly two points. Using Theorem~\ref{thmani}, the space $\mathcal{L}\SS^2(-\mathbf{1})$ has therefore $2$ connected components,
\[ \mathcal{L}\SS^2(-\mathbf{1})=\mathcal{L}\SS^2(-\mathbf{1})_c \sqcup \mathcal{L}\SS^2(-\mathbf{1})_n \]
where $\mathcal{L}\SS^2(\mathbf{1})_c$ is the component associated to convex curves and $\mathcal{L}\SS^2(\mathbf{1})_n$ the component associated to non-convex curves.

The space $\mathcal{L}\SS^3(-\1,\k)$ also contains convex curves. Indeed, recall the family of curves $\gamma_{1}^m \in \mathcal{L}\SS^3((-\1)^m,\k^m)$, $m \geq 1$, defined in Example~\ref{family1}; we will prove that the curve $\gamma_{1}^1 \in \mathcal{L}\SS^3(-\1,\k)$ is convex. Up to a reparametrization with constant speed, this curve is the same as the curve $\tilde{\gamma} : [0,\pi/2] \rightarrow \SS^3$ defined by
\begin{eqnarray*}
\tilde{\gamma}(t) & = & \left(\frac{1}{4}\cos\left(3t\right)+\frac{3}{4}\cos\left(t\right)\right., \\
&  & \frac{\sqrt{3}}{4}\sin\left(3t \right)+\frac{\sqrt{3}}{4}\sin\left(t\right) \\
&  & \frac{\sqrt{3}}{4}\cos\left(t\right)-\frac{\sqrt{3}}{4}\cos\left(3t\right) \\
&  & \left.\frac{3}{4}\sin\left(t\right)-\frac{1}{4}\sin\left(3t \right)\right).
\end{eqnarray*} 
Since being convex is independent of the choice of a parametrization (Chapter~\ref{chapter3}, \S\ref{s32}, $(i)$ of Proposition~\ref{propconv}), it is sufficient to prove that $\tilde{\gamma}$ is convex. Observe that for $t \in [0,\pi/2)$, the first component of $\tilde{\gamma}$ never vanishes, so if we define the central projection
\[ p : (x_1,x_2,x_3,x_4) \in \R^4 \mapsto \left(1,\frac{x_2}{x_1},\frac{x_3}{x_1},\frac{x_4}{x_1}\right), \]
then it is sufficient to prove that the curve $p(\tilde{\gamma})$, defined for $t \in [0,\pi/2)$: this follows from $(ii)$ of Corollary~\ref{corconv}, Chapter~\ref{chapter3}, \S\ref{s32}. We compute
\[ p(\tilde{\gamma}(t))=\left(1,\sqrt{3}\tan t,\sqrt{3}(\tan t)^2,(\tan t)^3\right) \]
and hence, if we reparametrize by setting $x=\tan t$, we obtain the curve
\[ x \in [0,+\infty) \mapsto (1,\sqrt{3}x,\sqrt{3}x^2,x^3) \in \R^4. \]
It is now obvious that this curve is convex, and therefore our initial curve $\gamma_1^1$ is convex. As before, using Theorem~\ref{thmani}, the space $\mathcal{L}\SS^3(-\mathbf{1},\k)$ has therefore $2$ connected components,
\[ \mathcal{L}\SS^3(-\mathbf{1},\k)=\mathcal{L}\SS^3(-\mathbf{1},\k)_c \sqcup \mathcal{L}\SS^3(-\mathbf{1},\k)_n \]
where $\mathcal{L}\SS^3(\mathbf{1},\k)_c$ is the component associated to convex curves and $\mathcal{L}\SS^3(\mathbf{1},\k)_n$ the component associated to non-convex curves.

Then we can use Theorem~\ref{th1} to define a continuous map
\[ L : \mathcal{L}\SS^3(-\mathbf{1},\k) \rightarrow \mathcal{L}\SS^2(-\mathbf{1})  \]
by setting $L(\gamma)=\gamma_l$, where $(\gamma_l,\gamma_r)$ is the pair of curves associated to $\gamma$. Since $L$ is continuous and $\mathcal{L}\SS^3(-\mathbf{1},\k)_c$ is connected, its image by $L$ is also connected. Moreover, we know $\gamma_1^1 \in \mathcal{L}\SS^3(-\mathbf{1},\k)_c$, and that $L(\gamma_1^1)=\sigma_1 \in \mathcal{L}\SS^2(-\mathbf{1})_c$, therefore the image of $\mathcal{L}\SS^3(-\mathbf{1},\k)_c$ by $L$ intersects $\mathcal{L}\SS^2(-\mathbf{1})_c$; since the latter is connected we must have the inclusion
\[ L\left(\mathcal{L}\SS^3(-\mathbf{1},\k)_c\right) \subset \mathcal{L}\SS^2(-\mathbf{1})_c. \]
This proves one part of the statement, namely that if $\gamma \in \mathcal{L}\SS^3(-\mathbf{1},\k)_c$, then its left part $\gamma_l=L(\gamma) \in \mathcal{L}\SS^2(-\mathbf{1})_c$. To prove the other part, it is enough to prove that    
\[ L\left(\mathcal{L}\SS^3(-\mathbf{1},\k)_n\right) \subset \mathcal{L}\SS^2(-\mathbf{1})_n. \]
To prove this inclusion, using continuity and connectedness arguments as before, it is enough to find one element in $\mathcal{L}\SS^3(-\mathbf{1},\k)_n$ whose image by $L$ belongs to $\mathcal{L}\SS^2(-\mathbf{1})_n$. We claim that the curve $\gamma_1^5$ from Example~\ref{family1} does the job. To see that $\gamma_1^5 \in \mathcal{L}\SS^3(-\mathbf{1},\k)_n$, one can easily check that if we define the plane 
\[ H=\{(x_1,0,0,x_4) \in \R^4 \; | \; x_1 \in \R, \; x_4 \in \R\} \]  
then
\[ \gamma_1^5(t_i) \in H, \quad t_i=\frac{i}{5}, \quad 1 \leq i \leq 4. \]
Hence $\gamma_1^5$ has at least $4$ points of intersection with $H$; this shows that $\gamma_5$ is not convex. To conclude, it is clear that $L(\gamma_1^5)=\sigma_5 \in \mathcal{L}\SS^2(-\mathbf{1})_n$. Hence this proves the inclusion we wanted, and this concludes the proof.

\end{proof}  

\medskip

Let us observe that we used in the proof that the curve $\gamma_{1}^1 \in \mathcal{L}\SS^3(-\1,\k)$, defined in Example~\ref{family1}, is convex. It is easy to see that its final lifted Frenet frame, which is $(-\1,\k) \in \SS^3 \times \SS^3$ projects down to the transpose of Arnold matrix $A ^\top \in B_4^+$, and we will see later that this implies that this curve is in fact stably convex ($(-\1,\k) \in \tilde{B}_4^+$ is the only stably convex spin).

\section{Proof of Theorem ~\ref{th3}}\label{s64}

Let us now give the proof of Theorem~\ref{th3}, which gives a necessary condition for a curve in $\mathcal{L}\SS^3(\1,-\1)$ to be locally convex by looking at its left part. 

First we need to recall some basic definition and properties. An \emph{open hemisphere} $H$ in $\SS^2$ is a subset of $\SS^2$ of the form
\[ H_h=\{x \in \SS^2 \; | \; h\cdot x >0\} \]
for some $h \in \SS^2$, and a \emph{closed hemisphere} is the closure $\bar{H}$ of an open hemisphere, that is it has the form 
\[ \bar{H}_h=\{x \in \SS^2 \; | \; h\cdot x \geq 0\}. \]
We can make the following definitions.

\begin{definition}
A closed curve $\gamma : [0,1] \rightarrow \SS^2$ is \emph{hemispherical} if it its image is contained in an open hemisphere of $\SS^2$. It is \emph{borderline hemispherical} if it is contained in a closed hemisphere but not contained in any open hemisphere.
\end{definition}

Now following~\cite{Zul12}, we will define a rotation number for any closed curve $\gamma$ in $\SS^2$ contained in a closed hemisphere (such a curve is either hemispherical or borderline hemispherical). To such a closed curve $\gamma$ contained in a closed hemisphere, there is a distinguished choice of hemisphere $h_\gamma$ containing the image of $\gamma$ (this hemisphere $h_\gamma$ is the barycenter of the set of all closed hemisphere containing the image of $\gamma$, the latter being geodesically convex, see~\cite{Zul12} for further details). Let $\Pi_{h_\gamma} : \SS^2 \rightarrow \R^2$ be the stereographic projection from $-h_\gamma$, and $\eta_\gamma= \Pi_{h_\gamma} \circ \gamma$. The curve $\eta_\gamma$ is now a closed curve in the plane $\R^2$, and it is an immersion. The definition of its rotation number $\mathrm{rot}(\eta_\gamma) \in \Z$ is now classical: for instance, it can be defined to be the degree of the map
\[ t \in \SS^1 \mapsto \frac{\eta_\gamma'(t)}{||\eta_\gamma'(t)||} \in \SS^1. \]    

\begin{definition}\label{rot}
Given a closed curve contained in a closed hemisphere in $\SS^2$, its rotation number $\mathrm{rot}(\gamma)$ is defined by
\[ \mathrm{rot}(\gamma):=-\mathrm{rot}(\eta_\gamma) \in \Z. \]
\end{definition}

The proof of Theorem~\ref{th3} will be based on two lemmas. The first lemma is a well-known property, so we just state it without proof.

\begin{lemma}\label{lem1}
Consider a continuous map $H : [0,1] \rightarrow \mathcal{L}\SS^2(\1)$ such that $\gamma_0=H(0)$ has the property of being hemispherical with rotation number equal to $2$ and $\gamma_1=H(1)$ which does not have this property. Then there exists a time $t>0$ such that $\gamma_t=H(t)$ is borderline hemispherical with rotation number equal to $2$. 
\end{lemma} 

The second lemma will be proven below.

\begin{lemma}\label{lem2}
Consider the map $L : \mathcal{L}\SS^3(\1,-\1) \rightarrow \mathcal{L}\SS^2(\1)$ given by $L(\gamma)=\gamma_l$, and let $\mathcal{L}\SS^3(\1,-\1)_c$ be the set of convex curves. Then the image of $\mathcal{L}\SS^3(\1,-\1)_c$ by $L$ does not contain a borderline hemispherical curve with rotation number equal to $2$. 
\end{lemma} 

In fact, we believe that a stronger statement is true: the image of the whole space $\mathcal{L}\SS^3(\1,-\1)$ by $L$ does not contain a borderline hemispherical curve with rotation number equal to $2$. With this stronger statement it would be easy to see from the proof below that our necessary condition for a curve in $\mathcal{L}\SS^3(\1,-\1)$ to be convex is also sufficient. Yet for the moment we are not able to prove this stronger statement.   

Let us now see how these lemmas are used to prove Theorem~\ref{th3}. 

\begin{proof}[{} of Theorem~\ref{th3}]
Recall that the map $L : \mathcal{L}\SS^3(\1,-\1) \rightarrow \mathcal{L}\SS^2(\1)$ given by $L(\gamma)=\gamma_l$ is continuous, and that $\mathcal{L}\SS^3(\1,-\1)$ contains exactly two connected components, one of which is made of convex curves $\mathcal{L}\SS^3(\1,-\1)_c$ and the other of non-convex curves $\mathcal{L}\SS^3(\1,-\1)_n$.  We need to prove that the image of $\mathcal{L}\SS^3(\1,-\1)_c$ by $L$ contains only curves which are hemispherical with rotation number equal to $2$.

First let us prove that this image contains at least one such element. Recall the family of curves $\gamma_{1}^m \in \mathcal{L}\SS^3((-\1)^m,\k^m)$, $m \geq 1$, defined in Example~\ref{family1}. For $m=2$, $\gamma_1^2=(\sigma_\pi^2,\sigma^1_{2\pi}) \in \mathcal{L}\SS^3(\1,-\1)$ is convex; the proof of this assertion is entirely similar to the proof of the fact that $\gamma_1^1 \in \mathcal{L}\SS^3(-\1,\k)$ is convex (see the proof of Theorem~\ref{th2} in \S\ref{s63}). Moreover, it is clear that $L(\gamma_1^2)=\sigma_\pi^2$ is hemispherical and has rotation number equal to $2$, and therefore the image of  $\mathcal{L}\SS^3(\1,-\1)_c$ by $L$ contains at least the curve $L(\gamma_1^2)=\sigma_\pi^2$. 

To prove that the image of $\mathcal{L}\SS^3(\1,-\1)_c$ by $L$ contains only curves which are hemispherical with rotation number equal to $2$, we argue by contradiction, and assume that the image of $\mathcal{L}\SS^3(\1,-\1)_c$ by $L$ contains a curve which is not hemispherical with rotation number equal to $2$. Since $L$ is continuous and $\mathcal{L}\SS^3(\1,-\1)_c$ is connected, its image by $L$ is connected and thus we can find a homotopy $H : [0,1] \rightarrow L\left(\mathcal{L}\SS^3(\1,-\1)_c\right) \subset \mathcal{L}\SS^2(\1)$ between $H(0)=\sigma_\pi^2$, which is hemispherical with rotation number equal to $2$, and a curve $H(1)$ which does not have this property. Using Lemma~\ref{lem1}, one can find a time $t>0$ such that $H(t) \in L\left(\mathcal{L}\SS^3(\1,-\1)_c\right)$ is borderline hemispherical with rotation number equal to $2$. But by Lemma~\ref{lem2}, such a curve $H(t)$ cannot belong to $ L\left(\mathcal{L}\SS^3(\1,-\1)_c\right)$, and so we arrive at a contradiction.   
\end{proof}

To conclude, it remains to prove Lemma~\ref{lem2}.

\begin{proof}[{} of Lemma~\ref{lem2}]
We argue by contradiction, and assume that there exists a curve $\beta \in \mathcal{L}\SS^3(\1,-\1)_c$ (that is a convex curve $\beta \in \mathcal{L}\SS^3(\1,-\1)$) such that its left part $\beta_l$ is  borderline hemispherical curve with rotation number equal to $2$. 

First we use our assumption that $\beta$ is convex, which implies that $\mathcal{F}_{\beta}(t)$ belongs to the Bruhat cell of $A ^\top$ for all time $t \in [0,1]$. Therefore, by definition, there exist matrices $U_1(t) \in Up_4^+$, $U_2(t) \in Up_4^+$ (recall that $Up_4^+$ is the group of upper triangular $4$ by $4$ matrices with positive diagonal entries) such that
\[ \mathcal{F}_{\beta}(t)=U_1(t){} A ^\top U_2(t)\]
that can be also written as
\[ \mathcal{F}_{\beta}(t)={}  A ^\top L_1(t)U_2(t), \quad L_1(t):=AU_1(t){}A ^\top\]
and $L_1(t) \in Lo_4^+$, where $Lo_4^+$ is the group of lower triangular $4$ by $4$ matrices with positive diagonal entries. Such a decomposition is not unique, but there exists a unique decomposition
\begin{equation}\label{dec}
\mathcal{F}_{\beta}(t)={}  A ^\top L(t)U(t)
\end{equation}
where $U(t) \in Up_4^+$, but this time $L(t)\in Lo_4^1$, where $Lo_4^1$ is the group of lower triangular $4$ by $4$ matrices with diagonal entries equal to one. Using the fact that $\mathcal{F}_{\beta}(t)^{-1}\mathcal{F}_{\beta}'(t)$ belongs to $\mathfrak{J}$ (because $\beta$ is in particular locally convex), it is easy to see, by a simple computation, that the matrix $L(t)$ in~\eqref{dec} is such that $L(t)^{-1}L'(t)$ has positive subdiagonal entries and all other entries are zero, that is we can write
\begin{equation}\label{L}
L(t)^{-1}L'(t)=
\begin{pmatrix}
    0 & 0 & 0 &  0 \\
    + & 0 & 0 &  0 \\
    0 & + & 0 &  0 \\
    0 & 0 & + &  0
\end{pmatrix}
, \quad t \in [0,1].
\end{equation}

Then we use our assumption that the left part $\beta_l$ is  borderline hemispherical curve with rotation number equal to $2$. 
\begin{figure}[H]
\centering
\includegraphics[scale=0.3]{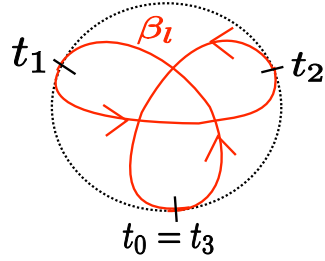}
\caption{The curve $\beta_l$.}
\label{fig:l}
\end{figure}
This implies (see Figure~\ref{fig:l}, where the dotted circle represent the equator of the sphere) that there exist times $t_1$ and $t_2$ and reals $\theta_1$ and $\theta_2$ such that
\[ \tilde{\mathcal{F}}_{\beta_l}(t_1)=\exp(\theta_1\k) \in \SS^3, \quad \tilde{\mathcal{F}}_{\beta_l}(t_2)=\exp(\theta_2\k) \in \SS^3\]
and consequently, for $\beta$, we have
\begin{equation}\label{assum}
\begin{cases}
\tilde{\mathcal{F}}_{\beta}(t_1)=(\exp(\theta_1\k),z_r(t_1)) \in \SS^3 \times \SS^3 \\
\tilde{\mathcal{F}}_{\beta}(t_2)=(\exp(\theta_2\k),z_r(t_2)) \in \SS^3 \times \SS^3.
\end{cases}
\end{equation}
Recalling that $\Pi_4 : \SS^3 \times \SS^3 \rightarrow SO_4$ is the canonical projection, it follows from~\eqref{assum} that $\mathcal{F}_{\beta}(t_1)=\Pi_4(\tilde{\mathcal{F}}_{\beta}(t_1))$ and $\mathcal{F}_{\beta}(t_2)=\Pi_4(\tilde{\mathcal{F}}_{\beta}(t_2))$ belong to the subgroup $H$ of matrices in $SO_4$ that commutes with the matrix $\k_l$ defined in~\S\ref{s61}, which is   
\[ \k_l=
\begin{pmatrix}
0 & 0 & 0 & -1 \\
0 & 0 & -1 & 0 \\
0 & +1 & 0  & 0 \\
+1 & 0 & 0 & 0
\end{pmatrix}. \]
Clearly, this subgroup $H$ consists of matrices of the form
\[
\begin{pmatrix}
q_{11} & q_{12} & -q_{42} & -q_{41} \\
q_{21} & q_{22} & -q_{32} & -q_{31} \\
q_{31} & q_{32} & q_{22}  & q_{21} \\
q_{41} & q_{42} & q_{12} & q_{11}
\end{pmatrix} \in SO_4. \]
Using this explicit form of $H$ and the fact that $\mathcal{F}_{\beta}(t_1) \in H$ and $\mathcal{F}_{\beta}(t_1) \in H$, one finds, after a direct computation, that the matrix 
\begin{equation}
L(t)=
\begin{pmatrix}
1 & 0 & 0 & 0 \\
l_{21}(t) & 1 & 0 & 0 \\
l_{31}(t) & l_{32}(t) & 1  & 0 \\
l_{41}(t) & l_{42}(t) & l_{43}(t) & 1
\end{pmatrix}
\end{equation}
defined in~\eqref{dec} satisfies, at $t=t_1$ and $t=t_2$, the conditions
\begin{equation}\label{contra}
l_{21}(t_1)=-l_{43}(t_1), \quad l_{21}(t_2)=-l_{43}(t_2). 
\end{equation}
But clearly, \eqref{contra} is not compatible with~\eqref{L}, and this gives the desired contradiction.  

\end{proof}


%
%

\chapter{Convex arcs and Bruhat cells}
\label{chapter7}

The aim of this Chapter is to explain the link between the convexity of arcs and the Bruhat decomposition on $SO_{n+1}$ and $\mathrm{Spin}_{n+1}$. In~\S\ref{s71}, we recall some well-known general facts about convex matrices and convex spins. In~\S\ref{s72} we prove Theorem~\ref{th4}: that is we determine, for $n=3$, the complete list of all convex matrices and spins, and the next open cell in which they immediately enter (this list is much shorter and well-known in the case $n=2$). The main tool for this classification of matrices is the positive and negative chopping operations that were introduced in Chapter~\ref{chapter5}. For spins this is more complicated, and we will strongly use the fact that locally convex curves in $\SS^3$ can be decomposed as a pair of curves on $\SS^2$, that is Theorem~\ref{th1}, and the examples we gave in Chapter~\ref{chapter6}, \S\ref{s62}. Finally, in~\S\ref{s73}, we prove that the spaces $\mathcal{L}\SS^3(\1,\1), \mathcal{L}\SS^3(-\1,-\1), \mathcal{L}\SS^3(\1,-\1)$ and $\mathcal{L}\SS^3(-\1,\1)$ are respectively homeomorphic to $\mathcal{L}\SS^3(-\1,-\k),  \mathcal{L}\SS^3(\1,\k), \mathcal{L}\SS^3(-\1,\k)$ and $\mathcal{L}\SS^3(\1,-\k)$.

\section{Characterization of convex arcs}\label{s71}

We start with some generalities, that are valid for any $n\geq 2$. Consider a smooth Jacobian curve
\[ \Gamma : [-1,1] \rightarrow SO_{n+1} \]
and a smooth holonomic curve
\[ \tilde{\Gamma} : [-1,1] \rightarrow \mathrm{Spin}_{n+1}. \]
Recall that this is the same thing as looking at the Frenet frame curve (or lifted Frenet frame curve) of a locally convex curve, but for the moment, we do not require the curve to be defined on $[0,1]$ (we want a symmetric interval since we will look at the past $t<0$ and at the future $t>0$ of $t=0$) and we do not require $\Gamma(0)=I$. We will assume these two conditions later on. 

If $\Gamma(0)$ (respectively $\tilde{\Gamma}(0)$) belongs to an open Bruhat cell, then there exists $\varepsilon>0$ such that $\Gamma(t)$ (respectively $\tilde{\Gamma}(t)$) stays in this Bruhat cell for all $t \in (-\varepsilon,\varepsilon)$. If $\Gamma(0)$ (respectively $\tilde{\Gamma}(0)$) belongs to a lower dimensional Bruhat cell, we can also determine the immediate future and past of the curve. 

\begin{proposition}
Assume that $\Gamma(t)$ (respectively $\tilde{\Gamma}(t)$) belongs to a lower dimensional Bruhat cell. Then there exists $\varepsilon>0$ such that $\Gamma(t)$ (respectively $\tilde{\Gamma}(t)$) belongs to $\Gamma(0)\mathrm{Bru}_{ {} A ^\top}$ (respectively $\tilde{\Gamma}(0)\mathrm{Bru}_{ {} \bar{\a}}$) for all $t \in (0,\varepsilon)$, and to $\Gamma(0)\mathrm{Bru}_{A}$ (respectively $\tilde{\Gamma}(0)\mathrm{Bru}_{\mathbf{a}}$) for all $t \in (-\varepsilon,0)$.  
\end{proposition}

\begin{proof}
Indeed, the curve $\Gamma(0)^{-1}\Gamma(t)$ is still Jacobian and equals the identity for $t=0$. Using the positive and negative chopping operations, Proposition~\ref{propchop} and Proposition~\ref{propchop2}, and the fact that $\mathbf{chop}^+(I)={} A ^\top$ and $\mathbf{chop}^-(I)=A$, we obtain
\[ \Gamma(0)^{-1}\Gamma(t) \in \mathrm{Bru}_{ {} A^\top}, \quad t \in (0,\varepsilon)  \]
and  
\[ \Gamma(0)^{-1}\Gamma(t) \in \mathrm{Bru}_{A}, \quad t \in (-\varepsilon,0). \]
In the spin case, the argument is similar. 
\end{proof}

\medskip

From now on, assume that $\Gamma(0)=I$ (respectively $\tilde{\Gamma}(0)=\mathbf{1}$). For small positive $t$ (respectively negative $t$), $\Gamma(t)$ belongs to the cell of $A ^\top$ (respectively to the cell of $A$). For small positive $t$ (respectively negative $t$), one can also show that $\Gamma$ is globally Jacobian. Indeed, we can write $\Gamma=\mathcal{F}_\gamma$ for some locally convex curve $\gamma$ defined on some interval around $0$. Since $\Gamma(0)=\mathcal{F}_\gamma(0)=I$, by a Taylor expansion we have that for $|t|$ arbitrarily small, $\gamma(t)$, considered as a curve in $\R^{n+1}$, is arbitrarily close to the curve $t \mapsto (1,t,\dots,t^n)$, and this curve is convex. Since being convex is an open property, this shows that for $|t|$ sufficiently small, $\gamma$ is convex, that is $\Gamma$ is globally Jacobian.    

This in fact a general phenomenon: a Jacobian curve is globally Jacobian if and only if its future (respectively its past) stays in the cell of $A ^\top$ (respectively in the cell of $A$). Let us state this as a theorem (see~\cite{Ani98}).

\begin{theorem}\label{ani}
Consider a smooth Jacobian curve $\Gamma : [-1,1] \rightarrow SO_{n+1}$ with $\Gamma(0)=I$. Then $\Gamma$ is globally Jacobian if and only $\Gamma(t)$ belongs to the cell of $A ^\top$ for all $t \in (0,1)$ and to the cell of $A$ for all $t \in (-1,0)$. 
\end{theorem}  

In fact, it is easy to prove that $\Gamma(t)$ belongs to the cell of $A ^\top$ for all $t \in (0,1)$ if and only $\Gamma(t)$ belongs to the cell of $A$ for all $t \in (-1,0)$: to show this, one can use the time reversal $\mathbf{TR}^*$ defined in Chapter~\ref{chapter5}, \S\ref{s53}. 

The same result holds in the spin case.

\begin{theorem}\label{ani2}
Consider a smooth holonomic curve $\tilde{\Gamma} : [-1,1] \rightarrow \mathrm{Spin}_{n+1}$ with $\tilde{\Gamma}(0)=\mathbf{1}$. Then $\tilde{\Gamma}$ is globally holonomic if and only $\tilde{\Gamma}(t)$ belongs to the cell of $\bar{\a}$ for all $t \in (0,1)$ and to the cell of $\mathbf{a}$ for all $t \in (-1,0)$. 
\end{theorem} 

This result explains in particular why closed convex curves can exist only on an even dimensional sphere, that is when $n+1$ is odd. When $n+1$ is odd, such closed convex curves exist as we already gave an example in Chapter~\ref{chapter3}, Example~\ref{exclosed}. Now assume we have a closed convex curve: its future should coincide with its past, hence in view of Theorem~\ref{ani}, we should have $A ^\top=A$, but this is true if and only if $n+1$ is odd.  

From now on, we assume that $\Gamma : [0,1] \rightarrow SO_{n+1}$, so that $\Gamma=\mathcal{F}_\gamma$ for an element $\gamma \in \mathcal{L}\SS^n$. From what we explained above, $\Gamma(t)=\mathcal{F}_\gamma(t)$ stays in the cell of $A ^\top$ for small positive $t$, and eventually loose convexity at some time $t_1>0$. Recalling the definition of stably convex and convex elements in $SO_{n+1}$, we easily identify stably convex matrices with the Bruhat cell of $A ^\top$, and convex matrices with the matrices of the form $\Gamma(t_1)=\mathcal{F}_\gamma(t_1)$. Convex matrices are a disjoint union of one open cell (the cell of stably convex matrices) and lower dimensional cells. There are $2^n((n+1)!-1)$ lower dimensional Bruhat cells in total, but the number of convex (but not stably convex) is much smaller.

\begin{proposition}\label{numberconvexm}
The set of convex matrices is a disjoint union of one open cell, and $(n+1)!-1$ closed cells.
\end{proposition}

\begin{proof}
Let $Q \in B_{n+1}^+$ be a representative of a convex Bruhat cell. From Proposition~\ref{propchop}, it has to satisfy
\[ \mathbf{chop_\varepsilon^-}(Q)=\Delta^-(Q)A={} A ^\top \]
since its past has to lie in the open cell of $A ^\top$. This fixes the diagonal matrix $\Delta^-(Q)$, and hence the sign of the signed permutation matrix. Therefore there are $(n+1)!$ possibilities for $Q$, and hence $(n+1)!$ cells, among which $(n+1)!-1$ lower dimensional cells.
\end{proof}

\medskip

For spins, the number of Bruhat cells is twice the number of Bruhat cells for matrices, so there could two times more possibilities for convex Bruhat cells. However, we have distinguished the pre-images $\mathbf{a}$ and $\bar{\a}$ of respectively $A$ and $A ^\top$ by the condition
\[ \mathbf{a}=\mathbf{chop^-}(\mathbf{1}), \quad \bar{\a}=\mathbf{chop^+}(\mathbf{1}) 
\] 
where $\mathbf{1} \in \mathrm{Spin}_{n+1}$ is the unit of the group. Hence the situation for spins is in fact similar.

\begin{proposition}\label{numberconvexs}
The set of convex spins is a disjoint union of one open cell, and $(n+1)!-1$ closed cells.
\end{proposition}

The set of convex matrices and spins is known for $n=2$ (see~\cite{KS99},~\cite{Sal09I},~\cite{Sal09II} and~\cite{Sal13}). In the next section, we will give an explicit list of convex matrices and spins in the case $n=3$. We expect to be useful in order to full understanding the homotopy type of the spaces $\mathcal{L}\SS^3$.

\section{Proof of Theorem~\ref{th4}}\label{s72}

Let us recall that $\mathrm{Spin}_{4} \simeq \SS^3 \times \SS^3$, so that spins are identified to a pair of unit quaternions.  

The Arnold matrix and its transpose are given by
\[
A=
\begin{pmatrix}
0 & 0 & 0 & 1 \\
0 & 0 & -1 & 0 \\
0 & 1 & 0 & 0 \\
-1 & 0 & 0 & 0 
\end{pmatrix}
\quad
A ^\top=
\begin{pmatrix}
0 & 0 & 0 & -1 \\
0 & 0 & 1 & 0 \\
0 & -1 & 0 & 0 \\
1 & 0 & 0 & 0 
\end{pmatrix}.
\]  
Recall that $\mathbf{1} \in \SS^3$ is the quaternion equal to $1$; this is the unit element of $\mathrm{Spin}_3$, and the unit element in $\SS^3 \times \SS^3$ is $(\mathbf{1},\mathbf{1})$. Let us first determine $\a$ and $\bar{\a}$.

\begin{proposition}\label{aa}
We have $\mathbf{a}=(-\1,-\k)$ and $\bar{\a}=(-\1,\k)$.
\end{proposition}

\begin{proof}
Let us consider the closed curve $\gamma_1^4 \in \mathcal{L}\SS^3(\1,\1)$ defined in Example~\ref{family1}, Chapter~\ref{chapter6}, \S\ref{s62}. It is easy to see that the lifted Frenet frame (or lifted Jacobian curve) $\tilde{\mathcal{F}}_{\gamma_1^4}(t)=\tilde{\Gamma}_{\gamma_1^4}(t)$ satisfies $\tilde{\mathcal{F}}_{\gamma_1^4}(1/4)=\tilde{\Gamma}_{\gamma_1^4}(1/4)=(-\1,\k)$ and that for all $t \in (0,1/2)$, $\tilde{\mathcal{F}}_{\gamma_1^4}(t)=\tilde{\Gamma}_{\gamma_1^4}(t)$ is Bruhat equivalent to $(-\1,\k)$. Hence
\[ \mathbf{chop}^+(\1,\1)={} \bar{\a}=(-\1,\k). \]
Similarly, $\tilde{\mathcal{F}}_{\gamma_1^4}(3/4)=\tilde{\Gamma}_{\gamma_1^4}(3/4)=(-\1,-\k)$ and for all $t \in (1/2,1)$, $\tilde{\mathcal{F}}_{\gamma_1^4}(t)=\tilde{\Gamma}_{\gamma_1^4}(t)$ is Bruhat equivalent to $(-\1,-\k)$, and therefore
\[ \mathbf{chop}^-(\1,\1)=\mathbf{a}=(-\1,-\k). \]
\end{proof}

\medskip

Next let us look at convex matrices. Bruhat cells are parametrized by signed permutation matrices. Permutations $\pi \in S_4$ will be given as a product of cycles. For example $\pi=(234)$ is the permutation that sends $2$ to $3$, $3$ to $4$, $4$ to $2$ and fixes $1$. The associated permutation matrix $P_{(234)}$ is
\begin{equation*}
P_{(234)}=
\begin{pmatrix}
1 & 0 & 0 & 0 \\
0 & 0 & 0 & 1 \\
0 & 1 & 0 & 0 \\
0 & 0 & 1 & 0 \\
\end{pmatrix}
\end{equation*} 
The identity permutation will be simply denoted by $e$. Then we need to introduce signs in our permutation matrices. To do this, we make a list of all the possible signs (there are 16 of them):
\[ 0=++++, \quad 1=+++-, \quad 2=++-+, \quad 3=++--, \]
\[ 4=+-++, \quad 5=+-+-, \quad 6=+--+, \quad 7=+---, \]
\[ 8=-+++, \quad 9=-++-, \quad 10=-+-+, \quad 11=-+--, \]
\[ 12=--++, \quad 13=--+-, \quad 14=---+, \quad 15=----. \]
For instance, the signed permutation matrix $P_{(234);9}$ is then
\begin{equation*}
P_{(234);9}=
\begin{pmatrix}
-1 & 0 & 0 & 0 \\
0 & 0 & 0 & -1 \\
0 & 1 & 0 & 0 \\
0 & 0 & 1 & 0 \\
\end{pmatrix}
\end{equation*} 
We can now make the list of convex matrices and spins. We will order them by the dimension of the Bruhat cells. There will be $1$ cell of dimension $0$, $3$ cells of dimension $1$, $5$ cells of dimension $2$, $6$ cell of dimension $3$, $5$ cells of dimension $4$, $3$ cells of dimension $5$ and $1$ cell of dimension $6$. 

The cell of dimension $0$ is the following one:
\begin{equation*}
\begin{tabular}{ | c  c  c | } 
\hline
$(\1,-\1)$ & $\longrightarrow$  & $(-\1,-\k)$ \\  
\hline
$ P_{e;15}=
\begin{pmatrix}
-1 & 0 & 0 & 0 \\
0 & -1 & 0 & 0 \\
0 & 0 & -1 & 0 \\
0 & 0 & 0 & -1 \\
\end{pmatrix}$ 
& $\longrightarrow$ &    
$A=\begin{pmatrix}
0 & 0 & 0 & 1 \\
0 & 0 & -1 & 0 \\
0 & 1 & 0 & 0 \\
-1 & 0 & 0 & 0 \\
\end{pmatrix}$ \\
\hline
\end{tabular}  
\end{equation*}
On the left hand side, we have indicated the convex spin on the top, and below the associated convex matrix. On the right, we have indicated the next open cell the curve enters, both for the spin case and for the matrix case.

The $3$ cells of dimension $1$ are as follows:
\begin{equation*}
\begin{tabular}{ | c  c  c | } 
\hline
$(\frac{\1-\i}{\sqrt{2}},\frac{-\1-\i}{\sqrt{2}})$ & $\longrightarrow$  & $(\i,\j)$ \\  
\hline
$ P_{(12);7}=
\begin{pmatrix}
0 & -1 & 0 & 0 \\
1 & 0 & 0 & 0 \\
0 & 0 & -1 & 0 \\
0 & 0 & 0 & -1 \\
\end{pmatrix}$ 
& $\longrightarrow$ &    
$\begin{pmatrix}
0 & 0 & 0 & -1 \\
0 & 0 & 1 & 0 \\
0 & 1 & 0 & 0 \\
-1 & 0 & 0 & 0 \\
\end{pmatrix}$ \\
\hline
\end{tabular}  
\end{equation*}
\begin{equation*}
\begin{tabular}{ | c  c  c | } 
\hline
$(\frac{\1-\k}{\sqrt{2}},\frac{-\1+\k}{\sqrt{2}})$ & $\longrightarrow$  & $(\k,-\1)$ \\  
\hline
$ P_{(23);11}=
\begin{pmatrix}
-1 & 0 & 0 & 0 \\
0 & 0 & -1 & 0 \\
0 & 1 & 0 & 0 \\
0 & 0 & 0 & -1 \\
\end{pmatrix}$ 
& $\longrightarrow$ &    
$\begin{pmatrix}
0 & 0 & 0 & 1 \\
0 & 0 & 1 & 0 \\
0 & -1 & 0 & 0 \\
-1 & 0 & 0 & 0 \\
\end{pmatrix}$ \\
\hline
\end{tabular}  
\end{equation*}
\begin{equation*}
\begin{tabular}{ | c  c  c | } 
\hline
$(\frac{\1-\i}{\sqrt{2}},\frac{-\1+\i}{\sqrt{2}})$ & $\longrightarrow$  & $(\i,-\j)$ \\  
\hline
$ P_{(34);13}=
\begin{pmatrix}
-1 & 0 & 0 & 0 \\
0 & -1 & 0 & 0 \\
0 & 0 & 0 & -1 \\
0 & 0 & 1 & 0 \\
\end{pmatrix}$ 
& $\longrightarrow$ &    
$\begin{pmatrix}
0 & 0 & 0 & 1 \\
0 & 0 & -1 & 0 \\
0 & -1 & 0 & 0 \\
1 & 0 & 0 & 0 \\
\end{pmatrix}$ \\
\hline
\end{tabular}  
\end{equation*}
Here are the $5$ cells of dimension $2$:
\begin{equation*}
\begin{tabular}{ | c  c  c | } 
\hline
$(\frac{\1-\i+\j-\k}{2},\frac{-\1+\i-\j+\k}{2})$ & $\longrightarrow$  & $(\i,-\j)$ \\  
\hline
$ P_{(234);9}=
\begin{pmatrix}
-1 & 0 & 0 & 0 \\
0 & 0 & 0 & -1 \\
0 & 1 & 0 & 0 \\
0 & 0 & 1 & 0 \\
\end{pmatrix}$ 
& $\longrightarrow$ &    
$\begin{pmatrix}
0 & 0 & 0 & 1 \\
0 & 0 & -1 & 0 \\
0 & -1 & 0 & 0 \\
1 & 0 & 0 & 0 \\
\end{pmatrix}$ \\
\hline
\end{tabular}  
\end{equation*}
\begin{equation*}
\begin{tabular}{ | c  c  c | } 
\hline
$(\frac{\1-\i-\j-\k}{2},\frac{-\1+\i+\j+\k}{2})$ & $\longrightarrow$  & $(\k,-\1)$ \\  
\hline
$ P_{(243);15}=
\begin{pmatrix}
-1 & 0 & 0 & 0 \\
0 & 0 & -1 & 0 \\
0 & 0 & 0 & -1 \\
0 & -1 & 0 & 0 \\
\end{pmatrix}$ 
& $\longrightarrow$ &    
$\begin{pmatrix}
0 & 0 & 0 & 1 \\
0 & 0 & 1 & 0 \\
0 & -1 & 0 & 0 \\
-1 & 0 & 0 & 0 \\
\end{pmatrix}$ \\
\hline
\end{tabular}  
\end{equation*}
\begin{equation*}
\begin{tabular}{ | c  c  c | } 
\hline
$(\frac{\1-\i-\j-\k}{2},\frac{-\1-\i-\j+\k}{2})$ & $\longrightarrow$  & $(\k,-\1)$ \\  
\hline
$ P_{(123);3}=
\begin{pmatrix}
0 & 0 & -1 & 0 \\
1 & 0 & 0 & 0 \\
0 & 1 & 0 & 0 \\
0 & 0 & 0 & -1 \\
\end{pmatrix}$ 
& $\longrightarrow$ &    
$\begin{pmatrix}
0 & 0 & 0 & 1 \\
0 & 0 & 1 & 0 \\
0 & -1 & 0 & 0 \\
-1 & 0 & 0 & 0 \\
\end{pmatrix}$ \\
\hline
\end{tabular}  
\end{equation*}
\begin{equation*}
\begin{tabular}{ | c  c  c | } 
\hline
$(\frac{\1-\i+\j-\k}{2},\frac{-\1-\i+\j+\k}{2})$ & $\longrightarrow$  & $(\i,\j)$ \\  
\hline
$ P_{(132);15}=
\begin{pmatrix}
0 & -1 & 0 & 0 \\
0 & 0 & -1 & 0 \\
-1 & 0 & 0 & 0 \\
0 & 0 & 0 & -1 \\
\end{pmatrix}$ 
& $\longrightarrow$ &    
$\begin{pmatrix}
0 & 0 & 0 & -1 \\
0 & 0 & 1 & 0 \\
0 & 1 & 0 & 0 \\
-1 & 0 & 0 & 0 \\
\end{pmatrix}$ \\
\hline
\end{tabular}  
\end{equation*}
\begin{equation*}
\begin{tabular}{ | c  c  c | } 
\hline
$(-\i,-\1)$ & $\longrightarrow$  & $(\1,-\k)$ \\  
\hline
$ P_{(12)(34);5}=
\begin{pmatrix}
0 & -1 & 0 & 0 \\
1 & 0 & 0 & 0 \\
0 & 0 & 0 & -1 \\
0 & 0 & 1 & 0 \\
\end{pmatrix}$ 
& $\longrightarrow$ &    
$A ^\top = \begin{pmatrix}
0 & 0 & 0 & -1 \\
0 & 0 & 1 & 0 \\
0 & -1 & 0 & 0 \\
1 & 0 & 0 & 0 \\
\end{pmatrix}$ \\
\hline
\end{tabular}  
\end{equation*}
Then, we have the $6$ cells of dimension $3$:
\begin{equation*}
\begin{tabular}{ | c  c  c | } 
\hline
$(\frac{-\i-\k}{\sqrt{2}},\frac{-\i+\k}{\sqrt{2}})$ & $\longrightarrow$  & $(\1,\k)$ \\  
\hline
$ P_{(13);11}=
\begin{pmatrix}
0 & 0 & -1 & 0 \\
0 & 1 & 0 & 0 \\
-1 & 0 & 0 & 0 \\
0 & 0 & 0 & -1 \\
\end{pmatrix}$ 
& $\longrightarrow$ &    
$A=\begin{pmatrix}
0 & 0 & 0 & 1 \\
0 & 0 & -1 & 0 \\
0 & 1 & 0 & 0 \\
-1 & 0 & 0 & 0 \\
\end{pmatrix}$ \\
\hline
\end{tabular}  
\end{equation*}
\begin{equation*}
\begin{tabular}{ | c  c  c | } 
\hline
$(\frac{-\i-\k}{\sqrt{2}},\frac{\i+\k}{\sqrt{2}})$ & $\longrightarrow$  & $(\1,\k)$ \\  
\hline
$ P_{(24);13}=
\begin{pmatrix}
-1 & 0 & 0 & 0 \\
0 & 0 & 0 & -1 \\
0 & 0 & 1 & 0 \\
0 & -1 & 0 & 0 \\
\end{pmatrix}$ 
& $\longrightarrow$ &    
$A=\begin{pmatrix}
0 & 0 & 0 & 1 \\
0 & 0 & -1 & 0 \\
0 & 1 & 0 & 0 \\
-1 & 0 & 0 & 0 \\
\end{pmatrix}$ \\
\hline
\end{tabular}  
\end{equation*}
\begin{equation*}
\begin{tabular}{ | c  c  c | } 
\hline
$(\frac{-\i-\k}{\sqrt{2}},\frac{-\1-\j}{\sqrt{2}})$ & $\longrightarrow$  & $(\1,-\k)$ \\  
\hline
$ P_{(1234);1}=
\begin{pmatrix}
0 & 0 & 0 & -1 \\
1 & 0 & 0 & 0 \\
0 & 1 & 0 & 0 \\
0 & 0 & 1 & 0 \\
\end{pmatrix}$ 
& $\longrightarrow$ &    
$A ^\top =\begin{pmatrix}
0 & 0 & 0 & -1 \\
0 & 0 & 1 & 0 \\
0 & -1 & 0 & 0 \\
1 & 0 & 0 & 0 \\
\end{pmatrix}$ \\
\hline
\end{tabular}  
\end{equation*}
\begin{equation*}
\begin{tabular}{ | c  c  c | } 
\hline
$(\frac{-\i-\k}{\sqrt{2}},\frac{-\1+\j}{\sqrt{2}})$ & $\longrightarrow$  & $(\1,-\k)$ \\  
\hline
$ P_{(1432);7}=
\begin{pmatrix}
0 & -1 & 0 & 0 \\
0 & 0 & -1 & 0 \\
0 & 0 & 0 & -1 \\
1 & 0 & 0 & 0 \\
\end{pmatrix}$ 
& $\longrightarrow$ &    
$A ^\top = \begin{pmatrix}
0 & 0 & 0 & -1 \\
0 & 0 & 1 & 0 \\
0 & -1 & 0 & 0 \\
1 & 0 & 0 & 0 \\
\end{pmatrix}$ \\
\hline
\end{tabular}  
\end{equation*}
\begin{equation*}
\begin{tabular}{ | c  c  c | } 
\hline
$(\frac{-\i-\j}{\sqrt{2}},\frac{-\1+\k}{\sqrt{2}})$ & $\longrightarrow$  & $(\k,-\1)$ \\  
\hline
$ P_{(1243);7}=
\begin{pmatrix}
0 & 0 & -1 & 0 \\
1 & 0 & 0 & 0 \\
0 & 0 & 0 & -1 \\
0 & -1 & 0 & 0 \\
\end{pmatrix}$ 
& $\longrightarrow$ &    
$\begin{pmatrix}
0 & 0 & 0 & 1 \\
0 & 0 & 1 & 0 \\
0 & -1 & 0 & 0 \\
-1 & 0 & 0 & 0 \\
\end{pmatrix}$ \\
\hline
\end{tabular}  
\end{equation*}
\begin{equation*}
\begin{tabular}{ | c  c  c | } 
\hline
$(\frac{-\i+\j}{\sqrt{2}},\frac{-\1+\k}{\sqrt{2}})$ & $\longrightarrow$  & $(-\k,-\1)$ \\  
\hline
$ P_{(1342);13}=
\begin{pmatrix}
0 & -1 & 0 & 0 \\
0 & 0 & 0 & -1 \\
-1 & 0 & 0 & 0 \\
0 & 0 & 1 & 0 \\
\end{pmatrix}$ 
& $\longrightarrow$ &    
$\begin{pmatrix}
0 & 0 & 0 & -1 \\
0 & 0 & -1 & 0 \\
0 & 1 & 0 & 0 \\
1 & 0 & 0 & 0 \\
\end{pmatrix}$ \\
\hline
\end{tabular}  
\end{equation*}
Here are the $5$ cells of dimension $4$:
\begin{equation*}
\begin{tabular}{ | c  c  c | } 
\hline
$(\frac{-\1-\i+\j-\k}{2},\frac{-\1-\i-\j+\k}{2})$ & $\longrightarrow$  & $(-\k,-\1)$ \\  
\hline
$ P_{(134);9}=
\begin{pmatrix}
0 & 0 & 0 & -1 \\
0 & 1 & 0 & 0 \\
-1 & 0 & 0 & 0 \\
0 & 0 & 1 & 0 \\
\end{pmatrix}$ 
& $\longrightarrow$ &    
$\begin{pmatrix}
0 & 0 & 0 & -1 \\
0 & 0 & -1 & 0 \\
0 & 1 & 0 & 0 \\
1 & 0 & 0 & 0 \\
\end{pmatrix}$ \\
\hline
\end{tabular}  
\end{equation*}
\begin{equation*}
\begin{tabular}{ | c  c  c | } 
\hline
$(\frac{-\1-\i-\j-\k}{2},\frac{-\1+\i-\j+\k}{2})$ & $\longrightarrow$  & $(-\i,-\j)$ \\  
\hline
$ P_{(124);5}=
\begin{pmatrix}
0 & 0 & 0 & -1 \\
1 & 0 & 0 & 0 \\
0 & 0 & 1 & 0 \\
0 & -1 & 0 & 0 \\
\end{pmatrix}$ 
& $\longrightarrow$ &    
$\begin{pmatrix}
0 & 0 & 0 & -1 \\
0 & 0 & 1 & 0 \\
0 & 1 & 0 & 0 \\
-1 & 0 & 0 & 0 \\
\end{pmatrix}$ \\
\hline
\end{tabular}  
\end{equation*}
\begin{equation*}
\begin{tabular}{ | c  c  c | } 
\hline
$(-\i,\k)$ & $\longrightarrow$  & $(\1,\k)$ \\  
\hline
$ P_{(13)(24);15}=
\begin{pmatrix}
0 & 0 & -1 & 0 \\
0 & 0 & 0 & -1 \\
-1 & 0 & 0 & 0 \\
0 & -1 & 0 & 0 \\
\end{pmatrix}$ 
& $\longrightarrow$ &    
$A=\begin{pmatrix}
0 & 0 & 0 & 1 \\
0 & 0 & -1 & 0 \\
0 & 1 & 0 & 0 \\
-1 & 0 & 0 & 0 \\
\end{pmatrix}$ \\
\hline
\end{tabular}  
\end{equation*}
\begin{equation*}
\begin{tabular}{ | c  c  c | } 
\hline
$(\frac{-\1-\i-\j-\k}{2},\frac{-\1-\i+\j+\k}{2})$ & $\longrightarrow$  & $(-\i,\j)$ \\  
\hline
$ P_{(143);3}=
\begin{pmatrix}
0 & 0 & -1 & 0 \\
0 & 1 & 0 & 0 \\
0 & 0 & 0 & -1 \\
1 & 0 & 0 & 0 \\
\end{pmatrix}$ 
& $\longrightarrow$ &    
$\begin{pmatrix}
0 & 0 & 0 & 1 \\
0 & 0 & -1 & 0 \\
0 & -1 & 0 & 0 \\
1 & 0 & 0 & 0 \\
\end{pmatrix}$ \\
\hline
\end{tabular}  
\end{equation*}
\begin{equation*}
\begin{tabular}{ | c  c  c | } 
\hline
$(\frac{-\1-\i+\j-\k}{2},\frac{-\1+\i+\j+\k}{2})$ & $\longrightarrow$  & $(-\k,-\1)$ \\  
\hline
$ P_{(142);6}=
\begin{pmatrix}
0 & -1 & 0 & 0 \\
0 & 0 & 0 & -1 \\
0 & 0 & 1 & 0 \\
1 & 0 & 0 & 0 \\
\end{pmatrix}$ 
& $\longrightarrow$ &    
$\begin{pmatrix}
0 & 0 & 0 & -1 \\
0 & 0 & -1 & 0 \\
0 & 1 & 0 & 0 \\
1 & 0 & 0 & 0 \\
\end{pmatrix}$ \\
\hline
\end{tabular}  
\end{equation*}
The $3$ cells of dimension $5$ are the following:
\begin{equation*}
\begin{tabular}{ | c  c  c | } 
\hline
$(\frac{-\1-\k}{\sqrt{2}},\frac{-\1+\k}{\sqrt{2}})$ & $\longrightarrow$  & $(-\k,-\1)$ \\  
\hline
$ P_{(14);1}=
\begin{pmatrix}
0 & 0 & 0 & -1 \\
0 & 1 & 0 & 0 \\
0 & 0 & 1 & 0 \\
1 & 0 & 0 & 0 \\
\end{pmatrix}$ 
& $\longrightarrow$ &    
$\begin{pmatrix}
0 & 0 & 0 & -1 \\
0 & 0 & -1 & 0 \\
0 & 1 & 0 & 0 \\
1 & 0 & 0 & 0 \\
\end{pmatrix}$ \\
\hline
\end{tabular}  
\end{equation*}
\begin{equation*}
\begin{tabular}{ | c  c  c | } 
\hline
$(\frac{-\1-\i}{\sqrt{2}},\frac{-\j+\k}{\sqrt{2}})$ & $\longrightarrow$  & $(-\i,-\j)$ \\  
\hline
$ P_{(1324);13}=
\begin{pmatrix}
0 & 0 & 0 & -1 \\
0 & 0 & 1 & 0 \\
-1 & 0 & 0 & 0 \\
0 & -1 & 0 & 0 \\
\end{pmatrix}$ 
& $\longrightarrow$ &    
$\begin{pmatrix}
0 & 0 & 0 & -1 \\
0 & 0 & 1 & 0 \\
0 & 1 & 0 & 0 \\
-1 & 0 & 0 & 0 \\
\end{pmatrix}$ \\
\hline
\end{tabular}  
\end{equation*}
\begin{equation*}
\begin{tabular}{ | c  c  c | } 
\hline
$(\frac{-\1-\i}{\sqrt{2}},\frac{\i+\k}{\sqrt{2}})$ & $\longrightarrow$  & $(-\i,\j)$ \\  
\hline
$ P_{(1423);7}=
\begin{pmatrix}
0 & 0 & -1 & 0 \\
0 & 0 & 0 & -1 \\
0 & -1 & 0 & 0 \\
1 & 0 & 0 & 0 \\
\end{pmatrix}$ 
& $\longrightarrow$ &    
$\begin{pmatrix}
0 & 0 & 0 & 1 \\
0 & 0 & -1 & 0 \\
0 & -1 & 0 & 0 \\
1 & 0 & 0 & 0 \\
\end{pmatrix}$ \\
\hline
\end{tabular}  
\end{equation*}
Finally, here's the only cell of dimension $6$:
\begin{equation*}
\begin{tabular}{ | c  c  c | } 
\hline
$(-\1,\k)$ & $\longrightarrow$  & $(-\1,\k)$ \\  
\hline
$ A ^\top =P_{(14)(23);5}=
\begin{pmatrix}
0 & 0 & 0 & -1 \\
0 & 0 & 1 & 0 \\
0 & -1 & 0 & 0 \\
-1 & 0 & 0 & 0 \\
\end{pmatrix}$ 
& $\longrightarrow$ &    
$A ^\top =\begin{pmatrix}
0 & 0 & 0 & -1 \\
0 & 0 & 1 & 0 \\
0 & -1 & 0 & 0 \\
1 & 0 & 0 & 0 \\
\end{pmatrix}$ \\
\hline
\end{tabular}  
\end{equation*}

Let us now explain how we found this list, and these explanations will actually serve as a proof of Theorem~\ref{th4}. 

The case of matrices is easy. The $24$ convex matrices in $B_4^+$ will be of the form $DP_\pi$, where $P_\pi$ is the matrix associated to a permutation $\pi \in S_4$ (there are $24$ such permutations) and $D \in \mathrm{Diag}_4^+$. We already explained how to compute the negative chopping of a matrix in Chapter~\ref{chapter5}, \S\ref{s53}: if we start with an arbitrary permutation $\pi \in S_4$, it it easy to find the unique $D=D_\pi \in \mathrm{Diag}_4^+$ such that
\[ \mathbf{chop}^-(D_\pi P_\pi)={} A ^\top. \]
In this way, we find the $24$ convex matrices of the form    $D_\pi P_\pi \in B_4^+$. Once we have found them, we use positive chopping to find their future, that is we simply compute $\mathbf{chop}^+(D_\pi P_\pi)$. 

The case of spins is more complicated. First, since we know all convex matrices in $B_4^+$, using the explicit expression of the projection $\Pi_4 : \SS^3 \times \SS^3 \rightarrow SO_4$ that we gave in Chapter~\ref{chapter2}, \S\ref{s21}, one can already find the two elements in $\tilde{B}_4^+$ that projects down to a convex matrices. For instance, we have
\[ \Pi_4\left(P_{(1423);7}\right)^{-1}=\pm \left(\frac{\1+\i}{\sqrt{2}},\frac{-\j-\k}{\sqrt{2}}\right) \in \tilde{B}_4^+. \]
Among these two elements in $\tilde{B}_4^+$, only one is convex: this convex spin will be characterized by the fact that there is a globally holonomic curve which is equal to this spin at time $t=0$ (for instance) and such that for $t<0$ small, the curve lies in the open cell of $\bar{\a}$. Let $\tilde{\Gamma}=\tilde{\Gamma}_{\gamma_1^1}$ be the curve defined by
\[ \tilde{\Gamma}(t)=\left(\exp\left(\pi t\frac{\sqrt{3}\i+\k}{2}\right),\exp\left(\pi t\frac{\k}{2}\right)\right), \quad t \in [0,1]. \]
This is the lifted Frenet frame curve of the curve $\gamma_1^1 \in \mathcal{L}\SS^3(-\1,\k)$, defined in Example~\ref{family1}, Chapter~\ref{chapter6}, \S\ref{s62}, and we know from Chapter~\ref{chapter6}, \S\ref{s63} that this curve is convex, and in fact, stably convex (since its final lifted Frenet frame is $\bar{\a} =(-\1,\k)$). Hence $\tilde{\Gamma}$ is globally holonomic. To test if one of these two elements is the convex one, for instance to test $\left(\frac{\1+\i}{\sqrt{2}},\frac{-\j-\k}{\sqrt{2}}\right)$, we look at the curve
\[ G(t)=\tilde{\Gamma}(t)\left(\frac{\1+\i}{\sqrt{2}},\frac{-\j-\k}{\sqrt{2}}\right) \] 
which is equal to $\left(\frac{\1+\i}{\sqrt{2}},\frac{-\j-\k}{\sqrt{2}}\right)$ for $t=0$. If this is the convex spin, for small $t<0$, $G(t)$ should be Bruhat equivalent to $\bar{\a} $. Using a first order approximation
\[ \tilde{\Gamma}(t) \approx \left(\1+\sqrt{3}\i\frac{\pi}{2}t+\k\frac{\pi}{2}t,\1+\k\frac{\pi}{2}t\right) \]
one computes the following approximation
\[ G(t) \approx \frac{1}{2\sqrt{2}}\left(\1(2-\sqrt{3}\pi t)+\i(2+\sqrt{3}\pi t)+\j\pi t+\k\pi t,\1\pi t+\i\pi t-2\j-2\k\right) \]
and hence for $t<0$ small, $G(t)$ is Bruhat equivalent to $(\1,-\k)=- \bar{\a}$. This spin is therefore not convex, and the convex one is thus $-\left(\frac{\1+\i}{\sqrt{2}},\frac{-\j-\k}{\sqrt{2}}\right)$. Now that we know the convex spin, let us explain how we can find his future. As before since we know the future of the matrix, in our case
\[ \mathbf{chop}^+(P_{(1423);7})
=\begin{pmatrix}
0 & 0 & 0 & 1 \\
0 & 0 & -1 & 0 \\
0 & -1 & 0 & 0 \\
1 & 0 & 0 & 0 
\end{pmatrix}
\]
we find two possibilities:
\[ \Pi_4\left(
\begin{pmatrix}
0 & 0 & 0 & 1 \\
0 & 0 & -1 & 0 \\
0 & -1 & 0 & 0 \\
1 & 0 & 0 & 0 
\end{pmatrix}
\right)^{-1}=\pm (\i,-\j). \]
To find the correct sign, look now at the curve
\[ -G(t)=\tilde{\Gamma}(t)\left(\frac{\1+\i}{\sqrt{2}},\frac{-\j-\k}{\sqrt{2}}\right) \]   
which is equal to the convex spin $-\left(\frac{\1+\i}{\sqrt{2}},\frac{-\j-\k}{\sqrt{2}}\right)$ for $t=0$. Using a first order approximation as before, one finds that for $t>0$ small, $-G(t)$ is Bruhat equivalent to $(-\i,\j)$. This procedure allow us to determine all convex spins and their future.

\section{Spaces $\mathcal{L}\SS^3(-\1,-\k)$, $\mathcal{L}\SS^3(\1,\k)$, $\mathcal{L}\SS^3(\1,-\k)$ and $\mathcal{L}\SS^3(-\1,\k)$}\label{s73}

Recall that the spaces we are interested in are the following ones:
\[\mathcal{L}\SS^3(\1,\1), \quad \mathcal{L}\SS^3(-\1,-\1), \quad \mathcal{L}\SS^3(\1,-\1), \quad \mathcal{L}\SS^3(-\1,\1). \]
In each case, the final lifted Frenet frame does not belong to an open Bruhat cell; moreover, from the list of convex spins we made in \S\ref{s73}, all of them are not convex except $(\1,-\1)$.

Using the chopping operation, we can replace these spaces by other equivalent spaces where the final lifted Frenet frame does belong to an open Bruhat cell (and in the case where it is convex, it becomes stably convex).

\begin{proposition}\label{spaces}
We have homeomorphisms
\[ \mathcal{L}\SS^3(\1,\1) \simeq \mathcal{L}\SS^3(-\1,-\k),  \]
\[ \mathcal{L}\SS^3(-\1,-\1) \simeq \mathcal{L}\SS^3(\1,\k),  \]
\[ \mathcal{L}\SS^3(\1,-\1) \simeq \mathcal{L}\SS^3(-\1,\k),  \]
\[\mathcal{L}\SS^3(-\1,\1) \simeq \mathcal{L}\SS^3(\1,-\k) . \]
\end{proposition} 

\begin{proof}
Using Proposition~\ref{chophomeo}, it is enough to prove that
\begin{equation}\label{relation}
\mathbf{chop^-}(\1,\1)=(-\1,-\k), \quad \mathbf{chop^-}(-\1,-\1)=(\1,\k)
\end{equation}
and
\begin{equation}\label{relation2}
 \mathbf{chop^-}(\1,-\1)=(-\1,\k), \quad \mathbf{chop^-}(-\1,\1)=(\1,-\k).
\end{equation}
We already proved the first equality of~\eqref{relation} in Proposition~\ref{aa}, since $\mathbf{chop^-}(\1,\1)=\a=(-\1,-\k)$. Using the same example we used in Proposition~\ref{aa}, one can prove the first equality of~\eqref{relation2}. Indeed, the curve $\gamma_1^4 \in \mathcal{L}\SS^3(\1,\1)$ (defined in Example~\ref{family1}, Chapter~\ref{chapter6}, \S\ref{s62}) is such that for all $t \in (0,1/2)$, $\tilde{\mathcal{F}}_{\gamma_1^4}(t)=\tilde{\Gamma}_{\gamma_1^4}(t)$ is Bruhat equivalent to $(-\1,\k)$. But then $\tilde{\mathcal{F}}_{\gamma_1^4}(1/2)=\tilde{\Gamma}_{\gamma_1^4}(1/2)=(\1,-\1)$, so this proves that $\mathbf{chop^-}(\1,-\1)=(-\1,\k)$.

Let us now prove the second equality of~\eqref{relation}. This time we consider the curve $\gamma_3^1 \in \mathcal{L}\SS^3(-\1,-\1)$ (defined in Example~\ref{family3}, Chapter~\ref{chapter6}, \S\ref{s62}) and we check that for all $t \in (0,1/2)$, $\tilde{\mathcal{F}}_{\gamma_3^1}(t)=\tilde{\Gamma}_{\gamma_3^1}(t)$ is Bruhat equivalent to $(-\1,\k)$ and for all $t \in (1/2,1)$, $\tilde{\mathcal{F}}_{\gamma_3^1}(t)=\tilde{\Gamma}_{\gamma_3^1}(t)$ is Bruhat equivalent to $(\1,\k)$. This shows that $\mathbf{chop^-}(-\1,-\1)=(\1,\k)$.

Finally, to prove the second equality of~\eqref{relation2}, we consider the curve $\gamma_4^1 \in \mathcal{L}\SS^3(-\1,\1)$ (defined in Example~\ref{family4}, Chapter~\ref{chapter6}, \S\ref{s62}). We check that for all small $t>0$ (for instance $t \in (0,0.41)$, $\tilde{\mathcal{F}}_{\gamma_4^1}(t)=\tilde{\Gamma}_{\gamma_4^1}(t)$ is Bruhat equivalent to $(-\1,\k)$ and for all $t<1$ close to $1$ (for instance $t \in (0.59,1)$), $\tilde{\mathcal{F}}_{\gamma_4^1}(t)=\tilde{\Gamma}_{\gamma_4^1}(t)$ is Bruhat equivalent to $(\1,-\k)$. This shows that $\mathbf{chop^-}(-\1,\1)=(\1,-\k)$.  
\end{proof}

\medskip

Therefore in the sequel, when convenient, we will rather look at the spaces
\[ \mathcal{L}\SS^3(-\1,-\k), \quad \mathcal{L}\SS^3(\1,\k), \quad \mathcal{L}\SS^3(\1,-\k), \quad \mathcal{L}\SS^3(-\1,\k). \]
All the spins (or pair of quaternions) $(-\1,-\k)$, $(\1,\k)$, $(\1,-\k)$ and $(-\1,\k)$ belongs to an open Bruhat cell, and the last one is stably convex. 


%
%

\chapter{Topology of $\mathcal{L}\SS^3(-\1,\k)$ and $\mathcal{L}\SS^3(\1,-\k)$}

\label{chapter8}

\section{Adding loops and spirals}\label{s81}

In this section, we describe an operation which geometrically consists in adding a pair of loops to a generic curve in $\SS^2$, and adding a closed spiral to a generic curve in $\SS^3$. In order to avoid repeating definitions, we will actually describe a similar operation for generic curves in $\SS^n$, for $n \geq 2$. We will study in more details the cases $n=2$ and $n=3$ in~\S\ref{s82} and~\S\ref{s83}

Let us fix an element $\omega_n \in \mathcal{L}\SS^n(\1)$. The operation we will perform actually depends on this choice, but since $\mathcal{L}\SS^n(\1)$ is connected, any other element in $\mathcal{L}\SS^n(\1)$ is homotopic to $\omega_n$, hence the choice of two different elements in $\mathcal{L}\SS^n(\1)$ will result in two different operation which will however yield homotopic curves. We will see later that this will be sufficient for our purpose. 

For $n=2$, one can choose for instance
\[ \omega_2=\sigma_c^2 \in \mathcal{L}\SS^2(\1) \]
where, for $t \in [0,1]$ and $0<c<2\pi$, 
\[ \sigma_c(t)=\cos\rho(\cos\rho,0,\sin\rho)+\sin\rho(\sin\rho\cos(2\pi t), \sin(2\pi t), -\cos\rho\cos(2\pi t))\]
and
\[ \sigma_c^2(t)=\sigma_c(2t). \]
This curve already appeared in Chapter~\ref{chapter6}, \S\ref{s62}; $\sigma_c : [0,1] \rightarrow \SS^2$ is the unique circle of length $c$, that is $||\sigma_c'(t)||=c$, with fixed initial and final Frenet frame equals to the identity, and so $\sigma_c^2$ consists in traveling along this circle twice.

For $n=3$, one can choose for example
\[ \omega_3=\gamma_1^4 \in \mathcal{L}\SS^3(\1,\1), \]
where $\gamma_1^4$ is extracted from the family of curves $\gamma_1^m$, $m\geq 1$, that we defined in Example~\ref{family1}, in Chapter~\ref{chapter6}, \S\ref{s62}. Explicitly we have
\begin{eqnarray*}
\gamma_1^4(t) & = & \left(\frac{1}{4}\cos\left(6t\pi \right)+\frac{3}{4}\cos\left(2t\pi \right)\right., \\
&  & \frac{\sqrt{3}}{4}\sin\left(6t\pi \right)+\frac{\sqrt{3}}{4}\sin\left(2t\pi \right) \\
&  & \frac{\sqrt{3}}{4}\cos\left(2t\pi \right)-\frac{\sqrt{3}}{4}\cos\left(6t\pi \right) \\
&  & \left.\frac{3}{4}\sin\left(2t\pi\right)-\frac{1}{4}\sin\left(6t\pi\right)\right).
\end{eqnarray*}   
Recall also that the left and right part of this curve are given by
\[  \gamma_{1,l}^4=\sigma_{\pi}^4 \in \mathcal{L}\SS^2(\1), \quad \gamma_{1,r}^4 =\sigma_{2\pi}^{2} \in \mathcal{G}\SS^2(\1). \]

Coming back to the general case $\omega_n \in \mathcal{L}\SS^n(\1)$, let us now define the operation of adding the closed curve $\omega_n$ to some curve $\gamma \in \mathcal{G}\SS^n(z)$ at some time $t_0 \in [0,1]$.

\begin{definition}\label{adding}
Let $\gamma \in \mathcal{G}\SS^n(z)$, and choose some point $t_0 \in [0,1]$. We define the curve $\gamma \ast_{t_0} \omega_n \in \mathcal{G}\SS^n(z)$ as follows. Given $\varepsilon>0$ sufficiently small, for $t_0 \in (0,1)$ we let
\begin{equation*}
\gamma \ast_{t_0} \omega_n(t)=
\begin{cases}
\gamma(t), & 0 \leq t \leq t_0-2\varepsilon \\
\gamma(2t-t_0+2\varepsilon), & t_0-2\varepsilon \leq t \leq t_0-\varepsilon \\
\mathcal{F}_\gamma(t_0)\omega_n\left(\frac{t-t_0+\varepsilon}{2\varepsilon}\right), & t_0-\varepsilon \leq t \leq t_0+\varepsilon \\
\gamma(2t-t_0-2\varepsilon), & t_0+\varepsilon \leq t \leq t_0+2\varepsilon \\
\gamma(t), & t_0+2\varepsilon \leq t \leq 1.
\end{cases}
\end{equation*}
For $t_0=0$, we let
\begin{equation*}
\gamma \ast_{0} \omega_n(t)=
\begin{cases}
\omega_n\left(\frac{t}{\varepsilon}\right), & 0 \leq t \leq \varepsilon\\
\gamma(2t-2\varepsilon), & \varepsilon \leq t \leq 2\varepsilon \\
\gamma(t), & 2\varepsilon \leq t \leq 1,
\end{cases}
\end{equation*}
and for $t_0=1$, we let
\begin{equation*}
\gamma \ast_{1} \omega_n(t)=
\begin{cases}
\gamma(t), & 0 \leq t \leq 1-2\varepsilon \\
\gamma(2t-1+2\varepsilon), & 1-2\varepsilon \leq t \leq 1-\varepsilon \\
\omega_n\left(\frac{t-1+\varepsilon}{\varepsilon}\right), & 1-\varepsilon \leq t \leq 1.
\end{cases}
\end{equation*}
\end{definition}

\begin{figure}[h]
\centering
\includegraphics[scale=0.5]{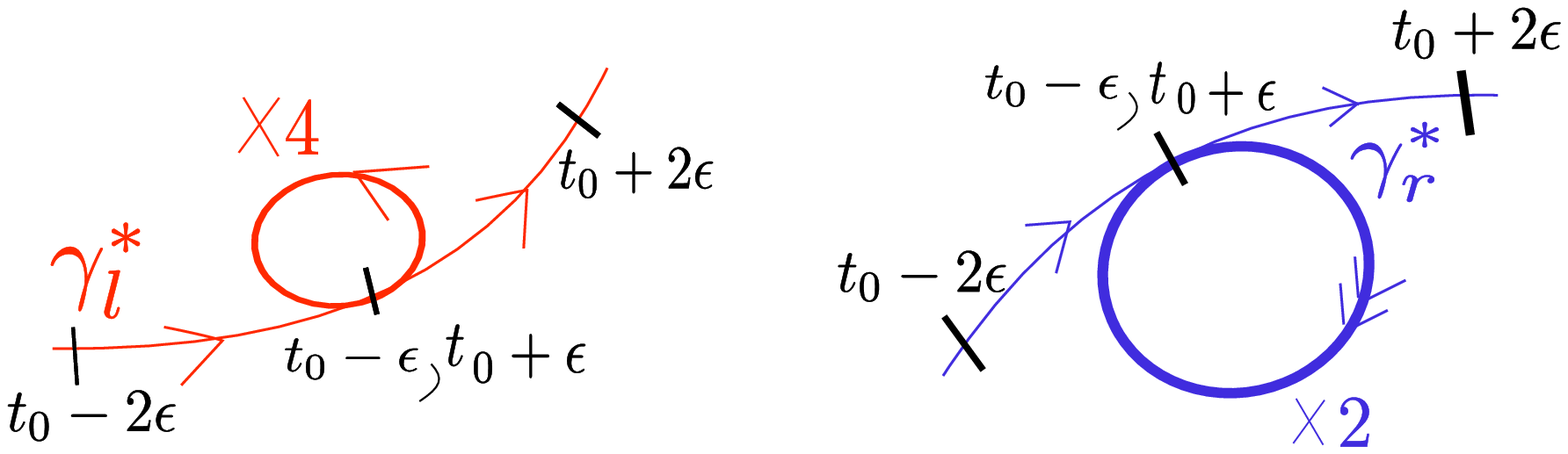}
\caption{Definition of the curve $\gamma^\ast=(\gamma_l^\ast,\gamma_r^\ast) \in \mathcal{L}\SS^3(z_l,z_r)$.}
\label{fig:m}
\end{figure}

This operation can be understood as follows (see Figure~\ref{fig:m} for an illustration in the case $n=3$). For $t_0 \in (0,1)$, we start by following the curve $\gamma$ normally, then we speed a little slightly before $t_0$ in order to have time to insert $\omega_n$ at time $t_0$ (that we moved to the correct position by a multiplication with $\mathcal{F}_\gamma(t_0)$), we speed again a little and finally at the end we follow $\gamma$ normally. For $t_0=0$ or $t_0=1$, we have a similar interpretation. 

The precise value of $\varepsilon$ is not important; a different value will yield a different parametrization but the same curve. 

The precise choice of $\omega_n$ will not be important either. Indeed, the space $\mathcal{L}\SS^n(\1)$ is path-connected (recall Theorem~\ref{thmani} from Chapter~\ref{chapter4}, \S\ref{s41}), hence if we choose any other element $\omega_n' \in \mathcal{L}\SS^n(\1)$, a homotopy between $\omega_n$ and $\omega_n'$ in $\mathcal{L}\SS^n(\1)$ will give a homotopy between the curves $\gamma \ast_{t_0} \omega_n $ and $\gamma \ast_{t_0} \omega_n'$ in $\mathcal{L}\SS^n(z)$. We will see later that the homotopy class of $\gamma \ast_{t_0} \omega_n $ is the only information we will be interested in. Therefore, to simplify notations, in the sequel we will write $\gamma_{t_0}^\ast$ instead of $\gamma \ast_{t_0} \omega_n$. 

It is clear from definition~\ref{adding} that if $\gamma \in \mathcal{L}\SS^n(z)$, then $\gamma_{t_0}^\ast \in \mathcal{L}\SS^n(z)$. Given an arbitrary compact set $K$, let us introduce the following definition. 

\begin{definition}\label{tight}
A continuous map $\alpha : K \rightarrow \mathcal{L}\SS^n(z)$ is \emph{loose} if there exist continuous maps
\[ A : K \times [0,1] \rightarrow \mathcal{L}\SS^n(z), \quad t_0 : K \rightarrow [0,1]  \]
such that for all $s \in K$:
\[ A(s,0)=\alpha(s), \quad A(s,1)=\alpha(s)_{t_0(s)}^\ast. \]
If the map $\alpha : K \rightarrow \mathcal{L}\SS^n(z)$ is not loose, then we say it is  \emph{tight}.
\end{definition}

If we identify $\alpha$ with a continuous (and hence uniform) family of curves $\alpha(s) \in \mathcal{L}\SS^n(z)$, $s \in K$, then $\alpha$ is loose if each curve $\alpha(s)$ is homotopic (with a homotopy depending continuously and hence uniformly in $s \in K$) to the curve $\alpha(s)_{t_0(s)}^\ast$, where the time $t_0(s)$ also depends continuously on $s$. Since the definition of being loose or tight just depend on the homotopy class of $\alpha(s)_{t_0(s)}^*$, it is independent of the choice of $\omega_n \in \mathcal{L}\SS^n(\1)$. To further simplify notations, we will often write $\gamma^\ast$ instead of $\gamma^\ast_{t_0}$ for a curve, and $\alpha^\ast$ for the family of curves $\alpha(s)_{t_0(s)}^*$ where $s$ varies in a compact set $K$.  

For the moment, it is not clear why this notion will be important in the sequel; this will be explained in~\S\ref{s82} and~\S\ref{s83}.

We have the following proposition.

\begin{proposition}\label{loosehom}
Consider two continuous maps $\alpha, \beta : K \rightarrow \mathcal{L}\SS^n(z)$, and assume that they are homotopic. Then $\alpha$ is loose if and only if $\beta$ is loose.  
\end{proposition}

\begin{proof}
Since $\alpha$ and $\beta$ are homotopic, there exist a continuous map
\[ H : K \times [0,1] \rightarrow \mathcal{L}\SS^n(z) \]
such that for all $s \in K$:
\[ H(s,0)=\alpha(s), \quad H(s,1)=\beta(s). \]
Let us define  
\[ H^\ast : K \times [0,1] \rightarrow \mathcal{L}\SS^n(z) \]
by setting, for all $(s,t) \in K \times [0,1]$:
\[ H^*(s,t)=(H(s,t))^*. \]
This is clearly a homotopy between $\alpha^\ast$ and $\beta^\ast$. Assume $\alpha$ is loose; we have a homotopy between $\alpha$ and $\alpha^\ast$, but since we also have a homotopy between $\beta$ and $\alpha$ and a homotopy between $\alpha^\ast$ and $\beta^\ast$, we have a homotopy between $\beta$ and $\beta^\ast$, hence $\beta$ is loose. Assuming $\beta$ loose, the exact same argument shows that $\alpha$ is loose.   
\end{proof}

The following corollary is obvious.

\begin{corollary}
If $\alpha: K \rightarrow \mathcal{L}\SS^n(z)$ is loose, then $\alpha^\ast: K \rightarrow \mathcal{L}\SS^n(z)$ is loose. 
\end{corollary}

Note that, in fact, $\alpha^\ast$ is always loose. 

A curve $\gamma \in \mathcal{L}\SS^n(z)$ can be identified with the image of a continuous map $\alpha : K \rightarrow \mathcal{L}\SS^n(z)$, where $K$ is a set with one element. In this way, a curve $\gamma \in \mathcal{L}\SS^n(z)$ can be either loose or tight. The following proposition is well-known (from the works of Shapiro~\cite{Sha93} and Anisov~\cite{Ani98}).

\begin{proposition}\label{tightconvex}
A curve $\gamma \in \mathcal{L}\SS^n(z)$ is tight if and only if it is convex.
\end{proposition} 

Now let us look at the case where $n=3$. Given a continuous map $\alpha : K \rightarrow \mathcal{L}\SS^3(z_l,z_r)$, one can define its left part, $\alpha_l : K \rightarrow \mathcal{L}\SS^2(z_l)$ simply by setting $\alpha_l(s)=(\alpha(s))_l$, for $s \in K$. The following proposition gives us the relation between the tightness of $\alpha$ and the tightness of its left part $\alpha_l$. 

\begin{proposition}\label{leftight}
If $\alpha : K \rightarrow \mathcal{L}\SS^3(z_l,z_r)$ is loose, then $\alpha_l : K \rightarrow \mathcal{L}\SS^2(z_l)$ is loose. As a consequence, if $\alpha_l : K \rightarrow \mathcal{L}\SS^2(z_l)$ is tight, then $\alpha : K \rightarrow \mathcal{L}\SS^3(z_l,z_r)$ is tight.   
\end{proposition}

\begin{proof}
We assume that $\alpha$ is loose. Then there exist a continuous map
\[ A : K \times [0,1] \rightarrow \mathcal{L}\SS^3(z_l,z_r) \]
such that for all $s \in K$:
\[ A(s,0)=\alpha(s), \quad A(s,1)=\alpha(s)^\ast. \]
Let us define the map
\[ A_l : K \times [0,1] \rightarrow \mathcal{L}\SS^2(z_l) \]
simply by setting $A_l(s,t)=(A(s,t))_l$. Since the map giving the left part of a curve is a continuous map, $A_l$ is continuous. But now it is easy to observe that
\[ A_l(s,1)=\alpha_l(s)^\ast \]
which proves that $\alpha_l$ is loose.
\end{proof}

\medskip

Using Proposition~\ref{tightconvex} and~\ref{leftight}, one immediately obtain the following proposition.

\begin{proposition}\label{leftconvex}
Let $\gamma \in \mathcal{L}\SS^3(z_l,z_r)$. If $\gamma_l $ is convex, then $\gamma$ is convex.
\end{proposition}

The converse is not true in general. The curve $\gamma_1^2$ defined in Example~\ref{family1} is convex, but its left part, which is of the form $\sigma_c^2$ for some $0<c<2\pi$, is clearly not convex.

\section{The case $n=2$}\label{s82}

In this section, we recall some results of~\cite{Sal13} in the case $n=2$. Our main task to prove Theorem~\ref{th5} will be to extend these results in the case $n=3$.

Let us first recall that by definition, $\gamma$ is loose if it is homotopic to $\gamma^\ast$ inside the space of locally convex curves. In the case $n=2$, if we allow the homotopy to be inside the space of generic curves, then this turns out to be very different as the following proposition shows.

\begin{proposition}\label{genericas}
Let $\alpha: K \rightarrow \mathcal{G}\SS^2(z)$ be a continuous map. Then $\alpha$ is homotopic to $\alpha^\ast$ inside the space $\mathcal{G}\SS^2(z)$. 
\end{proposition}

We refer to the Figure~\ref{fig:o} to see how such a homotopy can be constructed.
\begin{figure}[H]
\centering
\includegraphics[scale=0.5]{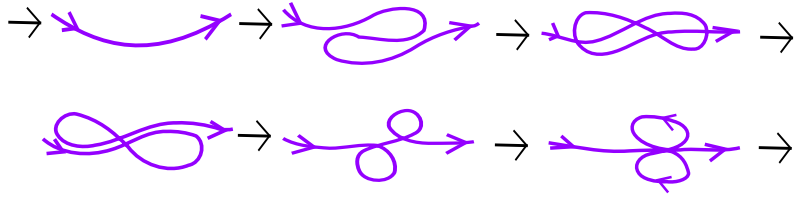}
\caption{Homotopy between $\alpha$ and $\alpha^\ast$ in $\mathcal{G}\SS^2(z)$.}
\label{fig:o}
\end{figure}

In particular, if $\alpha: K \rightarrow \mathcal{L}\SS^2(z)$ is a continuous map, then $\alpha^\ast: K \rightarrow \mathcal{L}\SS^2(z)$ but $\alpha$ is homotopic to $\alpha^\ast$ inside the space $\mathcal{G}\SS^2(z)$. The image of the homotopy does not necessarily lie in the space $\mathcal{L}\SS^2(z)$; if it does then $\alpha$ is by definition loose. 

Now assume that $\alpha^\ast$ is homotopic to a constant map in $\mathcal{L}\SS^2(z)$. Since from Proposition~\ref{genericas} we know that $\alpha$ is always homotopic to $\alpha^\ast$ in $\mathcal{G}\SS^2(z)$, we obtain in particular that $\alpha$ is homotopic to a constant map in $\mathcal{G}\SS^2(z)$. We do not prove it here, but the converse is also true.

\begin{proposition}\label{genericas2}
Let $\alpha: K \rightarrow \mathcal{G}\SS^2(z)$ be a continuous map. Then $\alpha$ is homotopic to a constant map in $\mathcal{G}\SS^2(z)$ if and only if $\alpha^\ast$ is homotopic to a constant map in $\mathcal{L}\SS^2(z)$. 
\end{proposition}    

We can now explain why we are interested in finding tight maps. In order to understand the difference between $\mathcal{L}\SS^2(z)$ and $\mathcal{G}\SS^2(z)$, one would like to find maps, say defined on $K=\SS^p$ for some $p \geq 1$, which are homotopic to a constant in $\mathcal{G}\SS^2(z)$ but not homotopic a to constant in $\mathcal{L}\SS^2(z)$. Indeed, if one find such a map, this would gives a non-zero element in $\pi_p(\mathcal{L}\SS^2(z))$ which is taken to zero in $\pi_p(\mathcal{G}\SS^2(z))$. 

In~\cite{Sal13}, Saldanha constructed tight maps, for an integer $k \geq 2$, 
\[ h_{2k-2} : \SS^{2k-2} \rightarrow \mathcal{L}\SS^2((-\1)^k) \]
which are homotopic to constants maps in $\mathcal{G}\SS^2((-\1)^k)$. To prove that these maps are not homotopic to a constant in $\mathcal{L}\SS^2((-\1)^k)$, he introduced the following notion.

\begin{definition}\label{multiconvex}
A curve $\gamma \in \mathcal{L}\SS^2(z)$ is \emph{multiconvex} of multiplicity $k$ if there exist times $0 = t_0 < t_1 < \cdots <
t_k = 1$ such that $\mathcal{F}_\gamma(t_i)=I$ for $0 \leq i<k$, and the restrictions of $\gamma$ to the intervals $[t_{i-1},t_i]$ are convex arcs for $1 \leq i \leq k$. 
\end{definition}

Let us denote by $\mathcal{M}_k(z)$ the set of multiconvex curves of multiplicity $k$ in $ \mathcal{L}\SS^2(z)$. It is clear that a curve is multiconvex of multiplicity $1$ if and only if it is convex. Also, one can also see that for $k$ odd, $\mathcal{M}_k(z) \neq \emptyset$ if and only if $z$ is convex, and for $k$ even, $\mathcal{M}_k(z) \neq \emptyset$ if and only if $-z$ is convex.

\begin{proposition}\label{manifold}
The set $\mathcal{M}_k(z)$ is a closed contractible submanifold of $\mathcal{L}\SS^2(z)$ of codimension $2k-2$ with trivial normal bundle.
\end{proposition} 

Therefore we can associate to $\mathcal{M}_k(z)$ a cohomology class $m_{2k-2} \in H^{2k-2}(\mathcal{L}\SS^2(z),\R)$ by counting intersection with multiplicity. Given any continuous map $\alpha : K \rightarrow \mathcal{L}\SS^2(z)$, by a perturbation we can make it smooth and transverse to $\mathcal{M}_k(z)$, and we denote by $m_{2k-2}(\alpha) \in \R$ the intersection number of $\alpha$ with $\mathcal{M}_k(z)$.

The following proposition was proved in~\cite{Sal13}.

\begin{proposition}\label{nico}
Given an integer $k \geq 2$, there exist (explicit) maps 
\[ h_{2k-2} : \SS^{2k-2} \rightarrow \mathcal{L}\SS^2((-\1)^k) \]
which are tight, homotopic to constant maps in $\mathcal{G}\SS^2((-\1)^k)$ and such that $m_{2k-2}(h_{2k-2})=\pm 1$. As a consequence, these maps $h_{2k-2}$ are not homotopic to constant maps in $\mathcal{L}\SS^2((-\1)^k)$. 
\end{proposition} 

Therefore, $h_{2k-2}$ defines extra generators in $\pi_{2k-2}(\mathcal{L}\SS^2((-\1)^k))$ (as compared to $\pi_{2k-2}(\mathcal{G}\SS^2((-\1)^k)$) and $m_{2k-2}$ defines extra generators in $H^{2k-2}(\mathcal{L}\SS^2((-\1)^k),\R)$ (as compared to $H^{2k-2}(\mathcal{G}\SS^2((-\1)^k,\R)$). 

Our objective will be to use Proposition~\ref{nico}, together with our decomposition results Theorem~\ref{th0} and Theorem~\ref{th1} to draw similar conclusions in the case $n=3$. We will be able to do this only in two cases, namely for $\mathcal{L}\SS^3(\1,-\1) \simeq \mathcal{L}\SS^3(-\1,\k)$ and $\mathcal{L}\SS^3(-\1,\1) \simeq \mathcal{L}\SS^3(\1,-\k)$. But first some extra work is needed.    

\section{The case $n=3$}\label{s83}

The goal of this section is to obtain proposition analogous to Proposition~\ref{genericas} and Proposition~\ref{genericas2} in our case $n=3$.

Recall that Proposition~\ref{genericas} states that a map $\alpha : K \rightarrow \mathcal{L}\SS^2(z)$ is always homotopic to $\alpha^\ast$ inside the space $\mathcal{G}\SS^2(z)$. In the case $n=3$, this is not so obvious. Yet using the result in the case $n=2$, we will prove below that a map $\alpha : K \rightarrow \mathcal{L}\SS^3(z_l,z_r)$ is always homotopic, in $\mathcal{G}\SS^3(z_l,z_r)$, to the curve $\alpha$ to which we attached a pair of loops with zero geodesic torsion, that is an element in $\mathcal{G}\SS^3(\1,\1)$ with zero geodesic torsion. One could then change the definition of $\alpha^*$ so that instead of attaching an element in $\mathcal{L}\SS^3(\1,\1)$, one attaches an element in $\mathcal{G}\SS^3(\1,\1)$ with zero geodesic torsion. The obvious problem is that if $\alpha$ takes values in $\mathcal{L}\SS^3(z_l,z_r)$, this would no longer be the case of $\alpha^*$. 

To solve this issue, recall that to an element in $g \in \mathcal{G}\SS^3(\1,\1)$ with zero geodesic torsion is associated a pair of curves $(g_l,g_r) \in \mathcal{L}\SS^2(\1) \times \mathcal{G}\SS^2(\1)$     such that $\kappa_{g_l}=-\kappa_{g_r}>0$ (this follows from Theorem~\ref{th0}, observe that according to our definition, $g_r$ is not locally convex but negative locally convex). Given a curve $\gamma \in \mathcal{G}\SS^3(z_l,z_r)$, let us decompose it into its left and right parts $\gamma=(\gamma_l,\gamma_r)$, and let $\gamma \ast g$ be the curve $\gamma$ to which we attached the curve $g$ at some point. Then it is easy to see that $\gamma \ast g=(\gamma_l \ast g_l,\gamma_r \ast g_r)$, that is the left (respectively right) part of $\gamma \ast g$ is obtained by attaching the left (respectively right) part of $g$ to the left (respectively right) part of $\gamma$. As we already explained, if $\gamma$ is locally convex, then $\gamma \ast g$ is not locally convex because it does not satisfy the condition on the geodesic curvature. A first attempt would be to slightly modify $g_l$ (or $g_r$) into $\tilde{g}_l$ so that the geodesic curvature condition is met; but then the condition on the norm of the speed would not be satisfied, that is $||(\gamma_l \ast \tilde{g}_l)'(t)|| \neq ||(\gamma_r \ast g_r)'(t)||$. Hence in order to satisfy both conditions at the same time, we will have to modify the whole curve in a rather subtle way. 

At the end we should obtain a curve, that we shall call $\gamma^\#$ (to distinguish from the curve $\gamma^\ast$ we previously defined), that has the property that if $\gamma$ is locally convex, then so is $\gamma^\#$. Then of course one has to know how this procedure is related to the procedure of adding loops we defined. The curve $\gamma^\#$ is of course different from the curve $\gamma^\ast$, but we will see later that $\gamma$ is loose (meaning that $\gamma$ is homotopic to $\gamma^\ast$) if and only if $\gamma$ is homotopic to $\gamma^\#$; hence defining loose and tight with respect to $\gamma^\ast$ or $\gamma^\#$ is just a matter of convenience.  

We will use the Lemma below, for construct the curve $\gamma^\sharp$.

\begin{lemma}\label{finallemma}
Consider a convex arc $\gamma: [t_{0}-2\varepsilon,t_{0}+2\varepsilon] \rightarrow \SS^2$ and positive numbers $K_0$, $K_1$, with $K_1 > \kappa_{\gamma}(t) > K_0$,
for all $t \in [t_0-2\varepsilon,t_0+2\varepsilon]$. Then given $t_{---} \in [t_0-2\varepsilon, t_{0})$ and $t_{+++} \in (t_0,t_0+2\varepsilon]$ there exist an unique arc $\nu:[t_0-2\varepsilon,t_0+2\varepsilon] \rightarrow \SS^2$ (up to reparametrization) and times $t_{--}, \; t_{++}$ with $t_{--} \in (t_{---},t_{0})$ and $t_{++} \in (t_{0},t_{+++})$ such that

\begin{equation}\label{cond0}
\nu(t)=\gamma(t), \quad t \notin [t_{---},t_{+++}], 
\end{equation}

\begin{equation}\label{cond1}
\kappa_{\nu}(t)= K_0, \quad t \in [t_{---},t_{--}] \cup  [t_{++},t_{+++}], 
\end{equation} 

\begin{equation}\label{cond2}
\kappa_{\nu}(t)= K_1, \quad t \in [t_{--},t_{++}] \quad \mathrm{and} 
\end{equation} 

\begin{equation}\label{cond3}
\int_{t_{---}}^{t_{+++}}||\gamma'(t)||dt < \int_{t_{---}}^{t_{+++}}||\nu'(t)||dt
\end{equation}

Futhermore, $t_{---}$ and $t_{+++}$ can be chosen so that there exist $t_{-}$, $t_{+}$, with $t_{-} \in (t_{--},t_{0})$ and $t_{+} \in (t_{0},t_{++})$
and

\begin{equation}\label{cond4}
\int_{t_{---}}^{t_{0}}||\gamma'(t)||dt = \int_{t_{---}}^{t_{-}}||\nu'(t)||dt
\end{equation}

\begin{equation}\label{cond5}
\int_{t_{0}}^{t_{+++}}||\gamma'(t)||dt = \int_{t_{+}}^{t_{+++}}||\nu'(t)||dt
\end{equation}

\end{lemma}

\begin{proof}
The proof of the Lemma is easy. The process is illustrated in Figure ~\ref{fig:lemma} 
\begin{figure}[H]
\centering
\includegraphics[scale=0.5]{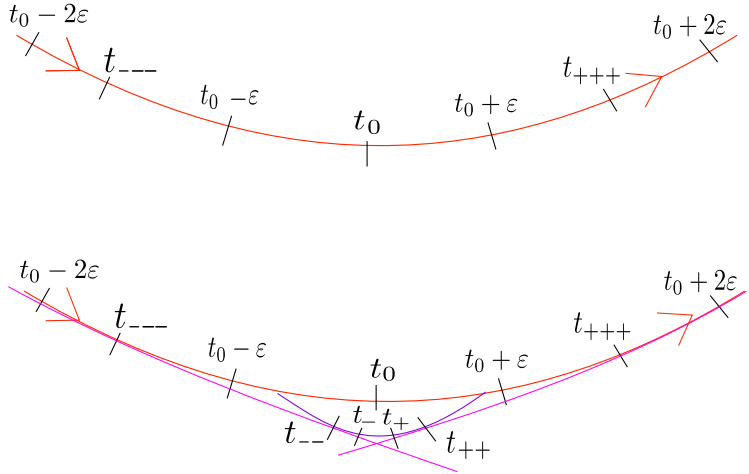}
\caption{How we modify a curve $\gamma \in \mathcal{L}\SS^2(z)$.}
\label{fig:lemma}
\end{figure}

\end{proof}

Given a curve $\gamma \in \mathcal{L}\SS^3(z_l,z_r)$, let $\gamma_l \in \mathcal{L}\SS^2(z_l)$ and $\gamma_r \in \mathcal{G}\SS^2(z_r)$. To define the curve $\gamma^\# \in \mathcal{L}\SS^3(z_l,z_r)$, we will define its pair of curves $\gamma_l^\# \in \mathcal{L}\SS^2(z_l)$ and $\gamma_r^\# \in \mathcal{G}\SS^2(z_r)$, using the Lemma \ref{finallemma}. We do not think that an explicit definition is helpful. We hope the picture below (\ref{fig:n}) is useful to understand the idea of the construction. Fix $t_0 \in (0,1)$ (the case $t_0=0$ and $t_0=1$ can be treated in the same way). The curve we are going to define depends of course on $t_0$, but as before, we will simply write $\gamma_{t_0}^{\#}=\gamma^\#$.   

The curvature of $\gamma_l$ and $\gamma_r$ at the point $t_0$ satisfy $\kappa_{\gamma_l}(t_0)>|\kappa_{\gamma_r}(t_0)|$. Since $\kappa_{\gamma_l}(t)$ and $|\kappa_{\gamma_r}(t)|$ can be assumed to be continuous, there exist $\varepsilon>0$ and $K_0>0$, $K_1>0$ such that for all $t_l \in [t_0-2\varepsilon,t_0+2\varepsilon]$ and $t_r \in [t_0-2\varepsilon,t_0+2\varepsilon]$, one has
\begin{equation}\label{curvature}
K_1> \kappa_{\gamma_l}(t_l) > K_0 > |\kappa_{\gamma_r}(t_r)|.
\end{equation} 

Now we are in the situation of the Lemma \ref{finallemma}, that we will use to construct $\gamma^{\#}=(\gamma_l^\#,\gamma_r^\#).$

Outside the interval $[t_{0}-2\varepsilon,t_{0}+2\varepsilon]$, we will not modify the curves $\gamma_l$ and $\gamma_r$, that is we set
\begin{equation}
\gamma_l^\#(t)=\gamma_l(t), \quad \gamma_r^\#(t)=\gamma_r(t), \quad t \notin [t_0-2\varepsilon,t_0+2\varepsilon]. 
\end{equation} 
Hence for $t \notin [t_0-2\varepsilon,t_0+2\varepsilon]$, the conditions to define a pair of curves are clearly satisfied.

In the set $[t_{0}-2\varepsilon,t_{0}-\varepsilon] \cup [t_{0}+\varepsilon,t_{0}+2\varepsilon]$, $\gamma_r^\#$ will simply correspond to a reparametrization of $\gamma_r$, such that the curve $\gamma_r^\#$ on these intervals has two times the velocity of $\gamma_r$ in the same interval.
For $\gamma_l^\#$, $t \in [t_{0}-2\varepsilon,t_{0}-\varepsilon] \cup [t_{0}+\varepsilon,t_{0}+2\varepsilon] $ we will follow  the curve $\nu$ reparametrized by $\varphi_{-}:[t_0-2\varepsilon,t_0-\varepsilon] \rightarrow [t_0-2\varepsilon,t_{-}]$ and $\varphi_{+}:[t_0+\varepsilon,t_0+2\varepsilon] \rightarrow [t_{+},t_{0}+2\varepsilon]$. 
Therefore, from this and (\ref{cond4}) and (\ref{cond5}) the condition on the length is satisfied. The condition of the geodesic curvature is also satisfied, since in this set
\[ \kappa_{\gamma_l^\#}(t)> K_{0} > |\kappa_{\gamma_r^\#}(t)|, \]
see (\ref{curvature}).

It remains to define the curve on the interval $[t_0-\varepsilon,t_0+\varepsilon]$. Observe here that $\gamma_r^\#(t_0-\varepsilon)=\gamma_r(t_0)=\gamma_r^\#(t_0+\varepsilon)$, while $\gamma_l^\#(t_0-\varepsilon)\neq \gamma_l^\#(t_0+\varepsilon)$. Note that, by construction, $\gamma_l^\#(t_{0}-\varepsilon)=\nu(t_{-})$ and $\gamma_l^\#(t_{0}+\varepsilon)=\nu(t_{+}).$
The curve $\gamma_l^\#$ for $t \in [t_{0}-\varepsilon,t_{0}+\varepsilon]$ follows a circle of length $c_{1}$ with geodesic curvature $K_1$, performing slightly more than 2 times. Therefore, for all $t \in [t_0-\varepsilon,t_0+\varepsilon],$ one has
\[ \kappa_{\gamma_l^\#}(t)=K_1 \quad \mathrm{and}\]
\[ \int_{t_{0}-\varepsilon}^{t_{0}+\varepsilon} ||(\gamma_l^\#)'(t)||dt =2c_1+ \int_{t_{-}}^{t_{+}} ||(\nu'(t)||dt \].

Recall that given any $0<c< 2\pi$, we defined $\sigma_c$ to be the unique circle with initial and final Frenet frame equals to the identity, and such that $||\sigma_c'(t)||=c$. Such a curve has constant geodesic curvature $\cot(\rho)$, where $c=2\pi\sin\rho$. Associated to $K_0>0$ let $0<c_0<2\pi$ such that the geodesic curvature of $\sigma_{c_0}$ is equal to $K_0$. 

Let us now choose $c_2=c_1+\frac{1}{2}\int_{t_{-}}^{t_{+}} ||\nu'(t)||dt$, and let $\bar{\sigma}_{c_2}$ be the curve obtained by reflecting the curve $\sigma_{c_2}$ with respect to the hyperplane $\{(x,y,z) \in \R^3 \; | \; z=0\}$ (that is, $\bar{\sigma}_{c_2}$ is the image of $\sigma_{c_2}$ by the map $(x,y,z) \mapsto (x,y,-z)$). Such a curve $\bar{\sigma}_{c_2}$ has constant negative geodesic curvature $-K_2$ (hence it is negative locally convex). Now define $\gamma_r^\#$ on $[t_0-\varepsilon,t_0+\varepsilon]$ by setting 
\begin{equation*}
\gamma_r^\#(t)=(\mathcal{F}_{\gamma_r}(t_0)) \bar{\sigma}_{c_2}^2\left(\frac{t-t_0+\varepsilon}{2\varepsilon}\right), \quad t \in [t_0-\varepsilon,t_0+\varepsilon].
\end{equation*}

So, since $c_2>c_1$ and the absolute value $K_2$ of the geodesic curvature of $\sigma_{c_2}$ satisfies $K_2<K_1$, hence
\[ \kappa_{\gamma_l^\#}(t)=K_1>K_2=|\kappa_{\gamma_r^\#}(t)|. \]
Therefore the conditions to define a pair of curves are also satisfied on $[t_0-\varepsilon,t_0+\varepsilon]$. Here's an illustration (Figure~\ref{fig:n}) summarizing the definition of the curve $\gamma^\#=(\gamma_l^\#,\gamma_r^\#)$.
\begin{figure}[H]
\centering
\includegraphics[scale=0.5]{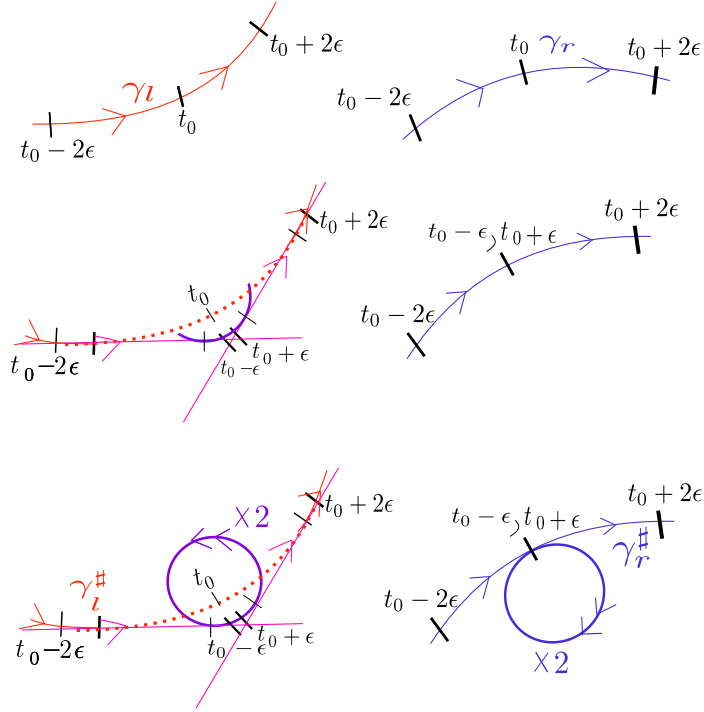}
\caption{Definition of the curve $\gamma^\#=(\gamma_l^\#,\gamma_r^\#) \in \mathcal{L}\SS^3(z_l,z_r)$.}
\label{fig:n}
\end{figure}

Let us now make the following definitions.

\begin{definition}\label{sustenido}
Given a curve $\gamma \in \mathcal{L}\SS^3(z_l,z_r)$ and a time $t_0 \in [0,1]$, we define $\gamma_{t_0}^\#=\gamma^\# \in \mathcal{L}\SS^3(z_l,z_r)$ by setting $\gamma^\#=(\gamma_l^\#,\gamma_r^\#)$, where $\gamma_l^\# \in \mathcal{L}\SS^2(z_l)$ and $\gamma_r^\# \in \mathcal{G}\SS^2(z_r)$ are defined like in the construction above. Given continuous maps $\alpha : K \rightarrow \mathcal{L}\SS^3(z_l,z_r)$ and $t_0 : K \rightarrow [0,1]$, we define $\alpha_{t_0}^\#=\alpha^\# : K \rightarrow \mathcal{L}\SS^3(z_l,z_r)$ by setting $\alpha_{t_0}^\#(s)=(\alpha(s))_{t_0(s)}^\#$ for all $s \in K$.
\end{definition}

\begin{definition}\label{tight2}
A continuous map $\alpha : K \rightarrow \mathcal{L}\SS^3(z)$ is \emph{$\#$-loose} if there exist continuous maps
\[ A : K \times [0,1] \rightarrow \mathcal{L}\SS^3(z), \quad t_0 : K \rightarrow [0,1]  \]
such that for all $s \in K$:
\[ A(s,0)=\alpha(s), \quad A(s,1)=\alpha(s)_{t_0(s)}^\#. \]
If the map $\alpha : K \rightarrow \mathcal{L}\SS^3(z)$ is not $\#$-loose, then we say it is   \emph{$\#$-tight}.
\end{definition}

Then we have the following proposition.

\begin{proposition}\label{tight12}
A continuous map $\alpha : K \rightarrow \mathcal{L}\SS^3(z)$ is $\#$-loose if and only if it is loose. Therefore a continuous map $\alpha : K \rightarrow \mathcal{L}\SS^3(z)$ is $\#$-tight if and only if it is tight.
\end{proposition}

One can use this proposition using the techniques of ``spreading loops along a curve" (see for instance~\cite{Sal13}), which can be seen as an easy instance of the h-principle of Gromov. We will actually not prove this proposition since we will not use it; in the sequel it will be more convenient to deal with this operation $\#$ since it will enable us to transfer more easily known results in the case $n=2$.

We can now prove the following proposition, which is the equivalent of Proposition~\eqref{genericas}. 

\begin{proposition}\label{Genericas}
Let $\alpha: K \rightarrow \mathcal{G}\SS^3(z_l,z_r)$ be a continuous map. Then $\alpha$ is homotopic to $\alpha^\#$ inside the space $\mathcal{G}\SS^3(z_l,z_r)$. 
\end{proposition}

\begin{proof}
It will be sufficient to consider the case where $K$ is a point, that is we will prove that any curve $\gamma \in \mathcal{G}\SS^3(z_l,z_r)$ is homotopic to the curve $\gamma^\# \in \mathcal{G}\SS^3(z_l,z_r)$ inside $\mathcal{G}\SS^3(z_l,z_r)$.

The point $t_0$ being fixed, it will be sufficient to construct the homotopy in an interval $(t_0-\varepsilon,t_0+\varepsilon)$ for a small $\varepsilon>0$. Moreover, the construction being local, after a central projection we can assume that $\gamma : (t_0-\varepsilon,t_0+\varepsilon) \rightarrow \R^3$. Moreover, without loss of generality, we may also assume that $\mathbf{t}_{\gamma}(t_0)=e_1$ and $\mathbf{n}_{\gamma}(t_0)=e_2$ where $e_1$ and $e_2$ are the first two vectors of the canonical basis of $\R^3$.

We will construct a homotopy $\gamma_s : (t_0-\varepsilon,t_0+\varepsilon) \rightarrow \R^3$, $0 \leq s \leq 1$, between $\gamma_0=\gamma$ and a curve $\gamma_1$ which is obtained from $\gamma$ by attaching, at time $t_0$, a pair of loops with zero torsion. In $\SS^3$, this will give a homotopy between $\gamma$ and a curve obtained from $\gamma$ by attaching an element in $\mathcal{G}\SS^3(\1,\1)$. The latter curve corresponds to attaching a pair of loops to $\gamma_l$, and a pair of reflected loops (that is with opposite geodesic curvature) to $\gamma_r$; but clearly such a curve is homotopic in $\mathcal{G}\SS^3(z_l,z_r)$ to the curve $\gamma^\#$ we defined.

So let us construct such a homotopy. Let $I_\varepsilon=(t_0-\varepsilon,t_0+\varepsilon)$, and consider a smaller closed interval $J_\varepsilon=[t_0-\varepsilon/10,t_0+\varepsilon/10]$. For any $0 \leq s \leq 1$, we set
\[ \gamma_s(t)=\gamma_0(t), \quad t \in I_\varepsilon \setminus J_\varepsilon. \]
Hence it remains to define $\gamma_s(t)$ for $t \in J_\varepsilon$. To do this, it is sufficient to define $\gamma_s'(t)$ for $t \in J_\varepsilon$ and check that the relation
\begin{equation}\label{Relation}
\int_{J_\varepsilon}\gamma_s'(t)dt=\int_{J_\varepsilon}\gamma_0'(t)dt
\end{equation}  
is satisfied. Indeed, for $0 \leq s \leq 1$ one can then define
\[ \gamma_s(t)=\gamma_s(t_0-\varepsilon/10)+\int_{t_0-\varepsilon/10}^t\gamma_s'(u)du, \quad t \in J_\varepsilon \]
and the continuity of $\gamma_s$ at $t=t_0+\varepsilon/10$ follows from~\eqref{Relation}.

Let us set 
\[ C_0=\int_{J_\varepsilon}\gamma_0'(t)dt. \]
From our assumption, the map
\[ \mathbf{t}_{\gamma} : t \in I_\varepsilon \mapsto \frac{\gamma'(t)}{||\gamma'(t)||} \in \SS^2 \]
is an immersion with $\mathbf{t}_{\gamma}(t_0)=e_1$ and $\mathbf{t}_{\gamma}'(t_0)$ is a positive multiple of $e_2$. Up to slightly perturbing $\gamma$ if necessary, we may assume that $C_0$ belongs to the interior of the convex hull of the image of $J_\varepsilon$ by $\mathbf{t}_{\gamma}$. Since the image of $J_\varepsilon$ by $\mathbf{t}_{\gamma}$, its convex hull is closed. Given $0<\delta<\varepsilon/100$, consider an even smaller open interval $L_\delta=(t_0-\delta,t_0+\delta)$. Choosing $\delta$ small enough, one may even assume that $C_0$ belongs to the interior of the convex hull of the image of $J_\varepsilon \setminus L_\delta$ by $\mathbf{t}_{\gamma}$. 

To define $\gamma_s'(t)$ for $t \in J_\varepsilon$, we will define $\mathbf{t}_{\gamma_s}(t)$ and $||\gamma_s'(t)||$ for $t \in J_\varepsilon$; this defines $\gamma_s'(t)$ by setting $\gamma_s'(t)=||\gamma_s'(t)||\mathbf{t}_{\gamma_s}(t)$.

We first define, for $0 \leq s \leq 1$,
\begin{equation}\label{tangente}
\mathbf{t}_{\gamma_s}(t)=\mathbf{t}_{\gamma_0}(t), \quad t \in J_\varepsilon \setminus L_\delta.
\end{equation} 
For $t \in L_\delta$, we define $\mathbf{t}_{\gamma_s}(t)$ to be a homotopy between $\mathbf{t}_{\gamma_0}(t)$ and $\mathbf{t}_{\gamma_1}(t)=\sigma^2_{2\pi}(t)$, where $\sigma^2_{2\pi}$ is the meridian curve on $\SS^2$ traveled $2$ times (such a curve was defined in Chapter~\ref{chapter6}, \S\ref{s62}). The existence of this homotopy follows from Proposition~\ref{genericas}. 
Then for $t \in L_\delta$ and any $0 \leq s \leq 1$, we can choose $||\gamma_s'(t)||$ sufficiently small to make the vector
\[ \int_{L_\delta}\gamma'_s(t)dt=\int_{L_\delta}||\gamma_s'(t)||\mathbf{t}_{\gamma_s}(t)dt \]
small enough in order that
\[ C_0 - \int_{L_\delta}\gamma'_s(t)dt \]
belongs to the convex hull of the image of $J_\varepsilon \setminus L_\delta$ by $\mathbf{t}_{\gamma}=\mathbf{t}_{\gamma_0}$. Therefore, for any $t \in J_\varepsilon \setminus L_\delta$ and any $0 \leq s \leq 1$, there exists $\rho_s(t)>0$ such that
\[ C_0-\int_{L_\delta}\gamma'_s(t)dt=\int_{J_\varepsilon \setminus L_\delta}\rho_s(t)\mathbf{t}_{\gamma_0}(t)dt \]
which, by~\eqref{tangente}, is also equal to
\begin{equation}\label{ped1}
C_0-\int_{L_\delta}\gamma'_s(t)dt=\int_{J_\varepsilon \setminus L_\delta}\rho_s(t)\mathbf{t}_{\gamma_s}(t)dt.
\end{equation}
We eventually define, for any $0 \leq s \leq 1$
\begin{equation}\label{ped2}
||\gamma_s'(t)||=\rho_s(t), \quad t \in J_\varepsilon-L_\delta.
\end{equation}
Observe that we have now completely determined $\gamma_s'(t)$ for $0 \leq s \leq 1$ and $t \in J_\varepsilon$, and~\eqref{ped1} and~\eqref{ped2} give
\[ C_0-\int_{L_\delta}\gamma'_s(t)dt=\int_{J_\varepsilon \setminus L_\delta}||\gamma_s'(t)||\mathbf{t}_{\gamma_s}(t)dt=\int_{J_\varepsilon \setminus L_\delta}\gamma_s'(t)dt \]
and hence~\eqref{Relation} is satisfied. 

To conclude, observe that we have constructed a homotopy $\gamma_s$ between $\gamma_0=\gamma$ and a curve $\gamma_1$ such that $\mathbf{t}_{\gamma_1}$ is obtained from $\mathbf{t}_{\gamma_0}$ by attaching at $t=t_0$ a pair of meridian curve in the direction $e_2=\mathbf{n}_{\gamma_0}(t_0)$ (see the Figure~\ref{fig:p}).
\begin{figure}[H]
\centering
\includegraphics[scale=0.5]{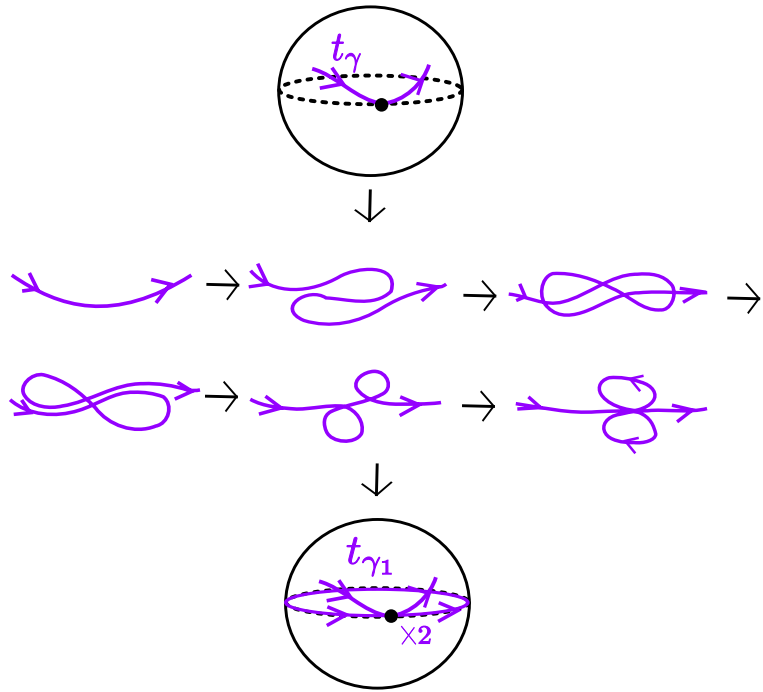}
\caption{Homotopy between $\alpha$ and $\alpha^\#$ in $\mathcal{G}\SS^3(z_l,z_r)$: first part.}
\label{fig:p}
\end{figure}
The curve $\gamma_1$ is therefore obtained from $\gamma=\gamma_0$ by attaching at $t=t_0$ a pair of loops with zero torsion, which is the same thing as saying that $\gamma_{1,l}$ is obtained from $\gamma_l$ by attaching at $t=t_0$ a pair of loops while $\gamma_{1,r}$ is obtained from $\gamma_r$ by attaching at $t=t_0$ a pair of ``reflected" loops (see the Figure~\ref{fig:q}). 
\begin{figure}[H]
\centering
\includegraphics[scale=0.5]{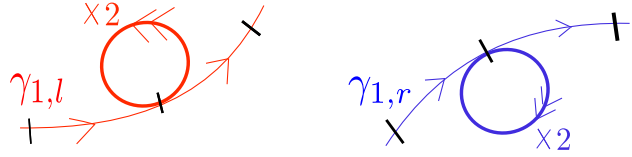}
\caption{Homotopy between $\alpha$ and $\alpha^\#$ in $\mathcal{G}\SS^3(z_l,z_r)$: second part.}
\label{fig:q}
\end{figure}
This is exactly what we wanted to prove, so this concludes the proof.   
\end{proof}

We also have the following proposition, which is the equivalent of Proposition~\eqref{genericas}.

\begin{proposition}\label{Genericas2}
Let $\alpha: K \rightarrow \mathcal{G}\SS^3(z_l,z_r)$ be a continuous map. Then $\alpha$ is homotopic to a constant map in $\mathcal{G}\SS^3(z_l,z_r)$ if and only if $\alpha^\#$ is homotopic to a constant map in $\mathcal{L}\SS^3(z_l,z_r)$. 
\end{proposition} 

One direction follows directly from Proposition~\ref{Genericas}: if $\alpha^\#$ is homotopic to a constant map in $\mathcal{L}\SS^3(z_l,z_r)$, since $\alpha$ is always homotopic to $\alpha^\#$ in $\mathcal{G}\SS^3(z_l,z_r)$, we obtain that $\alpha$ is homotopic to a constant in $\mathcal{G}\SS^3(z_l,z_r)$. The other direction can be proved exactly as in Proposition~\ref{genericas2}, using again the techniques of spreading loops along a curve; however, we will not use this statement in the following so as before, this will not be proved.    

\section{Relaxation-reflexion of curves in $\mathcal{L}\SS^2(\1)$ and $\mathcal{L}\SS^2(-\1)$}\label{s84}

The goal of this section is to address the following problem: given a continuous map $\alpha : K \rightarrow \mathcal{L}\SS^2(z_l)$, how to find a way to construct a continuous map $\hat{\alpha}: K \rightarrow \mathcal{L}\SS^3(z_l,z_r)$ such that $\hat{\alpha}_l=\alpha$. If we are able to do this, then we will be in a good position to use Proposition~\ref{nico} to obtain information on our spaces of locally convex curves. 

It will be sufficient to consider first the case of curves, that is given $\gamma \in \mathcal{L}\SS^2(z_l)$, we will try to construct $\hat{\gamma} \in \mathcal{L}\SS^3(z_l,z_r)$ such that $\hat{\gamma}_l=\gamma$. The first idea is simply to define $\hat{\gamma}_r$ to have a length equals to the length of $\hat{\gamma}_l=\gamma$, and just slightly less geodesic curvature, say the geodesic curvature of $\gamma$ reduced by a small constant $\delta$. Let us denote this curve by $R_\delta \gamma$ for the moment.

A first difficulty is that if $\hat{\gamma}=(\gamma,R_\delta \gamma)$, then the final Frenet frame of $\hat{\gamma}$ will depend on $\delta$ and also possibly on the curve $\gamma$ itself. But this is not a serious problem: exactly like for the chopping operations we described in Chapter~\ref{chapter5}, \S\ref{s53}, instead of looking at the final Frenet frame we can look at its Bruhat cell which will be independent of $\delta$ small enough and of $\gamma$, hence after a projective transformation we may assume that the curve has a fixed final Frenet frame. 

Let us denote by $R(z_l)$ a representative of the final Frenet frame of $R_\delta \gamma$; the final Frenet frame of $\hat{\gamma}$ would then be $(z_l,R(z_l))$. From this procedure one could see that $R(z_l)$ has to have a sign opposite to $z_l$, therefore such a procedure cannot give curves in two of the spaces we are interested in, namely $\mathcal{L}\SS^3(\1,\1)$ and $\mathcal{L}\SS^3(-\1,-\1)$. But this can (and in fact will) work for the other two spaces $\mathcal{L}\SS^3(-\1,\1) \simeq \mathcal{L}\SS^3(\1,-\k)$ and $\mathcal{L}\SS^3(\1,-\1) \simeq \mathcal{L}\SS^3(-\1,\k)$.  

Yet this is not sufficient. We will also want this relaxation process to be compatible with the operation $\#$ of ``adding loops" we defined in \S\ref{s53}. More precisely, one would like to know that if $\gamma$ is such that $\gamma_r=R_\delta \gamma_l$, then $\gamma^\#$ still has this property, namely we want $\gamma^\#_r=R_\delta \gamma^\#_l$. To obtain this symmetry, we will have to relax the geodesic curvature in a symmetric way by introducing another small parameter $\varepsilon>0$, and to reflect the curve obtained: this is what we will call the relaxation-reflection of a curve $\gamma$, and it will be denoted by $RR_{\varepsilon,\delta} \gamma$. We will show that for $\gamma \in \mathcal{L}\SS^2(\1)$, this will produce a curve $\hat{\gamma}=(\gamma,RR_{\varepsilon,\delta} \gamma) \in \mathcal{L}\SS^3(\1,RR(\1)) \simeq \mathcal{L}\SS^3(\1,-\k)$ and for $\gamma \in \mathcal{L}\SS^2(-\1)$, $\hat{\gamma}=(\gamma,RR_{\varepsilon,\delta} \gamma) \in \mathcal{L}\SS^3(-\1,RR(-\1)) \simeq \mathcal{L}\SS^3(-\1,\k)$.

Let us now give proper definitions.

\begin{definition}\label{relaxcurve}
Given $\gamma \in \mathcal{L}\SS^2(\pm\1)$, $\varepsilon>0$ and $\delta>0$ sufficiently small, let us define $R_{\varepsilon,\delta} \gamma$ to be the unique curve in $\mathcal{L}\SS^2$ such that
\[ ||(R_{\varepsilon,\delta} \gamma)'(t)||=||\gamma'(t)||, \quad\kappa_{R_{\varepsilon,\delta} \gamma }(t)=
\begin{cases}
\kappa_\gamma(t)-\delta, \quad t \in (0,\varepsilon) \cup (1-\varepsilon,1), \\ 
\kappa_\gamma(t)-\delta^2\varepsilon^2, \quad t \in (\varepsilon,1-\varepsilon).
\end{cases}
\]
Then let us define the curve $RR_{\varepsilon,\delta} \gamma$ to be the unique curve in $\mathcal{G}\SS^2$ such that
\[ ||(RR_{\varepsilon,\delta} \gamma)'(t)||=||(R_{\varepsilon,\delta} \gamma)'(t)||=||\gamma'(t)||, \quad \kappa_{RR_{\varepsilon,\delta} \gamma }(t)=-\kappa_{R_{\varepsilon,\delta} \gamma }(t). 
\]
\end{definition}

This definition should be understood as follows. On the small union of interval $(0,\varepsilon) \cup (1-\varepsilon,1)$, which is a symmetric interval around the initial point since our curve is closed, we relax the curvature by a constant $\delta$. On the large interval $(\varepsilon,1-\varepsilon)$, the curvature is relax by the much smaller constant $\delta^2\varepsilon^2$, so that the product of the relaxation of the curvature with the length of $(\varepsilon,1-\varepsilon)$, which is $\delta^2\varepsilon^2(1-2\varepsilon) \sim \delta^2\varepsilon^2$ is much smaller than the product of the relaxation of the curvature with the length of $(0,\varepsilon) \cup (1-\varepsilon,1)$, which is $2\delta\varepsilon \sim \delta\varepsilon$.   

If follows from Proposition~\ref{unique}, Chapter~\ref{chapter3}, \S\ref{s35}, that this curve $RR_{\varepsilon,\delta} \gamma$ is well-defined; according to our terminology, it is not locally convex but rather negative locally convex. The final Frenet frame of $RR_{\varepsilon,\delta} \gamma$, which for the moment may depend upon $\varepsilon,\delta$ and $\gamma$, will be denoted by $RR_{\varepsilon,\delta,\gamma} (\pm \1)$.

\begin{definition}\label{relcurve}
Given $\gamma \in \mathcal{L}\SS^2(\pm\1)$ and $\varepsilon,\delta>0$ sufficiently small, let us define 
\[ \hat{\gamma}_{\varepsilon,\delta}=(\gamma,RR_{\varepsilon,\delta} \gamma) \in \mathcal{L}\SS^2(\pm\1) \times \mathcal{L}\SS^2(RR_{\varepsilon,\delta,\gamma} (\pm \1)).  \]
\end{definition}

The following proposition is an obvious consequence of the definition of $RR_{\varepsilon,\delta}\gamma$ and Theorem~\ref{th1}.

\begin{proposition}
Given $\gamma \in \mathcal{L}\SS^2(\pm\1)$ and $\varepsilon,\delta>0$ sufficiently small,
\[ \hat{\gamma}_{\varepsilon,\delta} \in \mathcal{L}\SS^3(\pm\1,RR_{\varepsilon,\delta,\gamma} (\pm \1)).\] 
\end{proposition}

Then, exactly like for the chopping operation, we will use Bruhat cells to remove the dependence on $\varepsilon$, $\delta$ and $\gamma$ from the final lifted Frenet frame $(\pm\1,RR_{\varepsilon,\delta,\gamma} (\pm \1))$.

\begin{proposition}\label{relaxcurveprop}
For $\varepsilon,\delta>0$ sufficiently small and any $\gamma \in \mathcal{L}\SS^2(\pm\1)$, there exist homeomorphisms
\[ T_{\varepsilon,\delta}^+ : \mathcal{L}\SS^3(\1,RR_{\varepsilon,\delta,\gamma} (\1)) \rightarrow \mathcal{L}\SS^3(\1,-\k)\]
and 
\[ T_{\varepsilon,\delta}^- : \mathcal{L}\SS^3(-\1,RR_{\varepsilon,\delta,\gamma} (-\1)) \rightarrow \mathcal{L}\SS^3(-\1,\k).  \]
\end{proposition} 

\begin{proof}
It will be sufficient to prove that $(\1,RR_{\varepsilon,\delta,\gamma} (\1))$ is Bruhat equivalent to $(\1,-\k)$ and $(-\1,RR_{\varepsilon,\delta,\gamma} (-\1))$ is Bruhat equivalent to $(-\1,\k)$: the existence of the homeomorphisms $T_{\varepsilon,\delta}^+$ and $T_{\varepsilon,\delta}^-$ will then follow from Proposition~\ref{bruhathomeo}, Chapter~\ref{chapter3}, \S\ref{s35}.  

Let us prove that $(\1,RR_{\varepsilon,\delta,\gamma} (\1))$ is Bruhat equivalent to $(\1,-\k)$; the proof that $(-\1,RR_{\varepsilon,\delta,\gamma} (-\1))$ is Bruhat equivalent to $(-\1,\k)$ will be analogous.

We will first prove this for a specific curve $\gamma=\gamma_l \in \mathcal{L}\SS^2(\1)$; at the end we will explain how this implies the result for an arbitrary curve in $\mathcal{L}\SS^2(\1)$. Let us choose
\[ \gamma_l(t)=\Pi_3(\tilde{\Gamma}_l(t))(e_1), \quad t \in [0,1] \]
where 
\[ \tilde{\Gamma}_l(t)=\exp\left(2\pi h_l t\right) \in \SS^3, \quad t \in [0,1] \]
with $h_l=\cos(\theta_l)\i+\sin(\theta_l)\k$ and $\theta_l=\pi/4$, that is
\[ \tilde{\Gamma}_l(t)=\exp\left(2\pi \left(\frac{\i+\k}{\sqrt{2}}\right) t\right) \]
and where we recall that $\Pi_3 : \SS^3 \simeq \mathrm{Spin}_3 \rightarrow SO_3$ is the universal cover projection.

Then for $t \in [0,1]$ close to one, and given $\varepsilon,\delta>0$ small, we let
\[ \gamma_{r,\delta,\varepsilon}(t)=\Pi_3(\tilde{\Gamma}_{r,\delta,\varepsilon}(t))(e_1) \]
where $\tilde{\Gamma}_{r,\delta,\varepsilon}$ is defined by
\[ \tilde{\Gamma}_{r,\delta,\varepsilon}(t)=\exp\left((2\pi-\varepsilon) h_{r,\delta} t\right) \]
with $\varepsilon>0$ small and
\[ h_{r,\delta}=-\cos \theta_{r,\delta} \i+\sin\theta_{r,\delta} \k \]
with $\theta_{r,\delta}=\pi/4+\delta$, with $\delta$ small. Observe that this curve $\gamma_{r,\delta,\varepsilon}$, is not exactly the curve $RR_{\varepsilon,\delta}\gamma_l$ that we defined; yet clearly the two are homotopic hence it is enough to prove the result by considering $\gamma_{r,\delta,\varepsilon}$ instead of $RR_{\varepsilon,\delta}\gamma_l$. To simplify notations, we will suppress the dependence on $\varepsilon$ and $\delta$ and write $\gamma_{r}$ instead of $\gamma_{r,\delta,\varepsilon}$ (and similarly for $\tilde{\Gamma}_{r,\delta,\varepsilon}$, $h_{r,\delta}$ and $\theta_{r,\delta}$).  

Hence we can write again
\[ h_r=\cos\delta\left(\frac{-\i+\k}{\sqrt{2}}\right)+\sin\delta\left(\frac{\i+\k}{\sqrt{2}}\right) \]
and the final lifted Frenet frame of $\gamma_r$ is
\[ \tilde{\Gamma}_r(1)=\exp\left((2\pi-\varepsilon) h_r\right)=\exp\left(-\varepsilon h_r\right). \]
Let us first compute in which cell the image of $(\tilde{\Gamma}_l(1),\tilde{\Gamma}_r(1))=(\1,\tilde{\Gamma}_r(1))$ under the universal cover projection $\Pi_4 : \SS^3 \times \SS^3  \simeq \mathrm{Spin}_4 \rightarrow SO_4$ belongs to. Using the explicit expression of the map $\Pi_4$ (see Chapter~\ref{chapter2}, \S\ref{s21}), we can compute $\Pi_4(\1,\tilde{\Gamma}_r(1))$ and we find that it is equal to the matrix

\begin{equation*}
\Pi_4(\1,\tilde{\Gamma}_r(1))=
\begin{pmatrix}
P_1 & P_2 & P_3 & P_4
\end{pmatrix},
\end{equation*}
where the columns $P_i$, for $1 \leq i \leq 4$, are given by

\begin{equation*}
P_1=
\begin{pmatrix} 
 \cos(\varepsilon)\\ 
  (-\cos\delta+\sin\delta)\frac{\sin\varepsilon}{\sqrt{2}} \\
 0\\
(\cos\delta+\sin\delta)\frac{\sin\varepsilon}{\sqrt{2}}
\end{pmatrix}
\quad P_2=
\begin{pmatrix} 
(\cos\delta-\sin\delta)\frac{\sin\varepsilon}{\sqrt{2}} \\ 
 \cos(\varepsilon)\\
  (-\cos\delta-\sin\delta)\frac{\sin\varepsilon}{\sqrt{2}}\\
  0
\end{pmatrix}
\end{equation*}

\begin{equation*}
P_3=
\begin{pmatrix} 
0 \\ 
  (\cos\delta+\sin\delta)\frac{\sin\varepsilon}{\sqrt{2}} \\
 \cos(\varepsilon)\\
(\cos\delta-\sin\delta)\frac{\sin\varepsilon}{\sqrt{2}}
\end{pmatrix}
\quad P_4=
\begin{pmatrix} 
(-\cos\delta-\sin\delta)\frac{\sin\varepsilon}{\sqrt{2}} \\ 
0 \\
 (-\cos\delta+\sin\delta)\frac{\sin\varepsilon}{\sqrt{2}} \\
  \cos(\varepsilon) 
\end{pmatrix}.
\end{equation*}

Since $\varepsilon>0$ and $\delta>0$ are small, in particular $0<\varepsilon<\pi$ and $0<\delta<\pi/4$, one can check (using the algorithm described in Chapter~\ref{chapter5}, \S\ref{s53}), that this matrix is Bruhat equivalent to the transpose of the Arnold matrix ${} A ^\top \in SO_4$. Therefore $(\1,RR_{\varepsilon,\delta,\gamma_l}(\1))$ is Bruhat equivalent to $(\1,-\k)$ for the specific curve $\gamma_l$ we choose.

To conclude, observe that for the curve $\gamma_l$ we choose, the final lifted Frenet frame of $(\gamma_l,RR_{\varepsilon,\delta}\gamma_l)$ belongs to an open cell. Using this observation, and the fact that for any curve $\gamma \in \mathcal{L}\SS^2(\1)$, the curve $RR_{\varepsilon,\delta}\gamma$ is obtained from $\gamma$ by relaxing its geodesic curvature essentially in a small $\varepsilon$-neighborhood of $\gamma(0)=\gamma(1)$ (outside this neighborhood the geodesic curvature is only slightly altered), we deduce that for any curve $\gamma \in \mathcal{L}\SS^2(\1)$, the final lifted Frenet frame of $(\gamma,RR_{\varepsilon,\delta}\gamma)$ belongs   to the same open cell than the final lifted Frenet frame of $(\gamma_l,RR_{\varepsilon,\delta}\gamma_l)$. This shows that $(\1,RR_{\varepsilon,\delta,\gamma}(\1))$ is Bruhat equivalent to $(\1,-\k)$ for any curve $\gamma \in \mathcal{L}\SS^2(\1)$. 
\end{proof}

Let us now make the following definition.

\begin{definition}\label{relaxcurve2}
For $\gamma \in \mathcal{L}\SS^2(\1)$ and $\varepsilon,\delta>0$ sufficiently small, we define
\[ \hat{\gamma}=T^+_{\varepsilon,\delta}(\hat{\gamma}_{\varepsilon,\delta}) \in \mathcal{L}\SS^3(\1,-\k)\]
and for $\gamma \in \mathcal{L}\SS^2(-\1)$ and $\varepsilon,\delta>0$ sufficiently small, we define
\[ \hat{\gamma}=T^-_{\varepsilon,\delta}(\hat{\gamma}_{\varepsilon,\delta}) \in \mathcal{L}\SS^3(-\1,\k).\]
\end{definition}

Let us also make analogous definitions in the case of a continuous family of curves in $\mathcal{L}\SS^2(\pm\1)$. 

\begin{definition}\label{relaxcurve3}
For a continuous map $\alpha : K \rightarrow \mathcal{L}\SS^2(\1)$ (respectively a continuous map $\alpha : K \rightarrow \mathcal{L}\SS^2(-\1)$), we define a continuous map  $\hat{\alpha} : K \rightarrow \mathcal{L}\SS^3(\1,-\k)$ (respectively a continuous map  $\hat{\alpha} : K \rightarrow \mathcal{L}\SS^3(-\1,\k)$) by setting $\hat{\alpha}(s)=\widehat{\alpha(s)}$. 
\end{definition} 

To conclude, let us state the following proposition, which is rather obvious in view of our definitions of $\alpha^\ast$ (in the case $n=2$), $\alpha^\#$ (in the case $n=3$) and $\hat{\alpha}$.

\begin{proposition}\label{propp}
Let $\alpha : K \rightarrow \mathcal{L}\SS^3(\1,-\k)$ (respectively $\alpha : K \rightarrow \mathcal{L}\SS^3(-\1,\k)$) a continuous map. Assume that $\alpha=\hat{\beta}$ for some continuous map $\beta : K \rightarrow \mathcal{L}\SS^2(\1)$ (respectively $\beta : K \rightarrow \mathcal{L}\SS^3(-\1)$). Then $\alpha^\#$ is homotopic in $\mathcal{L}\SS^3(\1,-\k)$ (respectively in $\mathcal{L}\SS^3(-\1,\k)$) to $\widehat{\beta^\ast}$.   
\end{proposition}

\section{Proof of Theorem~\ref{th5}}\label{s85}

In this section, we finally give the proof of Theorem~\ref{th5}. First recall that for $k \geq 2$, we defined $\mathcal{M}_k(z)$ to be the set of multiconvex curves of multiplicity $k$ in $ \mathcal{L}\SS^2(z)$. We may now define $\hat{\mathcal{M}}_k(z_l,z_r)$ to be the set of curves $\gamma=(\gamma_l,\gamma_r) \in \mathcal{L}\SS^3(z_l,z_r)$ such that $\gamma_l \in \mathcal{M}_k(z_l)$. Exactly as in Proposition~\ref{manifold}, we have the following result.

\begin{proposition}\label{manifold2}
The set $\hat{\mathcal{M}}_k(z_l,z_r)$ is a closed contractible submanifold of $\mathcal{L}\SS^3(z_l,z_r)$ of codimension $2k-2$ with trivial normal bundle.
\end{proposition} 

As before, we can then associate to $\hat{\mathcal{M}}_k(z_l,z_r)$ a cohomology class $\hat{m}_{2k-2} \in H^{2k-2}(\mathcal{L}\SS^3(z_l,z_r),\R)$ by counting intersection with multiplicity.

In order to prove Theorem~\ref{th5}, we will need the following proposition.

\begin{proposition}\label{homot}
Let $\alpha_0,\alpha_1 : K \rightarrow \mathcal{L}\SS^3(\1,-\k)$ (respectively $\alpha_0,\alpha_1 : K \rightarrow \mathcal{L}\SS^3(-\1,\k)$) two continuous maps. Assume that $\alpha_0=\hat{\beta_0}$ and $\alpha_1=\hat{\beta_1}$ for some continuous map $\beta_0,\beta_1 : K \rightarrow \mathcal{L}\SS^2(\1)$ (respectively $\beta_0,\beta_1 : K \rightarrow \mathcal{L}\SS^2(-\1)$). Then $\alpha_0$ and $\alpha_1$ are homotopic in $\mathcal{L}\SS^3(\1,-\k)$ (respectively in $\mathcal{L}\SS^3(-\1,\k)$) if and only $\beta_0$ and $\beta_1$ are homotopic in $\mathcal{L}\SS^2(\1)$ (respectively in $\mathcal{L}\SS^2(-\1)$).    
\end{proposition} 

\begin{proof}
It is sufficient to consider the case where $\alpha_0,\alpha_1 : K \rightarrow \mathcal{L}\SS^3(\1,-\k)$ (the case where $\alpha_0,\alpha_1 : K \rightarrow \mathcal{L}\SS^3(-\1,\k)$ is, of course, the same). We know that $\alpha_0=\hat{\beta_0}$ and $\alpha_1=\hat{\beta_1}$ for some continuous map $\beta_0,\beta_1 : K \rightarrow \mathcal{L}\SS^2(\1)$.

On the one hand, if $H$ is a homotopy between $\alpha_0$ and $\alpha_1$, it can be decomposed as $H=(H_l,H_r)$, and it is clear that $H_l$ gives a homotopy between $\beta_0$ and $\beta_1$. On the other hand, if $H$ is a homotopy between $\beta_0$ and $\beta_1$, $\hat{H}$ provides a homotopy between $\alpha_0$ and $\alpha_1$. 
\end{proof}

Recall that from Proposition~\ref{nico}, we have maps
\[ h_{2k-2} : \SS^{2k-2} \rightarrow \mathcal{L}\SS^2((-\1)^k) \] 
that gives extra topology to $\mathcal{L}\SS^2((-\1)^k)$ with respect to the space of generic curves. Theorem~\ref{th5} will now be an easy consequence of the following proposition. 

\begin{proposition}\label{nico2}
Consider an integer $k \geq 2$. If $k$ is even, the maps
\[ \hat{h}_{2k-2} : \SS^{2k-2} \rightarrow \mathcal{L}\SS^3(\1,-\k) \]
are homotopic to constant maps in $\mathcal{G}\SS^3(\1,-\k)$ but satisfy $\hat{m}_{2k-2}(\hat{h}_{2k-2})=\pm 1$.
If $k$ is odd, the maps
\[ \hat{h}_{2k-2} : \SS^{2k-2} \rightarrow \mathcal{L}\SS^3(-\1,\k) \]
are homotopic to constant maps in $\mathcal{G}\SS^3(\1,-\k)$ but satisfy $\hat{m}_{2k-2}(\hat{h}_{2k-2})=\pm 1$.
\end{proposition} 

\begin{proof}
Let us consider the case where $k$ is even (the case where $k$ is odd is exactly the same). By definition of $\hat{h}_{2k-2}$ and $\hat{m}_{2k-2}$, it is clear that
\[ \hat{m}_{2k-2}(\hat{h}_{2k-2})=m_{2k-2}(h_{2k-2}) \]
and therefore, from Proposition~\ref{nico}, we have $\hat{m}_{2k-2}(\hat{h}_{2k-2})=\pm 1$.

It remains to prove that $\hat{h}_{2k-2}$ is homotopic to a constant map in $\mathcal{G}\SS^3(\1,-\k)$ . From  Proposition~\ref{nico}, we know that $h_{2k-2}$ is homotopic to a constant map in $\mathcal{G}\SS^2(\1)$. Using Proposition~\ref{genericas2}, this implies that $h_{2k-2}^\ast$ is homotopic to a constant map in $\mathcal{L}\SS^2(\1)$. Let us denote by $c : K \rightarrow \mathcal{L}\SS^2(\1)$ this constant map; then         $\hat{c} : K \rightarrow \mathcal{L}\SS^3(\1,-\k)$ is also a constant map. Now by Proposition~\ref{propp} $\hat{h}_{2k-2}^\#$ is homotopic in $\mathcal{L}\SS^3(\1,-\k)$ to $\widehat{h_{2k-2}^\ast}$. Since $h_{2k-2}^\ast$ is homotopic to $c$ in $\mathcal{L}\SS^2(\1)$, it follows from Proposition~\ref{homot} that $\widehat{h_{2k-2}^\ast}$ is homotopic to the constant map $\hat{c}$ in $\mathcal{L}\SS^3(\1,-\k)$, and so $\hat{h}_{2k-2}^\#$ is homotopic to the constant map $\hat{c}$ in $\mathcal{L}\SS^3(\1,-\k)$. Using Proposition~\ref{Genericas2} (we will only use the easy direction which follows from Proposition~\ref{Genericas2}), this shows that $\hat{h}_{2k-2}$ is homotopic to a constant map in $\mathcal{G}\SS^3(\1,-\k)$, which is what we wanted to prove.    
\end{proof}

\medskip

Therefore, given a integer $k \geq 2$, $\hat{h}_{2k-2}: \SS^{2k-2} \rightarrow \mathcal{L}\SS^3((-\1)^k,(-\1)^{(k-1)}\k)$ defines extra generators in $\pi_{2k-2}(\mathcal{L}\SS^3((-\1)^k,(-\1)^{(k-1)}\k))$ as compared to $\pi_{2k-2}(\mathcal{G}\SS^3((-\1)^k,(-\1)^{(k-1)}\k))$.

Using Proposition~\ref{nico2}, it will be easy to conclude.

\begin{proof}[{} of Theorem~\ref{th5}]
First let us recall that the inclusion
\[ \mathcal{L}\SS^3(z_l,z_r) \subset \mathcal{G}\SS^3(z_l,z_r)  \]
always induce surjective homomorphisms between homology groups with real coefficients. Therefore, for any $j \geq 1$, we have injective homomorphisms between cohomology groups with real coefficients 
\[ H^j(\mathcal{G}\SS^3(z_l,z_r),\R) \simeq H^j(\Omega\SS^3 \times \Omega\SS^3) \rightarrow H^j(\mathcal{L}\SS^3(z_l,z_r),\R).  \]
In our case, this implies
\begin{equation*}
\mathrm{dim}\; H^j(\mathcal{L}\SS^3(-\mathbf{1},\1),\R)=\mathrm{dim}\; H^j(\mathcal{L}\SS^3(\mathbf{1},-\k),\R) \geq 
\begin{cases}
0 & j\;\mathrm{odd} \\
l+1 & j=2l, 
\end{cases}
\end{equation*}
and
\begin{equation*}
\mathrm{dim}\; H^j(\mathcal{L}\SS^3(\mathbf{1},-\1),\R)=\mathrm{dim}\; H^j(\mathcal{L}\SS^3(-\mathbf{1},\k),\R) \geq
\begin{cases}
0 & j\;\mathrm{odd} \\
l+1 & j=2l.
\end{cases}
\end{equation*}
But now Proposition~\ref{nico2} gives, for $k \geq 2$ even, an extra element $\hat{m}_{2k-2}$ in the cohomology of degree $2k-2$ for $\mathcal{L}\SS^3(-\mathbf{1},\1) \simeq \mathcal{L}\SS^3(\mathbf{1},-\k)$. Writing $j=2l$, this gives an extra element when $j=2l$ with $l$ odd, therefore 
\begin{equation*}
\mathrm{dim}\; H^j(\mathcal{L}\SS^3(-\mathbf{1},\1),\R) \geq 
\begin{cases}
0 & j\;\mathrm{odd} \\
l+2 & j=2l, \; l\;\mathrm{odd}  \\
l+1 & j=2l, \; l\;\mathrm{even}.
\end{cases}
\end{equation*}
Similarly, Proposition~\ref{nico2} gives, for $k \geq 2$ odd, an extra element $\hat{m}_{2k-2}$ in the cohomology of degree $2k-2$ for $\mathcal{L}\SS^3(\mathbf{1},-\1) \simeq \mathcal{L}\SS^3(-\mathbf{1},\k)$. Writing $j=2l$, this gives an extra element when $j=2l$ with $l$ even, and so
\begin{equation*}
\mathrm{dim}\; H^j(\mathcal{L}\SS^3(\mathbf{1},-\1),\R) \geq
\begin{cases}
0 & j\;\mathrm{odd} \\
l+1 & j=2l, \; l\;\mathrm{odd}  \\
l+2 & j=2l, \; l\;\mathrm{even}.
\end{cases}
\end{equation*}
This ends the proof.
\end{proof}

\bibliography{bibliography}

\end{document}